\documentclass[12pt]{amsart}

 \usepackage{amsfonts,graphics,amsmath,amsthm,amsfonts,amscd,amssymb,amsmath,
 latexsym,multicol}
\usepackage{epsfig,url}
\usepackage{flafter}

\makeatletter

\def\jobis#1{FF\fi
  \def\predicate{#1}%
  \edef\predicate{\expandafter\strip@prefix\meaning\predicate}%
  \edef\job{\jobname}%
  \ifx\job\predicate
}

\makeatother

\if\jobis{proposal}%
\else
\fi

 \usepackage[matrix, arrow]{xy}
 
\DeclareMathOperator{\Bs}{Bs}
\DeclareMathOperator{\Supp}{Supp}

\DeclareMathOperator{\Fix}{Fix}

\DeclareMathOperator{\Mov}{Mov}

\DeclareMathOperator{\codim}{codim}

\DeclareMathOperator{\Proj}{Proj}
\DeclareMathOperator{\Spec}{Spec}
\DeclareMathOperator{\exc}{exc}

\DeclareMathOperator{\Hom}{Hom}
\DeclareMathOperator{\Pic}{Pic}

 \newcommand{\C}{\mathbb C}
 \newcommand{\N}{\mathbb N}
 \newcommand{\PP}{\mathbb P}
 \newcommand{\Q}{\mathbb Q}
 \newcommand{\R}{\mathbb R}
 \newcommand{\Z}{\mathbb Z}
 \newcommand{\bir}{\dashrightarrow}
 \newcommand{\rddown}[1]{\left\lfloor{#1}\right\rfloor} 

 \numberwithin{equation}{subsection}
 \numberwithin{footnote}{subsection}

 \newtheorem{cor}[subsection]{Corollary}
 \newtheorem{lem}[subsection]{Lemma}
 
 \newtheorem{thm}[subsection]{Theorem}
 \newtheorem{conj}[subsection]{Conjecture}

\newtheorem{prob}[subsection]{Problem}

{
    \newtheoremstyle{upright}%
        {8pt plus2pt minus4pt}%
        {8pt plus2pt minus4pt}%
        {\upshape}%
        {}%
        {\bfseries\scshape}%
        {}%
        {1em}%
        {}%
\theoremstyle{upright}

 \newtheorem{defn}[subsection]{Definition}
 \newtheorem{exa}[subsection]{Example}
 \newtheorem{exa-cr}[subsection]{Example--Construction}

 \newtheorem{rem}[subsection]{Remark}

 \newtheorem{exe}[subsection]{Exercise}
 \newtheorem{jimare}[subsection]{}
}

 \newcommand{\ke}[1]{$\acute{\mbox{e}}$}
 \newcommand{\ku}[1]{$\acute{\mbox{u}$}}
 \newcommand{\kl}[1]{$\acute{\mbox{l}}$}
 \newcommand{\kh}[1]{$\acute{\mbox{h}}$}
 \newcommand{\kr}[1]{$\acute{\mbox{r}}$}
 \newcommand{\kx}[1]{$\acute{\mbox{x}}$}
 \newcommand{\ki}[1]{${\^\i}$}
 
 \newcommand{\toover}{\xrightarrow}

\title{Lectures on birational geometry}\thanks{
These are the lecture notes of a course on birational geometry that I taught
at College de France, Paris, in Winter 2011. Warning: this is essentially identical to the first 
draft and I have not been through it again. So, it certainly contains mistakes (hopefully only minor.)}
\author{Caucher Birkar}\thanks{Email: c.birkar@dpmms.cam.ac.uk}
\begin{document}
\maketitle

\tableofcontents
\vspace{2cm}
\clearpage

~
\vspace{2cm}
\section{\textbf{Introduction: overview}}
\vspace{1cm}

All varieties in this lecture are assumed to be algebraic over some 
algebraically closed field $k$ (and this is going to be $\mathbb{C}$ 
for much of the course).\\

\emph{What classification means?} In every branch of mathematics, the problem of
classifying the objects arises naturally as the ultimate understanding of the
subject. If we have finitely many objects, then classification means trying to
see which object is isomorphic to another, whatever the isomorphism means. In this case,
the problem shouldn't be too difficult. However, if we have infinitely many objects, then
classification has a bit different meaning. Since each human can live for a finite
length of time, we may simply not have enough time to check every object in the theory.
On the other hand, objects in a mathematical theory are closely related, so one can
hope to describe certain properties of all the objects using only finitely many or
some of the objects which are nice in some sense.

For example let $V$ be a vector space over a field $k$. If $\dim V<\infty$, then
we can take a basis $\{v_1,\dots,v_n\}$ for this vector space and then every $v\in V$
can be uniquely written as $v=\sum a_iv_i$ where $a_i\in k$. We can say that we have
reduced the classification of elements of $V$ to the classification of the basis.
Formally speaking, this is what we want to do in any branch of mathematics but the
notion of "basis" and the basis "generating" all the objects could be very different.

In algebraic geometry, objects are varieties and relations are described using
maps and morphisms. One of the main driving forces of algebraic geometry is the following

\begin{prob}[Classification]
Classify varieties up to isomorphism.
\end{prob}

Some of the standard techniques to solve this problem, among other things, include
\begin{itemize}
\item Defining \emph{invariants} (e.g. genus, differentials, cohomology) so that
one can distinguish nonisomorphic varieties,
\item \emph{Moduli} techniques, that is, parametrising varieties by objects which are themselves
varieties,
\item \emph{Modifying} the variety (e.g. make it smooth) using certain operations
(eg birational, finite covers).
\end{itemize}

It is understood that this problem is too difficult even for curves so one needs to
try to solve a weaker problem.

\begin{prob}[Birational classification] Classify projective varieties up to birational isomorphism.
\end{prob}

By Hironaka's resolution theorem, each projective variety is \emph{birational} to a
smooth projective variety. So, we can try to classify
smooth projective varieties up to birational isomorphism. In fact, smooth birational
varieties have good common properties such as common plurigenera,
Kodaira dimension and irregularity.

In this stage, one difficulty is that in each birational class, except in dimension one,
there are too many smooth varieties. So, one needs to look for a more special and
subtle representative.
 
Before considering the simplest kind of varieties let me introduce an object that 
plays a fundamental role in the theory. For a smooth projective variety $X$ the canonical 
sheaf $\omega_X$ which is the same as the dualising sheaf can be defined as 
$\omega_X:=\wedge^d \Omega$ where $\Omega$ is the sheaf of regular differential forms. 
There is a divisor $K_X$ (unique up to linear equivalence) which gives $\omega$, that 
is, $\omega_X=\mathcal{O}_X(K_X)$.\\

\begin{jimare}{\bf{Curves}.} Projective curves are one dimensional projective varieties
(i.e. compact Riemann surfaces if $k=\C$). In each birational class, there is a unique 
smooth projective curve. So, the search for a special representative is quickly 
over and we can look at moduli spaces. Each curve $X$ in $\mathbb{P}^2$ is defined by a single homogeneous
polynomial $F$.
A natural and important "invariant" is the \emph{degree} defined as $\deg(X)=\deg F$. The degree
is actually not an invariant because a line and a conic have different degrees but they
could be isomorphic.
However, using the degree we can simply define an invariant: \emph{genus}, which is defined as
$$
g(X)=\frac{1}{2}(\deg(X)-1)(\deg(X)-2)
$$

There are other interpretations of the genus. When $k=\C$, $g(X)$ is the number of handles on $X$ 
when we consider $X$ as a Riemann surface. 
On the other hand, $g(X)=h^0(X,\omega_X)$. These new definitions make sense for every 
curve not just curves in $\PP^2$.

For each $g$, there is a moduli space $\mathcal{M}_g$ of
smooth projective curves of genus $g$. Studying these moduli spaces 
is still a hot topic in algebraic geometry. Moreover, constructing and studying 
moduli spaces is not just for the fun of it. Moduli spaces often provide information 
about families of the objects they parametrise (cf. see Kawamata's use of moduli spaces 
to prove his subadjunction formula [\ref{Kawamata-subadjunction}]).
\end{jimare}

\emph{More on genus and the canonical divisor.}  It is worth to note that 
the genus says a lot about a curve. In fact, 

\begin{displaymath}
g(X) = \left\{ \begin{array}{ll}
0 & \textrm{iff $X\simeq \mathbb{P}^1$}\\
1 & \textrm{iff $X$ is elliptic}\\
\ge 2 & \textrm{iff $X$ is of general type}
\end{array} \right.
\end{displaymath}
and these correspond to
 \begin{displaymath}
g(X) = \left\{ \begin{array}{ll}
0 & \textrm{iff $X$ has positive curvature}\\
1 & \textrm{iff $X$ has zero curvature}\\
\ge 2 & \textrm{iff $X$ has negative curvature}
\end{array} \right.
\end{displaymath}

And in terms of the canonical divisor we have:

 \begin{displaymath}
g(X) = \left\{ \begin{array}{ll}
0 & \textrm{iff $\deg K_X<0$}\\
\ge 1 & \textrm{iff $\deg K_X\ge 0$}
\end{array} \right.
\end{displaymath}

However, it turns out that in higher dimensions the genus is not 
such a useful thing. We need to consider not only the genus but a 
whole sequence of numbers. These numbers are determined by the canonical sheaf.

\begin{defn} For a smooth projective variety $X$, define the $m$-th plurigenus as
$$
P_m(X):=h^0(X,\omega_X^{\otimes m})
$$
Note that $P_1(X)=g(X)$.
Define the Kodaira dimension $\kappa(X)$ of $X$ as the largest integer $a$ 
satisfying 
$$
0<\limsup_{m\to\infty} \frac{P_m(X)}{m^a}
$$
if $P_m(X)>0$ for some $m>0$. Otherwise let $\kappa(X)=-\infty$.
\end{defn}

If $\dim X=d$, then $\kappa(X)\in\{-\infty, 0, 1, \dots, d\}$.
Moreover, the Kodaira dimension and the plurigenera $P_m(X)$ are all birational invariants.\\

\begin{exa} If $\dim X=1$, then
 \begin{displaymath}
\kappa(X) = \left\{ \begin{array}{ll}
-\infty & \textrm{iff $\deg K_X<0$}\\
0 & \textrm{iff $\deg K_X=0$}\\
> 0 & \textrm{iff $\deg K_X> 0$}
\end{array} \right.
\end{displaymath}
\vspace{0.1cm}
\end{exa}

\begin{jimare}{\bf{Classical MMP for surfaces.}}\label{rem-classical MMP} 
To get the above classification for surfaces 
one can use the classical \emph{minimal model program} (MMP) as
follows. Pick a smooth projective surface $X$ over $k$. 
If there is a $-1$-curve $E$ (i.e. $E\simeq \PP^1$ and $E^2=-1$) on $X$,
then by Castelnuovo theorem we can contract $E$ by a birational morphism $f\colon X\to X_1$
where $X_1$ is also smooth. Now replace $X$ with $X_1$ and continue the process. In each step,
the Picard number $\rho(X)$ drops by $1$. Therefore, after finitely many steps, we get a
smooth projective variety $Y$ with no $-1$-curves. Such a $Y$ turns out to have strong 
numerical properties. In fact, it is not hard to show that $Y=\PP^2$ or $Y$ is a ruled surface 
over some curve or that $K_Y$ is nef.
\end{jimare}

\emph{Enriques classification of surfaces.}
The $Y$ obtained in the process can be classified to a great extent.
 The following description of $Y$ in characteristic zero worked out by Enriques,  
Castelnuovo, Severi, Zariski, etc is now considered a classical result 
(Mumford and Bombieri obtained a similar statement in positive characteristic 
[\ref{Mumford-Bombieri}][\ref{Mumford-Bombieri2}][\ref{Mumford-Bombieri3}]). We have:\\\\
$\bullet$ If $\kappa(Y)=-\infty \implies$ $Y=\PP^2$ or $Y$ is a ruled surface 
over some curve.\\
$\bullet$ If $\kappa(Y)=0 \implies$ $Y$ is a K3 surface, an Enriques surface or an \'etale
quotient of an abelian surface.\\
$\bullet$ If $\kappa(Y)=1 \implies$ $Y$ is a minimal elliptic surface, that is, it 
is fibred over a curve with the general fibre being an elliptic curve.\\
$\bullet$ If $\kappa(Y)=2 \implies$ $Y$ is of general type.\\
$\bullet$ Moreover, the last three cases correspond to the situtaion when $K_Y$ is nef. 
Here the linear system $|mK_Y|$ is base point free for some 
$m>0$.\\

Except for the case $\kappa(Y)=2$, there are detailed classifications of $Y$ (see [\ref{Beauville}]).\\

\begin{jimare}{\bf{Higher dimension.}}
In dimension $>2$, the story is a lot more involved. The works of Fano, 
Iskovskikh, Iitaka, Ueno, Shokurov, Reid, etc, suggested that there must be 
a minimal model program for varieties in higher dimensions similar to that of surfaces.
However, this would not be without difficulties. One of the major 
obstacles was that it was not clear how to generalise $-1$-curves and 
their contractions to higher dimension. This problem was essentially 
solved by Mori who replaced $-1$-curves by the so-called extremal rays.
Another conceptual progress due to extremal rays was the fact that 
one could define analogues of ruled surfaces in higher dimension 
called Mori fibre spaces. A \emph{Mori fibre space} 
is defined as a fibre type contraction $Y\to Z$ which is a
 $K_Y$-negative extremal fibration with connected fibres. And a \emph{minimal}
variety is defined as $Y$ having $K_Y$ nef.

\end{jimare}

\begin{conj}[Minimal model] Let $X$ be a smooth projective variety. Then,
\begin{itemize}
\item If $\kappa(X)=-\infty$, then $X$ is birational to a Mori fibre space $Y\to Z$.\\
\item If $\kappa(X)\ge 0$, then $X$ is birational to a minimal variety $Y$.\\
\end{itemize}
\end{conj}

\begin{conj}[Abundance] Let $Y$ be a minimal variety. Then,
there is a fibration $\phi\colon Y\to S$ with connected fibres
 and an ample divisor $H$ on $S$ such that
$$
mK_Y=\phi^* H
$$
for some $m>0$. Moreover, \\\\
$\bullet$ $\phi(C)=\rm pt. \Longleftrightarrow  K_Y\cdot C=0$ for any curve $C$ on $Y$.\\
$\bullet$ $\dim S=\kappa(Y)\ge 0$.
\end{conj}

The two conjectures in particular say that every variety birationally admits a fibration  
the general fibre $F$ of which satisfies: $-K_F$ is ample or $K_F\sim_\Q 0$ or 
$K_F$ is ample. This resembles the three cases in the classification of curves.  
So, these extreme types of varieties are the building blocks of varieties.

As expected there are more surprises in the geometry of higher dimensional varieties. 
If we start with a smooth projective vartiety $X$, then after the contraction 
$X\to X_1$ of an extremal ray, singularities 
may appear on $X_1$. So we have to have a whole singularity theory at hand to be able to continue. 
Moreover, some times the singularities of $X_1$ are too bad, i.e. $K_{X_1}$ is not $\Q$-Cartier, that 
we have to modify the situation to get to the $\Q$-Cartier case without turning back to the starting point. 
Here, one needs a flip, a diagram $X\to X_1\leftarrow X_2$ such that $K_{X_2}$ is $\Q$-Cartier 
and ample over $X_2$. Now it makes sense to continue with $X_2$ as before. Another problem 
is that we should prove that the program terminates. Finally, even if we  
arrive at a minimal model, we need to deal with abundance. 
It is now understood that abundance 
is essentially the main problem.

Characteristic zero: in dimension $3$, the program has been established due the works of Mori, 
Shokurov, Kawamata, Koll\'ar, Reid, etc. In higher dimensions, major progress 
has been made due to the works of Shokurov, Hacon, McKernan, Birkar, Cascini, 
Ambro, Fujino, etc. It is worth to mention that Mori's methods which revolutionised 
the subject were replaced by mainly cohomological methods. This is because 
Mori's methods work well only for smooth (or very similar) varieties.

\emph{Tools.} The following theorems and their generalisations 
are the building blocks of the techniques used in minimal model program.

\begin{thm}[Adjunction]
Let $X$ be a smooth variety and $S\subset X$ a smooth subvariety.
Then,
$$
(K_X+S)\vert_S=K_S
$$
\end{thm}

\vspace{0.5cm}

\begin{thm}[Kodaira vanishing]
Let $X$ be a smooth projective variety over $k$ with ${\rm{char}}~ k=0$, and $H$ an ample divisor on $X$. Then,
$$
H^i(X,K_X+H)=0
$$
for any $i>0$.
\end{thm}

From the adjunction formula one sees that $K_S$ is closely related to $K_X+S$ 
rather than $K_X$. This is one the main reasons that in birational geometry we 
consider pairs $(X,B)$ rather than just a variety where $B$ is a $\Q$-divisor 
(even $\R$-divisor) on $X$ having coefficients in $[0,1]$. When one of the components $S$ of $B$ 
has coefficient $1$, we usually get an adjunction formula $K_X+B|_S=K_S+B_S$ 
for a certain $B_S$ on $S$. This frequently allows us to do induction on 
dimension. 

On the other hand, many of the problems and statements of birational geometry 
have a cohomological nature. In particular, We often have a restriction map 
$$
H^0(X,m(K_X+B)) \to H^0(S,m(K_S+B_S))
$$
and we like this map to be surjective. This surjectivity holds if 
$$
H^1(X,m(K_X+B)-S)=0
$$
 and this is where clones of the Kodaira vanishing 
come into play if $k=\C$. We artificially create a situation in which $m(K_X+B)-S$ 
is ample hence allowing us to use the vanishing theorem. 

\begin{defn}[Log resolution]
Let $X$ be a variety and $D$ a divisor on $X$. A projective birational morphism
$f\colon Y\to X$ is a log resolution of $X,D$ if $Y$ is smooth,
$\exc(f)$ is a divisor and $\exc(f)\cup f^{-1}(\Supp D)$ is a simple normal crossing divisor.
\end{defn}

\begin{thm}[Hironaka]
Let $X$ be a variety  and $D$ a divisor on $X$ where we assume ${\rm{char}}~ k=0$. 
Then, a log resolution of $X,D$ 
exists.
\end{thm}\begin{defn}[Log resolution]
Let $X$ be a variety and $D$ a $\Q$-divisor on $X$. A projective birational morphism
$f\colon Y\to X$ is a log resolution of $X,D$ if $Y$ is smooth,
$\exc(f)$ is a divisor and $\exc(f)\cup f^{-1}(\Supp D)$ is a simple normal crossing divisor.
\end{defn}

\emph{Linear systems.} To find 
an extremal ray and construct its contraction one can pick an ample divisor 
$H$ and take the smallest $a\ge 0$ for which $K_X+aH$ is nef. It turns out that 
$a$ is a rational number. If $K_X$ is not already 
nef, and if the linear system $|m(K_X+aH)|$ is base point free for some $m>0$, then 
there is a $K_X$-negative extremal face of the Mori cone defining a contraction. 
Such linear systems occur frequently in the subject.\\

\emph{Rationally connected varieties.} A theorem of Castelnuovo provides a simple test
of rationality of a smooth projective surface $X$ in characteristic zero: 
$P_2(X)=h^0(2K_X)=0$ and $h^1(0)=0$ 
iff $X$ is rational. In higher dimension, the rationality problem is notoriously 
difficult. Iskovskikh-Manin proved that there are unirational but not rational 3-folds.
However, unirationality is not easy to deal with either. Instead, Koll\'ar-Mori-Miyaoka, 
and Campana came up with the idea of rationally connected varieties which are much 
easier to study. A variety $X$ is rationally connected if every two points can be 
connected by a chain of rational curves. This is a place where Mori's original methods 
have proved most useful.\\

\emph{Explicit classification.} There are lots of results concerning detailed classification 
of varieities in dimension $3$, eg the works of Iskovskikh, Prokhorov, Reid, Corti, Cheltsov, 
etc. Most of these results concern Fano varieties.\\

\emph{Positive characteristic.} In higher dimension, there are only partial  
results when $k$ has positive characteristic. One reason is that the resolution of 
singularities conjecture is not proved yet. Another major reason is that the 
Kodaira vanishing theorem does not hold in this case. So, without these vital 
elements it is difficult to prove any thing unless one finds alternatives. 
Ironically though Mori's early results which opened up the field were mainly 
based on positive characteristic methods.\\

\emph{Schemes.} We have only discussed birational geometry for smooth varieties 
or alike. There is almost nothing known about birational geometry of schemes in 
general. Of course one can consider various levels of generality. For example, 
one may talk about schemes over an algebraically closed field, or any field, or 
any ring, or any base scheme.

\clearpage

~ \vspace{2cm}
\section{\textbf{Preliminaries}}
\vspace{1cm}

\begin{defn}
By a \emph{variety} we mean an irreducible quasi-projective variety over $k$ and 
from now on we take $k=\C$ unless stated otherwise.
Two varieties $X,X'$ are called \emph{birational} if there is a rational map $f\colon X\bir X'$
which has an inverse, or equivalently if $X,X'$ have isomorphic open subsets.
\end{defn}

\begin{defn}[Contraction]
A contraction is a projective morphism $f\colon X\to Y$ such that
$f_*\mathcal{O}_X=\mathcal{O}_Y$ which in particular means that $f$ has connected fibres. If 
$Y$ is normal and $f$ surjective, then $f$ being a contraction is equivalent to the connectedness of the fibres 
by the Stein factorisation. 
\end{defn}

\begin{rem}[Stein factorisation]
Let $f\colon X\to Y$ be a projective morphism. Then, it can be factored
through $g\colon X\to Z$ and $h\colon Z\to Y$ such that $g$ is a contraction
and $h$ is finite.
\end{rem}

\begin{exa}[Zariski's main theorem]
Let $f\colon X\to Y$ be a projective birational morphism where $Y$ is normal.
Then, $f$ is a contraction.
\end{exa}

\begin{exe} Give a morphism which is not a contraction.
\end{exe}

\begin{defn}[Exceptional set]
Let $f\colon X\to Y$ be a birational morphism of varieties. $\exc(f)$ is the set of those
$x\in X$ such that $f^{-1}$ is not regular at $f(x)$.
\end{defn}

\begin{defn}
Let $f\colon X\bir Y$ be a birational map of normal varieties and $V$ a prime cycle on $X$.
Let $U\subset X$ be the open subset where $f$ is regular. If $V\cap U\neq \emptyset$, then define
the \emph{birational transform} of $V$ to be the closure of $f(U\cap V)$ in $Y$. If
$V=\sum a_iV_i$ is a cycle and $U\cap V_i\neq \emptyset$, then the birational transform
of $V$ is defined to be $\sum a_iV_i^{\sim}$ where $V_i^{\sim}$ is the birational transform
of the prime component $V_i$.
\end{defn}

\begin{defn}
Let $X$ be a normal variety. 
A \emph{divisor} (resp. \emph{$\Q$-divisor, $\R$-divisor}) is
as $\sum_{i=1}^md_iD_i$ where $D_i$ are prime divisors and $d_i\in\Z$ (resp.
$d_i\in\Q$, $d_i\in\R$). A $\Q$-divisor $D$ is called \emph{$\Q$-Cartier} if
$mD$ is Cartier for some $m\in\N$; equivalently $D\in \Pic_\Q(X):=\Pic(X)\otimes_\Z \Q$.
 We say $X$ is 
\emph{$\Q$-factorial} if
every $\Q$-divisor is $\Q$-Cartier.

 An $\R$-divisor $D$ is called \emph{$\R$-Cartier} if
$D=\sum a_iE_i$ for some $a_i\in\R$ and some Cartier divisors $E_i$; equivalently 
$D\in \Pic_\R(X)=\Pic(X)\otimes_\Z \R$. 
We say that $D$ is \emph{ample} if $D=\sum a_iA_i$ for certain positive real numbers $a_i$ 
and usual ample divisors $A_i$.
\end{defn}

\begin{defn}
Let $X$ be a normal variety and $f\colon X\to Z$ a projective morphism. 
Two $\R$-divisors $D,D'$ on $X$ are called
\emph{$\R$-linearly equivalent$/Z$}, denoted by $D\sim_{\R}D'/Z$ or $D\sim_{\R,Z}D'$, if 
$$
D=D'+\sum a_i(\alpha_i)+f^*G
$$ 
for some $a_i\in \R$, rational functions $\alpha_i\in K(X)$, and $\R$-Cartier divisor $G$ 
on $Z$. 
\end{defn}

\begin{exe}
Prove that a $\Q$-divisor $D$ is $\R$-Cartier iff it is $\Q$-Cartier. 
Prove that two $\Q$-divisors $D,D'$ are 
\emph{$\R$-linearly equivalent/$Z$} iff $mD\sim mD'$
for some $m\in\N$. In this case we say that $D,D'$ are 
\emph{$\Q$-linearly equivalent$/Z$}, denoted by $D\sim_{\Q}D'/Z$.
\end{exe}

\begin{defn}
Let $X$ be a normal variety and $D$ a divisor on $X$. Let $U=X-X_{sing}$ be the smooth
subset of $X$ and $i\colon U\to X$ the inclusion. Since $X$ is normal,
$\dim X_{sing}\le \dim X-2$. So, every divisor on $X$ is uniquely determined
by its restriction to $U$. 
For a divisor $D$, we associate the sheaf $\mathcal{O}_X(D):=i_*\mathcal{O}_U(D)$.
Such sheaves are reflexive. 
We define the \emph{canonical divisor} $K_X$
of $X$ to be the closure of the canonical divisor $K_U$. Of course $K_X$ is not 
unique as a divisor but it is unique up to linear equivalence.
It is well-known that $\mathcal{O}_X(K_X)$ is the same as the dualising sheaf
$\omega_X$ if $X$ is projective [\ref{Kollar-Mori}, Proposition 5.75].
\end{defn}

\begin{exa}
The canonical divisor of $X=\PP^n$ is just $-(n+1)H$ where $H$ is a hyperplane. 
If $Y$ is a smooth hypersurface of degree $d$ in $\PP^n$, then the adjunction formula 
allows us to calculate the canonical divisors 
$$
K_Y=(K_X+Y)|_Y\sim (-(n+1)H+dH)|_Y=(d-n-1)H|_Y
$$
\end{exa}

\begin{defn}
Let $f\colon X\to Z$ be a projective morphism from a normal variety. 
Let $Z_1(X/Z)$ be the abelian group generated by curves $\Gamma\subseteq X$ 
contracted by $f$, that is,  with $f(\Gamma)={\rm{pt}}$.
An $\R$-Cartier divisor $D$ on $X$ is called
\emph{nef$/Z$} if $D\cdot C\ge 0$ for any $C\in Z_1(X/Z)$. 
Two $\R$-divisors $D,D'$ are called
\emph{numerically equivalent$/Z$}, denoted by $D\equiv D'/Z$ or $D\equiv_Z D'$, 
if $D-D'$ is $\R$-Cartier and $(D-D')\cdot C=0$ for
any $C\in Z_1(X/Z)$. Now we call $V,V'\in Z_1(X/Z)\otimes_\Z \R$ numerically
equivalent, denoted by $V\equiv V'$ if $D\cdot V=D\cdot V'$ for any $\R$-Cartier divisor
$D$ on $X$.
\end{defn}

\begin{defn}\label{d-free-ample-divisor/Z}
Let $f\colon X\to Z$ be a projective morphism of normal varieties.. 
We call a Weil divisor $D$ \emph{free over $Z$} if for each $P\in Z$, there is an affine open neighbourhood 
$P\in U\subseteq Z$
such that $D$ is free on $f^{-1}U$ in the usual sense, or equivalently if the natural
morphism $f^*f_*\mathcal{O}_X(D) \to \mathcal{O}_X(D)$ is surjective. 
 $D$ is called \emph{very ample over $Z$} if there is an embedding $i\colon X\to \PP_Z$ over $Z$
such that $\mathcal{O}_X(D)=i^*\mathcal{O}_{\PP_Z}(1)$.

Now assume that $D$ is an $\R$-Cartier divisor. We say that 
$D$ is \emph{ample over $Z$} if for each $P\in Z$, there is an affine open neighbourhood 
$P\in U\subseteq Z$ such that $D$ is ample on $f^{-1}U$. We say that $D$ is 
\emph{semi-ample over $Z$}, if there is a projective morphism $\pi\colon X\to Y/Z$ and an ample$/Z$ $\R$-Cartier 
divisor $A$ on $Y$ such that $D\sim_\R \pi^*A/Z$. 
Finally,  we say that $D$ is \emph{big$/Z$} 
if $D\sim_\R G+A/Z$ for some $\R$-divisor $G\ge 0$ and some ample$/Z$ $\R$-divisor $A$. 
Note that this is compatible with another definition of bigness that 
we will see later (see Kodaira's Lemma \ref{l-Kodaira-lemma}).
\end{defn}

\vspace{2cm}
\section{\textbf{Contractions and extremal rays}}
\vspace{1cm}

\begin{defn}
Let $X$ be a projective variety. We define\\\\
 $\bullet$ $N_1(X/Z):=Z_1(X/Z)\otimes_\Z \R/\equiv$\\
 $\bullet$ $NE(X/Z):=\mbox{the cone in $N_1(X/Z)$ generated by effective $\R$-1-cycles}$\\
 $\bullet$ $\overline{NE}(X/Z):=\mbox{closure of $NE(X/Z)$ inside $N_1(X/Z)$} $\\
 $\bullet$ $N^1(X/Z):= (\Pic_\R(X)/\equiv_Z)\simeq (\Pic(X)/\equiv_Z)\otimes_{\Z} \R$\\
 Intersection numbers give a natural pairing 
$$
N^1(X/Z)\times N_1(X/Z)\to \R
$$
which immediately gives injections 
$$
N^1(X/Z)\to N_1(X/Z)^* ~~\mbox{and} ~~N_1(X/Z)\to N^1(X/Z)^*
$$ 
On the other hand by the so-called N\'eron-Severi theorem, $N^1(X/Z)$ is a finite 
dimensional $\R$-vector space. Therefore, 
$$
\dim N_1(X/Z)=\dim N^1(X/Z)<\infty
$$ 
and  
this number is called the \emph{Picard number} of $X/Z$ and denoted by $\rho(X/Z)$. 
We refer to $N^1(X/Z)$ as the \emph{N\'eron-Severi space} of $X/Z$ and to $\overline{NE}(X/Z)$ 
as the \emph{Mori-Kleiman cone} of $X/Z$. 
\end{defn}

\begin{defn}
Let $C\subset \R^n$ be a cone with vertex at the origin.
A subcone $F\subseteq C$ is called an \emph{extremal face} of $C$ if
for any $x,y\in C$, $x+y\in F$ implies that $x,y\in F$. If $\dim F=1$,
we call it an \emph{extremal ray}.
\end{defn}

\begin{thm}[Kleiman ampleness criterion]
Let $f\colon X\to Z$ be a projective morphism of varieties and $D$ a $\Q$-Cartier divisor 
on $X$. Then, $D$ is ample$/Z$ iff $D$ is positive on $\overline{NE}(X/Z)-\{0\}$.
\end{thm}

\begin{rem}
Let $f\colon X\to Z$ be a projective morphism of varieties
 and $D$ a $\Q$-Cartier divisor. If $D$ is semi-ample$/Z$, 
 that is, if there is $m>0$ such that $mD$ is free locally 
 over $Z$, then $D$ defines a projective contraction 
 $g\colon X \to Y$ over $Z$ such that $D\sim_\Q g^*H/Z$ for 
 some ample $\Q$-divisor $H$ on $Y$. Now if $D$ is not already 
 ample, then $D$ should be numerically trivial on some curves of 
 $X$ and $g$ contracts exactly those curves: in other words, 
 $D$ is numerically trivial on some  extremal face $F$ of $\overline{NE}(X)$ 
 and $g$ contracts $F$.   
 Conversely, if we are given $g$ first, then one can take $D$ 
to be the pullback of some ample$/Z$ divisor on $Y$ so that 
$D$ defines the contraction $g$.   
 
 We are somewhat mostly interested in the converse of the above, that is, 
 we pick an extremal face $F$ and we like to know when this face can be contracted. 
 The first step is to look for a divisor that is nef$/Z$ but numerically 
 trivial on $F$ and positive on the rest of $\overline{NE}(X)$. 
 Assume that $D$ is such a a divisor. In general $F$ may not contain the class 
of any curve on $X$ in which case $F$ cannot 
be contracted. Even if there are curves whose classes belong to 
$F$, it is not in general possible to contract such curves (see 
Examples \ref{exa-nef-not-ample} and \ref{exa-uncontractible-curves}). 
Therefore, only nefness is not enough. We should look for a $D$ such that 
$mD$ is free$/Z$ for some $m>0$ giving a morphism contracting $F$.

When $K_X$ is negative on $F$, there are ways to find the desired $D$.
For example, suppose that $D=K_X+aA$ where $A$ is an ample divisor such that 
$D$ is nef$/Z$ but $D-\epsilon A$ is not nef$/Z$ for any $\epsilon>0$.  
Further assume that $D$ is numerically trivial exactly on $F$. 
Such $D$ turn out to be semi-ample (at least when $X$ has good singularities) 
and that is exactly how we will find extremal rays and contractions. 
\end{rem}

\begin{exa}
Let $X$ be a smooth projective curve. Then we have a natural exact sequence
$0\to \Pic^0(X) \to \Pic(X) \to \Z\to 0$. So, $\Pic/\Pic^0\simeq \Z$. Therefore,
$N_1(X)\simeq \R$ and $\overline{NE}(X)$ is just $\R_{\ge 0}$.
\end{exa}

\begin{exa}
Let $X$ be a normal projective variety with $\rho(X)=1$, for example $X=\PP^n$. 
Then, $N_1(X)\simeq \R$ and $\overline{NE}(X)$ is just $\R_{\ge 0}$.
\end{exa}

\begin{exa}
Let $Y$ be a smooth projective variety and $\mathcal{E}$ a locally free sheaf on $Y$. 
Let $X=\PP(\mathcal{E})$ and $f\colon X\to Y$ the structure morphism. Then, $f$ is 
the contraction of an extremal ray generated by the curves in the fibres of $f$. 
Indeed, it is well-known that $\Pic(X)\simeq \Pic(Y)\times \Z$ 
(cf. [\ref{Hart}, II, Exercise 7.9]). More precisely, 
each divisor $D$ on $X$ is as $D\sim f^*G+mH$ where $G$ is some divisor on $Y$ and 
$H$ is the divisor corresponding to $\mathcal{O}_X(1)$. This in particular also says 
that $D\cdot C=mH\cdot C$ for any curve $C$ in a fibre of $f$. So, if $C'$ is any other 
curve in a fibre of $f$, then $D\cdot C'=mH\cdot C'$ hence $D\cdot C/D\cdot C'$ does not 
depend on $m$ assuming $m\neq 0$. Thus, the classes $C$ and $C'$ belong to the same 
ray.  
\end{exa}

\begin{exa}
Let $Y$ be a smooth projective variety, $S$ a smooth subvariety, and $X$ the 
blow up of $Y$ along $S$ with the structure mophism $f\colon X \to Y$. Obviosuly, 
$f$ is the contraction of some extremal face $F$ of $\overline{NE}(X)$. 
 The exceptional divisor of $f$ is a $\PP^r$-bundle 
over $S$ hence all the curves in the fibres belong to the same ray in $\overline{NE}(X)$. 
Thus, $F$ is an extremal ray.
\end{exa}

The following theorem is our first systematic attempt in locating those curves 
which generate an extremal ray.

\begin{thm}\label{t-surfaces-rays}
Let $X$ be a normal projective surface and $C$ an irreducible curve which is $\Q$-Cartier 
as a divisor.  If $C^2<0$, then
$C$ generates an extremal ray of $\overline{NE}(X)$. But if $C^2>0$, then the class of $C$ 
cannot belong to any extremal ray unless $\rho(X)=1$. 
\end{thm}
\begin{proof}
First assume that $C^2<0$. 
Let $\mathcal{C}$ be the subcone of $\overline{NE}(X)$ consisting of those 
classes $\alpha$ for which $C\cdot \alpha\ge 0$ and let $H$ be the hyperplane 
in $N_1(X)$ where $C$ is numerically zero. In particular, if $C'$ 
is any curve other than $C$, then the class of $C'$ is in $\mathcal{C}$. 
Moreover, $\overline{NE}(X)$ is nothing but the convex hull of $[C]$ and 
$\mathcal{C}$. Therefore, $[C]$ generates an extremal ray of $\overline{NE}(X)$ 
as $[C]$ is on one side of $H$ and $\mathcal{C}$ on the opposite side.

Now the second statement: assume that $IC$ is Cartier and 
let $f\colon Y\to X$ be a resolution of singularities. 
Then, $(f^*C^2)>0$ and the Riemann-Roch theorem shows that $h^0(mIf^*C)$ grows 
like $m^2$ hence the same holds for $h^0(mIC)$. Pick a general very ample divisor 
$A$ and consider the exact sequence 
$$
0\to H^0(X,mIC-A)\to H^0(X,mIC)\to H^0(A,mIC|_A)
$$
Since $A$ is a smooth curve, $h^0(A,mIC|_A)$ grows at most like $m$ 
which shows that $h^0(X,mIC-A)$ grows like $m^2$ hence $mIC\sim A+C'$ 
for some $m>0$ and some effective divisor $C'$. In particular, the classes 
of both $A$ and $C'$ are in $\overline{NE}(X)$. This implies that the class 
of $C$  cannot generate 
any extremal ray unless $\overline{NE}(X)$ is just a half-line and $\rho(X)=1$.
\end{proof}

\begin{exa}\label{exa-rays-surfaces}
Here we list a few simple examples of extremal rays and contractions on surfaces.\\
$\bullet$ Let $X=\PP^1\times \PP^1$ be the quadric surface. $N_1(X)\simeq \R^2$ and $\overline{NE}(X)$
has two extremal rays, each one is generated by fibres of one of the two natural projections.

$\bullet$ If $X$ is the blow up of
$\PP^2$ at a point $P$, then $N_1(X)\simeq \R^2$ and $\overline{NE}(X)$ has two extremal
rays. One is generated by the exceptional curve of the blow up and the other one
is generated by the birational transform of all the lines passing through $P$.

$\bullet$ If $X$ is a cubic surface, it is well-known that it contains exactly 27 lines. 
By definition, $X\subset \PP^3$ and $K_X=-H|_X$ where $H$ is a hyperplane. 
If $L$ is one of the lines, then $L^2=-1$ hence by Theorem \ref{t-surfaces-rays}, 
$L$ generates an extremal ray of $\overline{NE}(X)$. On the other hand, more explicit 
calculations show that    
$N_1(X)\simeq \R^7$ and that $\overline{NE}(X)$ has no more extremal rays.

$\bullet$ If $X$ is an abelian surface, one can prove that $\overline{NE}(X)$ has a round shape,
that is it does not look like polyhedral unless $\rho(X)\le 2$. This happens because an abelian variety
is in some sense homogeneous.

$\bullet$ Finally, there are surfaces $X$ which have infinitely many $-1$-curves. So,
they have infinitely many extreml rays. An example of such a surface is the blow up of the
projective plane at nine points which are the base points of a general pencil of cubics.
\end{exa}

\begin{exa}\label{exa-nef-not-ample}
There is a curve $C$ of genus at least $2$ and a locally free sheaf $\mathcal{E}$ 
of rank $2$ on $C$ such that the divisor $D$ corresponding to the invertible sheaf 
$\mathcal{O}_{\PP(\mathcal{E})}(1)$ is strictly 
nef (i.e. positively intersects every curve on $X=\PP(\mathcal{E})$) but not ample 
[\ref{Hartshorne}, I-10-5]. 
By Kleiman criterion, $D$ is numerically trivial on some extremal face of $\overline{NE}(X)$. 
Since $\rho(X)=2$ and since $D$ is not numerically trivial, $\overline{NE}(X)$ has exactly 
two extremal rays and $D$ is trivial on one of them, say $R_1$. Therefore, 
since $D$ is strictly nef, $R_1$ does not contain the class of any curve on $X$. 
The other extremal ray $R_2$ is generated by the curves in the fibres of $X\to C$.   
\end{exa} 

\begin{exa}\label{exa-uncontractible-curves}
Let $C$ be an elliptic curve in $\PP^2$, and $p_1,\dots,p_{12}$ be 
distinct points on $C$, and $X\to \PP^2$ the blow up of $\PP^2$ at 
the above points.  If $E\subset X$ is the birational transform of $C$, then $E^2=-3$ 
hence $E$ generates an extremal ray $R$ of $\overline{NE}(X)$ by Theorem \ref{t-surfaces-rays}. 
However, $R$ cannot be contracted. In fact, let $f\colon X\to Y$ be a 
projective morphism contracting $R$ and let $L$ be the pullback of 
an ample divisor on $Y$. Note that there is no other curve on $X$ whose class 
belong to $R$ hence $f$ contracts only $E$.  Now  $L|_E\sim 0$ and we can write 
$L\sim hf^*H+\sum e_iE_i$ for certain $h,-e_i\in \Z^{>0}$ where $H$ is a 
line on $\PP^2$. Therefore, we have $0\sim L|_E\sim hH|_C+\sum e_ip_i$ 
where we identify $E$ and $C$. But if we choose the $p_i$ general enough 
then the latter linear equivalence is not possible (see below) which gives a contradiction. 
In particular, the extremal ray generated by $E$ is not contractible.

We choose the $p_i$ as follows. First we can assume that $H$ intersects $C$ 
at a single point $c$ so that $H|_C=3c$ (equivalently, we can first choose a point $c$ 
and then consider the linear system $|3c|$ which is free and gives an embedding 
of $C$ into $\PP^2$ making $3c$ the restriction of some line). If $p$ is any point on $C$ other than 
$c$, then $c-p$ is numerically trivial but not linearly trivial. 
Suppose that we have chosen $p_1,\dots,p_r$ so that if $\alpha_{i}=c-p_i$ 
then the $\alpha_{i}$ are $\Z$-linearly independent in $\Pic^0(C)$. 
Now choose $p_{r+1}$ so that no positive interger multiple of $\alpha_{r+1}$ belongs to the 
subgroup generated by $\alpha_1,\dots,\alpha_r$ (such $p_{r+1}$ 
exists otherwise $\Pic^0(C)$ would be countable which is not the case). This ensures that 
$\alpha_1,\dots,\alpha_{r+1}$ are $\Z$-linearly independent. Inductively, 
we find $p_1,\dots,p_{12}$. Now if $hH|_C+\sum e_ip_i\sim 0$, then we can write 
$hH|_C+\sum e_ip_i=\sum a_i\alpha_i$ for certain $a_i\in \Z$. However, 
this contradicts the independence of the $\alpha_i$. 
\end{exa}

\begin{exa}
Let $S=\PP^2$ and $X=\PP(\mathcal{E})$ where 
$$
\mathcal{E}=\mathcal{O}_S\oplus \mathcal{O}_X(1)\oplus \mathcal{O}_S(1)
$$ 
and $\pi\colon X\to S$ the structure map. 
Then, $\rho(X)=2$ and $\overline{NE}(X)$ has two extremal rays. Let $A$ be the divisor 
corresponding to $\mathcal{O}_X(1)$. Since $\pi_*\mathcal{O}_X(A)=\mathcal{E}$ and since 
$\mathcal{E}$ is generated by global sections, $\mathcal{O}_X(A)$ is generated 
by global sections as $\pi^*\mathcal{E}\to \mathcal{O}_X(A)$ is surjective.    
So, $A$ defines a contraction $f\colon X\to Y$ which is birational as $A$ is big. 
The summand $\mathcal{O}_S$ 
gives an embedding $S\subset X$ such that $A|_S\sim 0$ which means that $S$ is 
contracted by $f$. This in particular implies that if $L$ is a line on $S$, then its class 
belongs to the extremal ray, say $R_1$, which is $A$-trivial (the other extremal, $R_2$, 
is generated by the curves in the 
fibres of $\pi$). 

Now let $C$ be a curve having a non-constant morphism $g\colon C\to S$. Then, 
$$
g^*\mathcal{E}=\mathcal{O}_C \oplus \mathcal{O}_C(1) \oplus \mathcal{O}_C(1)
$$
Any morphism $h\colon C\to X$ over $S$ such that $h^*A\sim 0$ corresponds to a 
quotient $g^*\mathcal{E}\to \mathcal{O}_C$. Here two quotients are considered 
the same if they have the same kernels. On the other hand, we see that 
$$
\Hom(g^*\mathcal{E},\mathcal{O}_C)\simeq \Hom(\mathcal{O}_C,\mathcal{O}_C)\simeq \C
$$
hence there is only one quotient (as the quotients are just given by quotients 
$\mathcal{O}_C\to \mathcal{O}_C$ which are nothing but multiplication by elements of $\C$) 
so we have only one possible morphism $h$ 
satisfying the required assumptions. This implies that if $C$ is any curve 
whose class is in $R_1$, then $C$ should already be inside $S$. Therefore, 
$f$ contracts exactly $S$.

Furthermore, $K_X$ is negative on $R_1$ so this gives an example of a flipping contraction 
in dimension $4$. Finally, note that if we replace $\PP^2$ by $\PP^1$ and take a 
similar $\mathcal{E}$ and $X=\PP(\mathcal E)$, then $K_X$ would be numerically 
trivial on the contraction defined by $A$ giving a flop rather than a flip.
\end{exa}

\begin{exa}
Let $X$ be a smooth projective variety such that $-K_X$ is ample. Such $X$ are called 
\emph{Fano} varieties. In this case, $K_X$ is negative on the whole $\overline{NE}(X)$. 
It turns out that $\overline{NE}(X)$ is as nice as possible, that is, 
it is a polyhedral cone generated by finitely many extremal rays. The cubic surface 
mentioned above is an example of a Fano variety. More generally, if $X$ is a hypersurface 
in $\PP^d$ of degree $\le d$, then the adjunction formula says that $X$ is a Fano variety.
\end{exa}

\begin{exa}
Let $Z\subset \C^4$ be the 3-fold defined by the quadratic equation $xy-zu=0$. 
Obviously, $Z$ is singular at the origin. Blow up $\C^4$ at the origin and let $X$ be 
the birational transform of $Z$ with $f\colon X\to Z$ the induced morphism. It turns out 
that $X$ is smooth and that $f$ has a single exceptional divisor $E$ isomorphic to 
$\PP^1\times \PP^1$. The two projections of $E$ onto $\PP^1$ determine two 
extremal rays of $\overline{NE}(X/Z)$ giving contractions onto smooth 3-folds. 
\end{exa}

\begin{exa}
Let $X$ be a projective toric variety corresponding to a fan $\Delta$ in $\R^d$. 
Then, $\overline{NE}(X)$ is a rational polyhedral cone and each of its extremal rays  
is generated by the curve corresponding to some $(d-1)$-dimensional cone $\sigma$ 
in $\Delta$. Furthermore, assume that $X$ is $\Q$-factorial. This just corresponds to 
the cones in $\Delta$ being simplicial. Then, each extremal ray $R$ of $\overline{NE}(X)$ 
can be contracted in the following way. Obtain a new fan out of $\Delta$ as follows: 
if $\sigma$ is a $(d-1)$-dimensional cone in $\Delta$ such that the class of the corresponding 
curve belongs to $R$, then remove $\sigma$ from the cone. It turns out by removing all 
such $\sigma$, we get a new fan which we denote by $\Delta'$. This induces a contraction  
$X\to Y$ which is the contraction of the extremal ray $R$. For more details see 
[\ref{Matsuki}, Chapter 14].
\end{exa}

\clearpage

~ \vspace{2cm}
\section{\textbf{Pairs and their singularities}}
\vspace{1cm}

\begin{defn}
Let $X$ be a variety and $D$ a divisor on $X$. We call $X,D$ \emph{log smooth} if
$X$ is smooth and the components of $D$ have simple normal crossings.
A \emph{pair} $(X/Z,B)$ consists of a projective morphism $X\to Z$ between normal 
varieties, and an $\R$-divisor $B$ on $X$ with
coefficients in $[0,1]$ such that $K_X+B$ is $\R$-Cartier. When 
 we are only interested in $X$ and $B$ and not in the morphism $X\to Z$, we usually 
 drop $Z$ (it is also customary to drop $Z$ when it is a point but this will be clear 
 from the context, eg by saying that $X$ is projective).
  
Now let $f\colon Y\to X$ be a log resolution of a pair $(X,B)$. Then, we can write

$$
K_Y=f^*(K_X+B)+A
$$
where we assume that $f_*K_Y=K_X$ as divisors for some choice of canonical divisors 
$K_Y$ and $K_X$ (i.e. choice as a specific Weil divisor not as a linear equivalence class). 

 For a prime divisor $E$ on $Y$, we define the \emph{discrepancy} of $E$ with respect
 to $(X,B)$ denoted by $d(E,X,B)$ to be the coefficient of $E$ in $A$. Note that
if $E$ appears as a divisor on any other resolution, then $d(E,X,B)$ is the same.
\end{defn}

\begin{rem}[Why pairs?]
The main reason for considering pairs is the various kinds of adjunction, that is,
relating the canonical divisor of two varieties which are closely related. We have already
seen the adjunction formula $(K_X+S)\vert_S=K_S$ where $X,S$ are smooth and $S$ is a prime
divisor on $X$. It is natural to consider $(X,S)$ rather than just $X$.

Now let $f\colon X\to Z$ be a finite morphism. It often happens that $K_X=f^*(K_Z+B)$
for some boundary $B$. For example, if $f\colon X=\PP^1\to Z=\PP^1$ is a finite map of 
degree $2$ ramified at two points $P$ and $P'$, then $K_X\sim f^*(K_Z+\frac{1}{2}P+\frac{1}{2}P')$.
Similarly, when $f$ is a contraction and $K_X\sim_{\R} 0/Z$, 
then under good conditions $K_X\sim_{\R}f^*(K_Z+B)$
for some boundary $B$ on $Z$. Kodaira's canonical bundle formula for
an elliptic fibration of a surface is a clear example.
\end{rem}

\begin{defn}[Singularities]
Let $(X,B)$ be a pair. We call it \emph{terminal} (resp. \emph{canonical}) if $B=0$ and
there is a log resolution $f\colon Y\to X$ for which $d(E,X,B)>0$ (resp. $\ge 0$) for
any exceptional prime divisor $E$ of $f$. We call the pair \emph{Kawamata log terminal}
 ( resp. \emph{log canonical}) if there is a log resolution $f$ for which $d(E,X,B)>-1$ (resp. $\ge -1$)
 for any prime divisor $E$ on $Y$ which is exceptional for $f$ or the birational transform of
 a component of $B$. The pair is called \emph{divisorially log terminal} if there is a log resolution
 $f$ for which $d(E,X,B)>-1$ for any exceptional prime divisor $E$ of $f$. We
 usually use abbreviations klt, dlt and lc for Kawamata log terminal, divisorially log terminal
 and log canonical respectively.
\end{defn}

\begin{rem}\label{r-discrepancy-blowup}
Let $X$ be a smooth variety and $D$ an $\R$-divisor on $X$. Let $V$ be a
smooth subvariety of $X$ of codimension $\ge 2$ and $f\colon Y\to X$ the
blow up of $X$ at $V$, and $E$ the exceptional divisor. Then, the coefficient
of $E$ in $A$ is $\codim V -1-\mu_V D$ where

$$
K_Y=f^*(K_X+D)+A
$$
and $\mu$ stands for multiplicity.
\end{rem}

\begin{lem}
Definition of all kind of singularities (except dlt) is independent of the choice of the log resolution.
\end{lem}
\begin{proof}
Suppose that $(X,B)$ is lc with respect to a log resolution $f\colon Y\to X$. Let $f'\colon Y'\to X$ be
another log resolution. There is a smooth $W$ with morphisms $g\colon W\to Y$ and 
$g'\colon W\to Y'$ such that $g$ is a sequence of smooth blow ups. It is enough to show that 
$(X,B)$ is also lc with respect to $fg$, and using induction we may assume that $g$ is the  
blow up of some smooth $V\subset Y$ of codimension $c$. Write $K_Y=f^*(K_X+B)+A$. Then,  
 we have 
$$
K_W=g^*K_Y+(c-1)E=g^*f^*(K_X+B)+g^*A+(c-1)E
$$ 
where $E$ is the exceptional divisor of $g$. Now the fact that $A$ has simple normal crossing 
singularities and the fact that its coefficients are all at least $-1$ imply that 
the coefficient of $E$ in $g^*A$ is at least $-c$. Thus, the coefficients of $g^*A+(c-1)E$ 
are at least $-1$ and this shows $(X,B)$ is lc with respect to $fg$. 
Similar arguments appy to klt, etc cases.
\end{proof}

\begin{exe}
 Let $(X,B)$ be a pair and $f\colon Y\to X$ a log resolution. Let $B_Y$ be the divisor on $Y$
 for which $K_Y+B_Y=f^*(K_X+B)$. Prove that $(X,B)$ is\\
$\bullet$ terminal iff $B_Y\le 0$ and $\Supp B_Y=\exc(f)$,\\
$\bullet$ canonical iff $B_Y\le 0$,\\
$\bullet$ klt iff each coefficient of $B_Y$ is $<1$,\\
$\bullet$ lc iff each coefficient of $B_Y$ is $\le 1$.
\end{exe}

\begin{exa}
Let $(X,B)$ be a pair of dimension $1$. Then, it is lc (or dlt) iff each coefficient of $B$ is
$\le 1$. It is klt iff each of coefficient of $B$ is $<1$. It is canonical (or terminal) iff
$B=0$.
\end{exa}

\begin{exa}[Log smooth pairs]
When $(X,B)$ is log smooth, i.e. $X$ is smooth and $\Supp B$ has simple normal crossing singularities, 
then we have the most simple yet crucial kind of singularity. It is
easy to see what type of singularities this pair has:  $(X,B)$ is lc iff it is dlt iff each coefficients of 
$B$ is $\le 1$;  $(X,B)$ is klt iff each coefficients of $B$ is $<1$; 
\end{exa}

\begin{exe} Prove that a smooth variety is terminal.
\end{exe}

\begin{exe} Prove that terminal $\implies$ canonical $\implies$ klt $\implies$ dlt $\implies$ lc.
\end{exe}

\begin{exe} If $(X,B+B')$ is terminal (resp. canonical, klt, dlt or lc) then so is $(X,B)$ where
$B,B'\ge 0$ are $\R$-Cartier.
\end{exe}

\begin{exe} Prove that if $(X,B)$ is not lc and $\dim X>1$, then for any integer $l$ there is $E$
such that $d(E,X,B)<l$.
\end{exe}

\begin{exa}
Let $(\PP^2,B)$ be a pair where $B$ is a curve with one nodal singularity. 
This pair is lc but not dlt. On the other hand, the pair $(\PP^2,B')$ where 
$B'$ is a curve with one cuspidal singularity is not lc.
\end{exa}

\begin{exa}
  Let's see what terminal, etc. mean for some of the simplest
surface singularities. Let $Y$ be a smooth surface containing
a curve $E=\PP^1$ with $E^2 = -a$, $a>0$. (Equivalently, the normal
bundle to $E$ in $Y$ has degree $-a$.) It's known that one can
contract $E$ to get a possibly singular surface $X$ via some contraction 
$f\colon Y\to X$. Explicitly,
$X$ is locally analytically the cone over the rational normal curve
    $\PP^1 < \PP^a$;
so for $a=1$, $X$ is smooth, and for $a=2$, $X$ is the surface node
${x^2+y^2-z^2=0}$ in $\mathbb{A}^3$ (a canonical singularity).

Then $K_E = (K_Y+E)\vert_E$, and $K_E$ has degree $-2$ on $E$
since $E$ is isomorphic to $\PP^1$, so
   $K_Y \cdot E = -2+a$.
This determines the discrepancy c in
   $K_Y = f^*(K_X) + cE$,
because $f^*(K_X)\cdot E = K_X\cdot (f_*(E)) = K_X \cdot 0=0$. Namely,
  $c = (a-2)/(-a)$.
So for $a=1, c=1$ and $X$ is terminal (of course, since it's smooth);
for $a=2, c=0$ and $X$ is canonical (here $X$ is the node); and for $a\ge 3$,
$c$ is in $(-1,0)$ and $X$ is klt.

Just for comparison:
if you contract a curve of genus 1, you get an lc singularity
which is not klt; and if you contract a curve of genus at least 2,
it is not even lc.
\end{exa}

\begin{lem}[Negativity lemma]\label{l-negativity}
Let $f\colon Y\to X$ be a projective birational morphism of normal varieties.
Let $D$ be an $\R$-Cartier divisor on $Y$ such that $-D$ is nef over $X$. Then,
$D$ is effective iff $f_*D$ is.
\end{lem}
\begin{proof}
First by localising the problem on $X$ and taking hyperplane sections on $Y$ we can assume
that $X,Y$ are surfaces and that $f$ is not an isomorphism.
Now after replacing $Y$ with a
resolution we can find
an effective exceptional divisor $E$ which is antinef/$X$ whose support 
contains $\exc(f)$. To find such $E$ 
we could take a non-zero effective Cartier divisor $H$ on $Y$ whose support 
contains the image of $\exc(f)$. Then, $f^*H=H^\sim+E$ where $H^\sim$ 
is the birational transform of $H$ and $E$ an exceptional effective divisor. 
Obviously, $H^\sim$ is nef$/Y$ hence 
$E$ is antinef$/Y$.

Let $e$ be the minimal non-negative number for which $D+eE\ge 0$. If $D$ is not
effective, then $D+eE$ has coefficient zero at some exceptional curve $C$. 
On the other hand, locally over $X$, 
$E$ is connected. So, if $D+eE\neq 0$, we can choose $C$ so that it intersects 
some component of $D+eE$. But in that case $(D+eE)\cdot C>0$ which contradicts the 
assumptions. Therefore, $D+eE=0$ which is not possible otherwise $D$ and $E$ 
would be both numerically trivial over $X$.
\end{proof}

\begin{defn}[Minimal resolution]
Let $X$ be a normal surface and $f\colon Y\to X$ a resolution. We call $f$ or $Y$ 
the \emph{minimal resolution} of $X$ if any other resolution $f'\colon Y'\to X$ 
factors through $f$.
\end{defn}

\begin{rem} It is well-known that the matrix $[E_i\cdot E_j]$ of a 
minimal resolution $Y\to X$ is negative definite, in particular, has non-zero determinant 
where $E_i$ denote the exceptional curves. 
So, if $D$ is any $\R$-divisor on $Y$, we can find an 
$\R$-divisor $G$ such that $D+G\equiv 0/X$ and such that $G$ is exceptional$/X$. 
We can find $G=\sum e_iE_i$   
by requiring  $(D+G)\cdot E_i=0$ for each $i$. These equations give us a unique 
solution for the $e_i$.
\end{rem}

\begin{lem}
Let $f\colon Y\to X$ be the minimal resolution of a normal surface $X$ and let 
$G$ be the exceptional divisor satisfying $K_Y+G\equiv 0/X$. Then, $G\ge 0$.
\end{lem}
\begin{proof}
First prove that $K_Y$ is nef over $X$ using the formula $(K_Y+C)\cdot C=2p_a(C)-2$
for a proper curve $C$ on $Y$. Indeed, $K_Y\cdot C\ge -2-C^2$ hence if $C^2\le-2$, 
clearly $K_Y\cdot C\ge 0$. But if $C^2=-1$, then again $K_Y\cdot C\ge 0$ otherwise 
$K_Y\cdot C=-1$ which implies that $C$ is a $-1$-curve contradicting the fact 
that $Y\to X$ is a minimal resolution.  
Finally, the negativity lemma implies that $G\ge 0$.
\end{proof}

\begin{thm}
A surface $X$ is terminal iff it is smooth.
\end{thm}
\begin{proof}
If $X$ is smooth then it is terminal. Now suppose that $X$ is terminal
and let $Y\to X$ be a minimal resolution and let $K_Y+B_Y$ be the pullback of
$K_X$. Since $X$ is terminal, $d(E,X,0)>0$ for any exceptional divisor. Thus,
$Y\to X$ is an isomorphism otherwise $B_Y<0$ giving a contradiction.
\end{proof}

\begin{cor}\label{cor-terminal} By taking hyperplane sections, one can show that terminal
varieties are smooth in codimension two.
\end{cor}

\begin{rem}
Surface canonical singularities are classically known as Du Val singularities. 
Locally analytically a Du Val singularity is given by one of the 
following equations as a hypersurface in $\mathbb{A}^3$:\\
 A: $x^2+y^2+z^{n+1}=0$\\
 D: $x^2+y^2z+z^{n-1}=0$\\
 E$_6$: $x^2+y^3+z^4=0$\\
 E$_7$: $x^2+y^3+yz^3=0$\\
 E$_8$: $x^2+y^3+z^5=0$\\
\end{rem}

\begin{lem}
If $X$ is klt, then all the exceptional curves of the minimal resolution are smooth rational
curves.
\end{lem}
\begin{proof}
Let $E$ be an exceptional curve appearing on the minimal resolution $Y$. Then, $(K_Y+eE)\cdot E\le 0$
for some $e<1$. So, 
$$
(K_Y+E+(e-1)E)\cdot E=2p_a(E)-2+(e-1)E^2\le 0
$$
 which in turn implies that
$p_a(E)\le 0$. Therefore, $E$ is a smooth rational curve.
\end{proof}

\begin{exa} \emph{Singularities in higher dimension}.
Let $X$ be defined by $x^2+y^2+z^2+u^2=0$ in $\mathbb{A}^4$. Then, by blowing up the
origin of $\mathbb{A}^4$ we get a resolution $Y\to X$ such that we have a
single exceptional divisor $E$ isomorphic to the quadric surface $\PP^1\times \PP^1$.
Suppose that $K_Y=f^*K_X+eE$. Take a fibre $C$ of the projection $E\to \PP^1$.
Either by calculation or more advanced methods, one can show that $K_Y\cdot C<0$ and
$E\cdot C<0$. Therefore, $e>0$. So, $X$ has a terminal singularity at the origin.
\end{exa}

\begin{rem}[Toric varieties]  Suppose that $X$ is the toric variety
associated to a cone $\sigma\subset N_{\R}$,
\begin{itemize}
\item $X$ is smooth iff $\sigma$ is regular, that is primitive generators of each face of $\sigma$ consists
of a part of a basis of $N$,
\item $X$ is $\Q$-factorial iff $\sigma$ is simplicial,
\item $X$ is terminal iff $\sigma$ is terminal, that is, there is $m\in M_{\Q}$ such that
$m(P)=1$ for each primitive generator $P\in \sigma \cap N$, and $m(P)>1$ for any other
$P\in N\cap \sigma -\{0\}$,
\item If $K_X$ is $\Q$-Cartier, then $X$ is klt.
\end{itemize}

See [\ref{Dais}] and [\ref{Matsuki}] for more information.
\end{rem}

\clearpage
~ \vspace{2cm}
\section{\textbf{Kodaira dimension}}
\vspace{1cm}

\begin{rem}[Divisorial sheaves]
Let $X$ be a normal variety. We can interpret sections of divisors on $X$ 
as rational functions. If $D$ is a Weil divisor on $X$, then we can 
describe $\mathcal{O}_X(D)$ in a canonical way as 
$$
\mathcal{O}_X(D)(U)=\{f\in K(X) \mid (f)+D|_U\ge 0\}
$$
where $U$ is an open subset of $X$ and $K(X)$ is the function field of $X$.
If $D'\sim D$, then of course $\mathcal{O}_X(D')\simeq \mathcal{O}_X(D)$ but these sheaves 
are not canonically identical, i.e. they have different embedding in the constant sheaf associated 
to $K(X)$. 

If $W$ is any closed subset of $X$ such that $\codim X\setminus W\ge 2$, then 
$\mathcal{O}_X(D)=j_*\mathcal{O}_W(D|_W)$, where $j\colon W\to X$ is the inclusion, 
because 
$$
(f)+D\ge 0 ~~~~~\mbox{iff}~~~~~ (f)+D|_W\ge 0
$$
In particular, by taking $W=X\setminus X_{\rm sing}$ one can easily see that the above 
sheaves are reflexive.
\end{rem}

\begin{defn}\label{def-D-morphism}
Let $D$ be a Cartier divisor on a normal projective variety $X$. If $h^0(X,D)\neq 0$, then we define
a rational map $\phi_{D}\colon X \bir \PP^{n-1}$ as
$$
\phi_{D}(x)=(f_1(x):\dots:f_n(x))
$$
where $\{f_1,\dots,f_n\}$ is a basis for $H^0(X,D)$. 
\end{defn}

\begin{defn} A Cartier divisor $D$ on a normal variety $X$ is called \emph{free} if its base locus
$$
\Bs |D|:=\bigcap_{D\sim D'\ge 0} \Supp D'
$$
is empty. This is equivalent to saying that $\mathcal{O}_X(D)$ is generated by global sections.
For a free divisor $D$, the rational map $\phi_D\colon X\bir \PP^{n-1}$ associated
to $D$  is actually a morphism. The Stein factorisation of $\phi_D$
gives us a contraction $\psi_D\colon X\to Y$ such that $D\sim \psi_D^*H$ for some ample divisor $H$
on $Y$.
\end{defn}

\begin{rem}\label{rem-decomposition}
 Let $D$ be a Weil divisor on a normal variety $X$ with $h^0(X,D)\neq 0$. 
Let $F\ge 0$ be the biggest Weil divisor satisfying $F\le G$ for any $0\le G\sim D$. 
We call $F$ the \emph{fixed} part of $D$ and write $F=\Fix D$. 
We can write $D=M+F$ and call $M$ 
the \emph{movable} part of $D$ denoted by $\Mov D$. 
Note that $\Fix M=0$ and $H^0(X,M)=H^0(X,D)$. If $D$ is Cartier, then Hironaka's work 
yield a resolution
 $f\colon Y\to X$ such that $f^* D=M'+F'$ such that $M'=\Mov D$ is a free divisor.
 In particular, the sections $H^0(Y,M')$ define a contraction $\pi\colon  Y\to Z$ 
 such that $M\sim \pi^* A$ for some ample divisor $A$ on $Z$. To obtain $\pi$ 
 one first considers the rational map $\phi_{f^*D}\colon Y \bir \PP^{n-1}$ as
$$
\phi_{f^*D}(x)=(f_1(x):\dots:f_n(x))
$$
where $\{f_1,\dots,f_n\}$ is a basis for $H^0(Y,f^*D)$. Since $M'$ is free, the 
above map is a morphism. If $V$ is the image of $\phi_{f^*D}$, take $\pi$ to be 
the contraction given by the Stein factorisation of $Y\to V$. 
\end{rem}

\begin{defn}[Kodaira dimension] 
For a $\Q$-divisor $D$ on a normal projective variety $X$,
define the \emph{Kodaira dimension} of $D$ as the largest integer $\kappa(D)$ satisfying 
$$
0<\limsup_{{\rm integer}~~~ m\to +\infty} \frac{h^0(X,\lfloor mD\rfloor)}{m^{\kappa(D)}}
$$
if $h^0(X,\lfloor mD\rfloor)\neq 0$ for some $m>0$, otherwise let  
$\kappa(D)=-\infty$.
\end{defn}

\begin{lem}
The Kodaira dimension satisfies the following basic rules:\\
$(1)$  $\kappa(D)=\kappa(aD)$ for any positive  $a\in\Q$, and \\
$(2)$ $\kappa(D)=\kappa(D')$ if $D\sim_\Q D'$.
\end{lem}
\begin{proof}
(1) First assume that $a\in \Z$ and that $aD$ is integral. 
By definition, we have  $\kappa(aD)\le \kappa(D)$. 
For the converse we argue as follows. If $\kappa(D)=-\infty$ then 
equality $\kappa(aD)=\kappa(D)$ holds obviously. So, assume $\kappa(D)\ge 0$. 
If $h^0(X,\rddown{mD})=0$ for some $m>0$, then clearly
$$
h^0(X,\rddown{maD})\ge h^0(X,\lfloor mD\rfloor)
$$
If $h^0(X,\rddown{mD})\neq 0$ for some $m>0$, then $\rddown{mD}\sim G$ for some integral 
$G\ge 0$. Thus, $mD\sim G+\langle mD\rangle$ 
hence 
$$
\rddown{maD}=maD\sim aG+a\langle mD\rangle
$$
 where $a\langle mD\rangle$ is 
effective and integral. In particular, 
$$
h^0(X,\rddown{maD})\ge h^0(X,aG)\ge h^0(X,G)= h^0(X,\lfloor mD\rfloor)
$$ 
Thus, $h^0(X,\rddown{maD})$ grows at least as much as $h^0(X,\lfloor mD\rfloor)$ which implies that 
$$
\limsup_{m\to\infty} \frac{h^0(X,\lfloor maD\rfloor)}{m^{\kappa(D)}}\ge 
\limsup_{m\to\infty} \frac{h^0(X,\lfloor mD\rfloor)}{m^{\kappa(D)}}>0
$$ 
which implies that $\kappa(D)\le \kappa(aD)$ hence $\kappa(D)=\kappa(aD)$. 

For the general case: choose positive $b,b'\in \Z$ so that $bD=b'aD$ 
is integral. Then,  $\kappa(D)=\kappa(bD)=\kappa(b'aD)=\kappa(aD)$.

(2) This follows from the fact that there is some positive $a\in \Z$ such that $aD\sim aD'$. 
\end{proof}

\begin{exe}[Contractions]
Let $f\colon Y\to X$ be a contraction of normal projective varieties and $D$ a $\Q$-Cartier divisor on $X$.
Prove that $\kappa(D)=\kappa(f^*D)$.
\end{exe}

\begin{exe}
Let $D$ be a $\Q$-Cartier divisor on a normal projective variety $X$. Prove that,\\
$\bullet$ $\kappa(D)=-\infty$ $\Longleftrightarrow$ $h^0(\rddown{mD})=0$ for any $m\in\N$.\\
$\bullet$ $\kappa(D)=0$  $\Longleftrightarrow$ $h^0(\rddown{mD})\le 1$ for any $m\in\N$ with equality for some $m$.\\
$\bullet$ $\kappa(D)\ge 1$ $\Longleftrightarrow$ $h^0(\rddown{mD})\ge 2$ for some $m\in\N$.
\end{exe}

\begin{exa}[Curves]
Let $X$ be a smooth projective curve and $D$ a $\Q$-divisor on $X$. 
Then,\\ 
$\bullet$ $\kappa(D)=-\infty$ $\Longleftrightarrow$ $\deg D<0$, or $\deg D=0$ and $D$ 
is not torsion.\\
$\bullet$ $\kappa(D)=0$  $\Longleftrightarrow$ $\deg D=0$ and $D$ is torsion.\\
$\bullet$ $\kappa(D)= 1$ $\Longleftrightarrow$ $\deg D>0$.
\end{exa}

\begin{exa}
Let $X$ be a normal projective variety and $D$ a $\Q$-divisor on $X$.\\ 
$\bullet$ If $D=0$, then $\kappa(D)=0$.\\
$\bullet$ If $D\le 0$ but $D\neq 0$, then $\kappa(D)=-\infty$.\\ 
$\bullet$ Now assume that 
we have a projective birational morphism $f\colon X\to Y$ of normal varieties such that 
$D$ is contracted by $f$ and $D\ge 0$. Then, $\kappa(D)=0$. Moreover, 
if $L$ is any $\Q$-Cartier divisor on $Y$, then $\kappa(L)=\kappa(f^*L+D)$.
\end{exa}

\begin{exa}[Ample divisors]
Assume that $D$ is an ample divisor on a normal projective variety $X$. 
Using the Riemann-Roch theorem and the Serre vanishing it is easy to see that 
$\kappa(D)=\dim X$. Alternatively, One can assume that $D$ is a general very ample 
divisor and use the exact sequence 
$$
0\to \mathcal{O}_X((m-1)D) \to \mathcal{O}_X(mD) \to \mathcal{O}_D(mD)\to 0
$$
along with Serre vanishing and induction to show that $\kappa(D)=\dim X$.
This in particular implies that $\kappa(L)\le \dim X$ for any $\Q$-divisor $L$.
\end{exa}

\begin{defn}[Big divisor]
A $\Q$-divisor $D$ on a normal projective variety $X$ is called 
\emph{big} if $\kappa(D)=\dim X$.
\end{defn}

\begin{thm}[Kodaira lemma]\label{l-Kodaira-lemma}
Let $D,L$ be $\Q$-divisors on a normal projective variety $X$ where $D$ is big and 
$\Q$-Cartier. Then,
there is a rational number $\epsilon>0$ such that $D-\epsilon L$ is big.
\end{thm}
\begin{proof}
We may assume that $D$ is Cartier. If $\dim X=1$, the claim is trivial 
as $D$ would be ample. So, assume 
$\dim X\ge 2$. First assume that $A$ is a general very ample divisor. 
Then, the exact sequence 
$$
0\to \mathcal{O}_X(mD-A) \to \mathcal{O}_X(mD) \to \mathcal{O}_A(mD)\to 0
$$
shows that $h^0(X,mD-A)>0$ for some $m>0$ since $\kappa(D|_A)\le \dim X-1$.
So, we can find  $G\ge 0$ such that $mD\sim A+G$.
Choose $\epsilon>0$ sufficiently small so that $A':=A-m\epsilon L$ is ample. 
Then, $D-\epsilon L\sim_\Q \frac{1}{m}A'+\frac{1}{m}G$ which is big. 
\end{proof}

\begin{cor}\label{c-nef-to-ample}
 Let $D$ be a nef $\Q$-Cartier divisor on a normal projective variety $X$. 
 Then, the following are equivalent:
 
 $(1)$ $D$ is big,
 
 $(2)$ there is an effective 
 $\Q$-divisor $G$ and ample $\Q$-divisors $A_m$ such that $D\sim_\Q A_m+\frac{1}{m}G$ 
 for every $m$.
\end{cor}
\begin{proof}
(1) $\implies$ (2): By the Kodaira lemma, we can write  $D\sim_\Q A+G$ for some 
 effective $\Q$-divisor $G$ and ample $\Q$-divisor $A$. For each $m>0$, 
$$
mD\sim_\Q (m-1)D+A+G
$$ 
where $(m-1)D+A$ is ample by Kleiman's criterion. Now take 
 $A_m=\frac{1}{m}((m-1)D+A)$.
 
(2) $\implies$ (1): For each $m$, $\kappa(D)\ge \kappa(A_m)=\dim X$.

\end{proof}

One of the most useful theorems about Kodaira dimension is the existence of the so-called 
\emph{Iitaka fibration} which was proved by Iitaka. 

\begin{thm}[Iitaka fibration]
Let $D$ be a $\Q$-Cartier divisor on a normal projective variety $X$ 
with $\kappa(D)\ge 0$. Then, there are projective morphisms $f\colon W\to X$ and $g\colon W\to Z$ 
from a smooth $W$ such that\\
$\bullet$ $f$ is birational,\\
$\bullet$  $g$ is a contraction,\\ 
$\bullet$ $\kappa(D)=\dim Z$, and\\ 
$\bullet$ if $V$ is the generic fibre of $g$, then $\kappa(f^*D|_V)=0$. 
\end{thm}

For a proof see [\ref{Iitaka}]. If $\kappa(D)=0$, we can take $f$ to be any resolution 
and $g$ the constant map to a point. If $\kappa(D)=\dim X$, then again we can take $f$ 
to be any resolution and $g=f$. 

\begin{rem}
Let $D$ be a Cartier divisor on a normal projective variety $X$. It is well-known that 
$$
\kappa(D)=\max\{\dim \phi_{mD}(X)\mid h^0(X,mD)\neq 0\}
$$
if $h^0(X,mD)\neq 0$ for some $m\in \N$.
\end{rem}

\clearpage
~ \vspace{2cm}
\section{\textbf{The log minimal model program: main conjectures}}
\vspace{1cm}

\begin{defn}[Contraction of an extremal ray]
Let $R$ be an extremal ray of $\overline{NE}(X/Z)$ of a normal variety $X/Z$. 
A contraction $f\colon X\to Y/Z$ is the contraction of $R$ if
$$
f(C)=pt. \Longleftrightarrow [C]\in R
$$
for any curve $C\subset X$.
\end{defn}

\begin{rem}[Types of contractions]
For the contraction of an extremal ray $R$ we have the following possibilities:
\begin{description}
 \item[Divisorial] $f$ is birational and it contracts at least some divisors.
 \item[Small]   $f$ is birational and it does not contract divisors.
 \item[Fibration]  $f$ is not birational hence $\dim X>\dim Y$.
\end{description}
\end{rem}

\begin{defn}[Minimal model-Mori fibre space] \label{d-mmodel}
 Let $(X/Z,B)$, $(Y/Z,B_Y)$ be lc pairs and $\phi\colon X\bir Y/Z$ a birational map whose inverse does not
 contract any divisors such that $B_Y=\phi_*B$. Moreover, assume that  
 
 $\bullet$ $d(E,X,B)\le d(E,Y,B_Y)$ for any prime divisor $E$ on resolutions of $X$ 
 but with strict inequality if $E$ is on $X$ and contracted by $\phi$.\\
  We say that $(Y/Z,B_Y)$ is a \emph{log minimal model} for $(X/Z,B)$
 if

 $\bullet$ $K_Y+B_Y$ is nef$/Z$.\\
On the other hand, we say that $(Y/Z,B_Y)$ is a \emph{Mori fibre space}  for $(X/Z,B)$ if
 
 $\bullet$ there is a contraction $g\colon Y\to T/Z$ of a $(K_Y+B_Y)$-negative extremal ray 
 with $\dim Y>\dim T$.
\end{defn}

\begin{rem}
There is no consensus on what the definition of  minimal models and Mori fibre spaces 
should be. There are varying definitions. Sometime one might like to 
add assumptions such as  $(Y/Z,B_Y)$ being $\Q$-factorial dlt, or one might like to weaken assumptions 
such as allowing $\phi^{-1}$ to contract certain divisors. However, these definitions 
are very similar and differ only in minor details. 
The above definition is most suitable for these lectures. 
\end{rem}

Now we come to the two most important problems in birational geometry.

\begin{conj}[Minimal model] Let $(X/Z,B)$ be a lc pair.
Then, $(X/Z,B)$ has a log minimal model or a Mori fibre space.
\end{conj}

After the contributions from many people in the 1980's notably Mori, Kawamata, Shokurov, Koll\'ar, Reid, 
etc, this conjecture was settled in dimension $3$ by Shokurov [\ref{Shokurov-log-flips}][\ref{Shokurov-log-models}] 
and Kawamata [\ref{Kawamata-termination}]. The conjecture was verified by Shokurov [\ref{Shokurov-ordered}] 
and Birkar [\ref{Birkar-mmodel-Shokurov}][\ref{Birkar-mmodel}] in dimension $4$. The latter paper also settles most cases in dimension $5$. Moreover, Birkar [\ref{Birkar-mmodel}][\ref{Birkar-mmodel-II}]
[\ref{Birkar-WZD}] proposes inductive approaches to the minimal model conjecture. 
Finally, Birkar-Cascini-Hacon-McKernan [\ref{BCHM}] settles the conjectures for pairs of general 
type in any dimension (see also [\ref{Birkar-Paun}]).

\begin{conj}[Abundance] Let $(Y/Z,B_Y)$ be a lc pair with $K_Y+B_Y$ nef$/Z$.
Then, $K_Y+B_Y$ is semi-ample$/Z$, that is, 
there is a contraction $h\colon Y\to S/Z$ 
and an ample$/Z$ $\R$-divisor $H$ on $S$ such that
$$
K_Y+B_Y\sim_\R h^* H
$$
\end{conj}

The conjecture was proved in a series of papers of Miyaoka [\ref{Miyaoka-Kodaira}]
[\ref{Miyaoka-abundance}], Kawamata [\ref{Kawamata-abundance}], and 
Keel-Matsuki-McKernan [\ref{Keel-Matsuki-McKernan}] (see also 
Shokurov [\ref{Shokurov-log-models}] and Miyaoka-Peternell [\ref{Miyaoka-Peternell}]).
It is expected that analytic methods are more capable of attacking the 
abundance conjecture. For example see Demailly-Hacon-P\u{a}un [\ref{Demailly-Hacon-Paun}] 
for a recent development.

One problem that is at the heart of the two above conjectures is the following.

\begin{conj}[Nonvanishing] Let $(X/Z,B)$ be a lc pair such that $K_X+B$ is 
pseudo-effective$/Z$.
Then, $K_X+B\sim_\R M/Z$ for some $M\ge 0$.
\end{conj}
 
 The conjecture is proved in dimension $3$ mainly through Miyaoka's paper [\ref{Miyaoka-Kodaira}]. 
 In higher dimension very little is known. Even worse, there does not seem to be the slightest 
 idea to approach this conjecture. Miyaoka's arguments heavily rely on the Riemann-Roch 
 theorem which becomes very complicated in dimension $\ge 4$. 
Birkar [\ref{Birkar-mmodel-II}]
proves that the nonvanishing conjecture implies the minimal model conjecture.
On the other hand, the methods of Demailly-Hacon-P\u{a}un [\ref{Demailly-Hacon-Paun}] 
seem to be capable of reducing abundance to the nonvanishing conjecture.
So, it is fair to say that the nonvanishing conjecture is currently the most central 
problem in birational geometry.
Note that if $B$ has rational coefficients, then the conjecture is saying that 
$h^0(X,m(K_X+B))\neq 0$ for some integer $m>0$. This explains the nonvanishing terminology.

In general minimal models and Mori fibre spaces are not unique.
This makes it difficult or perhaps impossible to canonically identify a minimal model or a  
Mori fibre space using data on $(X/Z,B)$. We will see later that the LMMP 
provides a way to reach these models but the path is not unique. 
However, when $K_X+B$ is rational and big$/Z$, one gets something that is quite 
close to a minimal model.

\begin{conj}[Finite generation]\label{conj-fg}
 Let $(X/Z,B)$ be a lc pair such that $K_X+B$ is a 
$\Q$-divisor. Then, the log canonical algebra 
$$
\mathcal{R}(X/Z,K_X+B):=\bigoplus_{m\ge 0} f_*\mathcal{O}_X(\rddown{m(K_X+B)}) 
$$
is a finitely generated $\mathcal{O}_Z$-algebra where $f$ is the given morphism 
$X\to Z$.
\end{conj}

When $K_X+B$ is big$/Z$, the finite generation conjecture implies that 
the scheme $S'=\Proj \mathcal{R}(X/Z,K_X+B)$ is actually projective over $Z$. 
It turns out that $S'$ is very close to being a log minimal model for $(X/Z,B)$.
In fact, the minimal model and abundance conjectures imply that one can 
construct a log minimal model $(Y/Z,B_Y)$ for $(X/Z,B)$ and $S'$ is nothing but 
the $S$ appearing in the abundance conjecture. In particular, one has a 
birational contraction $Y\to S$ contracting those curves $C$ on $Y/Z$ 
which are $K_Y+B_Y$-numerically trivial.

It is worth to mention that the finite generation conjecture is very strong. 
It is more or less understood that it is actually equivalent to the minimal model and abundance 
conjectures combined.

The following conjecture of Iitaka has been a central problem in birational geometry 
since the 1970's. Kawamata [\ref{Kawamata-Iitaka}] proved that it follows from the minimal model and 
abundance conjectures.

\begin{conj}[Iitaka]
Let $f\colon X\to Z$ be a contraction of smooth projective varieties and $F$ the generic fibre.
Then,
$$
\kappa(K_X)\ge \kappa(K_Z)+\kappa(K_F)
$$
\end{conj}

Another major problem due to Iitaka is the following.

\begin{thm}[Invariance of plurigenera]
Let $f\colon X\to Z$ be a smooth projective contraction of smooth varieties.
Then, the function
$$
h^0(X_z,mK_{X_z})
$$
is constant.
\end{thm}

The theorem was proved by Siu [\ref{Siu-invariance}] and Tsuji [\ref{Tsuji-invariance}] using analytic methods. However, it is very desirable to 
find an algebraic proof. Nakayama [\ref{Nakayama-invariance}] proved that 
the theorem follows from the minimal model and abundance conjectures. When the fibres are of 
general type, the theorem follows easily from [\ref{BCHM}] and Grothendieck's standard 
techniques used in the proof of the semi-continuity theorem (cf. 
Birkar [\ref{Birkar-topics}, Remark 3.9.2]) at least when $m\ge 2$.

Let $D$ be an $\R$-Cartier divisor on a normal variety $X/Z$. A Fujita-Zariski decomposition$/Z$ 
for $D$ is an expression $D=P+N$ such that 
\begin{enumerate}
\item $P$ and $N$ are $\R$-Cartier divisors, 
\item $P$ is nef$/Z$, $N\ge 0$, and 
\item if $f\colon W\to X$ is a projective birational morphism from a normal variety, and $f^*D=P'+N'$ 
with $P'$ nef$/Z$ and $N'\ge 0$, then $P'\le f^*P$.
\end{enumerate}

\begin{conj}[Zariski decomposition] Let $(X/Z,B)$ be a lc pair such that $K_X+B$ is 
pseudo-effective$/Z$. Then, there is a resolution of $X$ on which the pullback of 
$K_X+B$ admits a Fujita-Zariski decomposition.
\end{conj}

It is not difficult to see that this conjecture follows from the minimal model 
and abundance conjectures. For the inverse direction see Birkar [\ref{Birkar-WZD}].

\begin{defn}[Log flip]
 Let $(X/Z,B)$ be a lc pair and $f\colon X\to Y/Z$ the contraction of a 
 $K_X+B$-negative extremal ray of small type. The log flip of this flipping 
 contraction is a  diagram

$$
\xymatrix{ X\ar[rd]_{f}& \dashrightarrow & X^+\ar[ld]^{f^+} \\
&Y&}
$$

such that

\begin{itemize}
\item $X^+$ is a normal variety, projective$/Z$,
\item $f^+$ is a small projective birational contraction$/Z$,
\item $-(K_X+B)$ is ample over $Y$ (by assumption), and $K_{X^+}+B^+$ is ample over $Y$
where $B^+$ is the birational transform of $B$.
 \end{itemize}
\end{defn}

Existence of log flips occupied a central place in birational geometry for 
two decades. Mori proved its existence for terminal singularities in dimension 
$3$ [\ref{Mori-flip}]. Shokurov proved the existence in dimensions $3$ 
[\ref{Shokurov-log-flips}][\ref{Shokurov-log-models}] 
in the lc case and in dimension $4$ in the klt case [\ref{Shokurov-pl-flips}]. 
Shokurov also made major advances on the problem in any dimension in the above papers 
and on the way introduced some of the most important techniques in the field.
Fujino proved the existence in dimension $4$ in the lc case. Hacon-McKernan [\ref{Hacon-McKernan}]
filled in the main missing part of the Shokurov's flip program through their 
extension theorem which originated from Siu's methods in analytic geometry.
Birkar-Cascini-Hacon-McKernan 
[\ref{BCHM}] finally settled the problem in the klt case in any dimension.
It is worth to mention that the $\Q$-factorial dlt case follows immediately 
from the klt case.

\begin{defn}[Log minimal model program: LMMP]
Let $(X/Z,B)$ be a lc pair. The following process is called
the log minimal model program if every step exists:
If $K_X+B$ is not nef/Z, then there is an $(K_X+B)$-extremal ray $R/Z$ and its contraction $f\colon X\to Y/Z$.
If $\dim Y<\dim X$, then we get a Mori fibre space and we stop. If $f$ is a divisorial contraction,
we replace $(X/Z,B)$ with $(Y/Z,B_Y:=f_*B)$ and continue. If $f$ is a flipping contraction, 
we replace $(X/Z,B)$ with the
right hand side  $(X^+/Z,B^+)$ of the flip and continue. 
After finitely many steps, we get a log minimal model or a Mori fibre space. So, 
the LMMP implies the minimal model conjecture.
\end{defn}

The LMMP can be considered in different levels of generality. For example one can start with 
a klt pair or a dlt pair maybe with the extra assumption of $\Q$-factoriality. It turns out that 
these properties are preserved in the course of the program. In particular, in these cases, all 
the pieces of the program are already established with the exception of termination, that is, 
the expectation that the program stops after finitely many steps.

\begin{exa}
Classical MMP for smooth projective surfaces.
\end{exa}

\begin{conj}[Termination]
Let $(X/Z,B)$ be a lc pair. Any sequence of $(K_X+B)$-flips$/Z$ terminates.
\end{conj}

This conjecture was established by Kawamata [\ref{Kawamata-termination}] for klt 
pairs in dimension $3$ and in full generality in dimension $3$ by Shokurov [\ref{Shokurov-log-models}].
Not much is known in higher dimension. Alexeev-Hacon-Kawamata [\ref{Alexeev-Hacon-Kawamata}] verified some 
special cases in dimension $4$ and upon this Birkar [\ref{Birkar-termination-4}] proved that 
the conjecture holds in dimension $4$ when $\kappa(K_X+B)\ge 2$. Moreover, Birkar 
[\ref{Birkar-ACC-termination}] proved that the conjecture follows, in most cases, from 
the LMMP in lower dimensions and the ACC conjecture on lc thresholds. Birkar-Cascini-Hacon-McKernan 
[\ref{BCHM}] proved the conjecture for some special kind of sequences of flips (i.e. with scaling)
in the klt case when $B$ is big$/Z$.
 
\begin{defn}[LMMP with scaling]\label{d-LMMP-scaling}
Let $(X/Z,B)$ be a lc pair. In the process of the LMMP, there are choices to be made, that is, 
in each step one needs to choose a negative extremal ray. However, there are usually more than 
one ray and sometimes infinitely many (see Example \ref{exa-rays-surfaces}). So, the LMMP 
is not a uniquely determined process. Different processes may lead to different outcomes.
However, there is a special version of the LMMP which is more restricted: it is directed by some divisor.

Assume that $C\ge 0$ is an $\R$-divisor such that $(X/Z,B+C)$ is lc and 
$K_X+B+C$ is nef$/Z$. For example, one can take $C$ to be a general and sufficiently 
ample divisor. 
 Suppose that either $K_X+B$ is nef/$Z$ or there is an extremal ray $R/Z$ such
that $(K_X+B)\cdot R<0$, $(K_X+B+\lambda_1 C)\cdot R=0$, and $K_X+B+\lambda_1 C$ is nef$/Z$ where
$$
\lambda_1:=\inf \{t\ge 0~|~K_X+B+tC~~\mbox{is nef/$Z$}\}
$$
 If $R$ defines a Mori fibre structure, we stop. Otherwise assume that $R$ gives a divisorial 
contraction or a log flip $X\bir X'$. We can now consider $(X'/Z,B'+\lambda_1 C')$  where $B'+\lambda_1 C'$ is 
the birational transform 
of $B+\lambda_1 C$ and continue the argument. That is, suppose that either $K_{X'}+B'$ is nef/$Z$ or 
there is an extremal ray $R'/Z$ such
that $(K_{X'}+B')\cdot R'<0$, $(K_{X'}+B'+\lambda_2 C')\cdot R'=0$, and $K_{X'}+B'+\lambda_2 C'$ is nef$/Z$ where
$$
\lambda_2:=\inf \{t\ge 0~|~K_{X'}+B'+tC'~~\mbox{is nef/$Z$}\}
$$
 By continuing this process, we obtain a 
special kind of LMMP$/Z$ which is called the \emph{LMMP$/Z$ on $K_X+B$ with scaling of $C$}; note that it is not unique. 
This kind of LMMP was first used by Shokurov [\ref{Shokurov-log-flips}].
When we refer to \emph{termination with scaling} we mean termination of such an LMMP.
\end{defn}

The existence of the extremal rays that we will need in the LMMP with scaling is 
ensured by the following lemma (see Birkar [\ref{Birkar-mmodel}, Lemma 3.1] for a more general 
statement). Let $D$ be an $\R$-divisor on a normal variety $X$ projective$/Z$. 
Recall from Definition \ref{d-free-ample-divisor/Z} that $D$ is \emph{big$/Z$} 
if $D\sim_\R G+A/Z$ for some $\R$-divisor $G\ge 0$ and some ample$/Z$ $\R$-divisor $A$. 

\begin{lem}\label{l-ray-scaling}
Let $(X/Z,B+C)$ be a $\Q$-factorial lc pair where $B,C\ge 0$, 
$K_X+B+C$ is nef/$Z$, and $(X/Z,B)$ is klt with $B$ big$/Z$. Then, either $K_X+B$ is also nef/$Z$ or there is an extremal ray $R/Z$ such
that $(K_X+B)\cdot R<0$, $(K_X+B+\lambda C)\cdot R=0$, and $K_X+B+\lambda C$ is nef$/Z$ where
$$
\lambda:=\inf \{t\ge 0~|~K_X+B+tC~~\mbox{is nef/$Z$}\}
$$
\end{lem}
\begin{proof}
Suppose that $K_X+B$ is not nef$/Z$. Let $A$ be an ample$/Z$ divisor. Since $B$ is 
big$/Z$,  $B\sim_\R G+A/Z$ for some $\R$-divisor $G\ge 0$ and some ample$/Z$ $\R$-divisor $A$. 
We can write 
$$
K_X+B= K_X+(1-\epsilon) B+\epsilon B\sim_\R K_X+(1-\epsilon) B+\epsilon (G+A)/Z
$$
and if $\epsilon>0$ is small enough then $(X/Z,(1-\epsilon) B+\epsilon G)$ is klt.
Now if $R$ is a $K_X+B$-negative extremal ray$/Z$, then by the cone theorem there is 
some curve $\Gamma$ generating $R$ such that 
$$
-2\dim X\le (K_X+(1-\epsilon) B+\epsilon G)\cdot \Gamma<0
$$ 
hence $A\cdot \Gamma<2\dim X$. Since $A$ is ample, $\Gamma$ belongs to a bounded 
family of curves on $X$. So, such $\Gamma$ can generate only finitely many extremal 
rays. Therefore, there are only finitely many $K_X+B$-negative extremal rays$/Z$. 
Now the lemma is trivial by letting 
$$
\frac{1}{\lambda}:=\min \{\frac{C\cdot R_i}{-(K_X+B)\cdot R_i}\}
$$
where $R_i$ runs through the finitely many $K_X+B$-negative extremal rays$/Z$.
\end{proof}

\clearpage
~ \vspace{2cm}
\section{\textbf{{Cone and contraction, vanishing, nonvanishing, and base point freeness}}}
\vspace{1cm}

The cone theorem allows us to perform the very first step of the LMMP, that is, 
to identify a negative extremal ray and to contract it. The formulation of the 
cone theorem was mainly inspired by Mori's work. However, the proof we present 
came from an entirely different set of ideas conceived and developed by Shokurov 
and Kawamata (except the existence of rational curves which relies on Mori's 
original work). These latter ideas proved to be fundamental, 
far beyond the proof of the cone theorem.

\begin{thm}[Cone and contraction]
Let $(X/Z,B)$ be a klt pair of dimension $d$ with $B$ rational.
Then, there is a countable set of $(K_X+B)$-negative extremal rays $\{R_i\}/Z$ such that\\
$\bullet$ $\overline{NE}(X/Z)=\overline{NE}(X/Z)_{K_X+B\ge 0}+\sum_iR_i$.\\
$\bullet$ $R_i$ can be contracted.\\
$\bullet$ $\{R_i\}$ is discrete in $\overline{NE}(X/Z)_{K_X+B< 0}$.\\
$\bullet$ Each $R_i$ contains the class of some rational curve $C_i$ satisfying 
$$
-2d\le (K_X+B)\cdot C_i
$$
$\bullet$ Let $f\colon X\to Y/Z$ be the contraction of a $K_X+B$-negative 
extremal ray $R/Z$, and let $L$ be a Cartier divisor on $X$ with $L\cdot R=0$. 
Then, there is a Cartier divisor $L_Y$ on $Y$ such that $L\sim f^*L_Y$. 
\end{thm}

Let $(X/Z,B)$ be a klt pair. If $R$ is a $K_X+B$-negative extremal ray$/Z$, then 
we will see that it is not difficult   
to find an ample$/Z$ $\Q$-divisor $H$ such that $H+t(K_X+B)$ is nef$/Z$ and that $R$  
is the only extremal ray$/Z$ satisfying $(H+t(K_X+B))\cdot R=0$. 
The main idea is to prove that $t$ is a rational number and to prove that 
$H+t(K_X+B)$ is semi-ample$/Z$. The latter divisor then defines a contraction that 
happens to be exactly the contraction of $R$. This should make it clear why we are interested 
in the next two theorems. These theorems imply the cone and contraction (except the 
statement about rational curves and the boundedness $-2d\le (K_X+B)\cdot C_i$) theorem using  
a rather easy and combinatorial argument.

\begin{thm}[Base point free]\label{t-base-point-free}
Let $(X/Z,B)$ be a klt pair with $B$ rational. Suppose that for a Cartier
divisor $D$ which is nef$/Z$, there is some rational number $a>0$ such that $aD-(K_X+B)$ is nef and big$/Z$. Then,
$mD$ is free$/Z$ for any natural number $m\gg 0$.
\end{thm}

\begin{thm}[Rationality]
Let $(X/Z,B)$ be a klt pair with $B$ rational. Let $H$ be an ample$/Z$ Cartier divisor on $X$.
Suppose that $K_X+B$ is not nef$/Z$. Then,
$$
\lambda=\max\{t>0\mid t(K_X+B)+H ~\mbox{is nef$/Z$}\}
$$
is a rational number. Moreover, one can write $\lambda=\frac{a}{b}$ where $a,b\in\N$
and $b$ is bounded depending only on $(X/Z,B)$.
\end{thm}

The proof of this theorem is quite similar to the proof of the base point free theorem and
Shokurov nonvanishing theorem. So, we omit its proof (cf. [\ref{Kollar-Mori}]).
Moreover, the proof of above theorems in the relative case is, conceptually, very much the same 
as their proof in the absolute projective case (i.e. $Z=\rm pt$) so for simplicity 
we just give the proofs in the projective case.

\begin{proof}(of Cone and Contraction theorem when $Z=\rm pt$)\\
\emph{Step 1.} We may assume that $K_X+B$ is not nef. For any nef $\Q$-Cartier divisor $D$ define
$$
F_D=\{c\in \overline{NE}(X)\mid D\cdot c=0\}
$$
which is an extremal face of $\overline{NE}(X)$. Let $\mathcal{C}$ be the closure of 
$$
\overline{NE}(X)_{K_X+B\ge 0}+{\sum_{D}F_D}
$$
where $D$ runs over nef $\Q$-Cartier divisors for which $\dim F_D=1$.
Suppose that $\mathcal{C}\neq \overline{NE}(X)$.
Choose a point $c\in\overline{NE}(X)$ which does not belong to $\mathcal{C}$. Now choose a rational
linear function $\alpha\colon N_1(X)\to \R$ which is positive on $\mathcal{C}-\{0\}$ but negative on
$c$. This linear function is defined by some $\Q$-Cartier divisor $G$.

\emph{Step 2.} If $t\gg 0$, then $G-t(K_X+B)$  is positive on $\overline{NE}(X)_{K_X+B\le 0}$.
\footnote{To see this note that it is enough to verify this positivity on the 
compact set $\Omega\cap \overline{NE}(X)_{K_X+B\le 0}$ 
where $\Omega$ is an appropriate hyperplane not passing through the origin.  
If $t\gg 0$ and if $(G-t(K_X+B))\cdot c_t=0$ for some 
$c_t\in \Omega\cap \overline{NE}(X)_{K_X+B< 0}$, then $G\cdot c_t<0$. But such $c_t$ are away from 
$\overline{NE}(X)_{K_X+B=0}$ hence $0<a_1<|(K_X+B)\cdot c_t|<a_2$ and $0<|G\cdot c_t|<a_3$ 
for certain $a_1,a_2,a_3$ independent of $t$. This implies that for $t\gg 0$, there cannot be 
such $c_t$.}
Now let
$$
\gamma=\min\{t>0\mid G-t(K_X+B) ~\mbox{is nef on}~ \overline{NE}(X)_{K_X+B\le 0}\}
$$
So 
$$
(G-\gamma(K_X+B))\cdot c' =0
$$ 
for some 
$$
c'\in \overline{NE}(X)_{K_X+B<0}
$$ 
In particular, this implies that $G-\gamma(K_X+B)$ is 
positive on $\overline{NE}(X)_{K_X+B\ge 0}$ otherwise there would be $c''\in \overline{NE}(X)_{K_X+B>0}$ 
such that 
$$(G-\gamma(K_X+B))\cdot c'' \le 0
$$
 which in turn implies that $G-\gamma(K_X+B)$ is non-positive 
on the point $[c',c'']\cap \overline{NE}(X)_{K_X+B=0}$, a contradiction.
 Therefore, $G-\gamma(K_X+B)$ is nef and 
 $$
 H=G-\gamma'(K_X+B)
 $$ 
 is ample for some rational
number $\gamma'>\gamma$ close to $\gamma$. After replacing $H$ with a multiple we could 
assume that $H$ is Cartier.

\emph{Step 3.} Now the rationality theorem shows that
$$
\lambda=\max\{t>0\mid H+t(K_X+B) ~\mbox{is nef}\}
$$
is a rational number (note that $\lambda=\gamma'-\gamma$). Put $D=H+\lambda(K_X+B)$. By the arguments in step 2, $F_D\cap\mathcal{C}=0$.

Here it may happen that $\dim F_{D}>1$.
In that case, let $H_1$ be an ample divisor which is linearly independent of $K_X+B$ on $F_D$. 
For $s>0$ let
$$
\lambda(s)=\max\{t>0\mid sD+ H_1+t(K_X+B) ~\mbox{is nef}\}
$$
 
 Obviously, $\lambda(s)$ is bounded from above where the bound does not depend on $s$.
 Moreover, since $\lambda(s)$ has bounded denominator and $\lambda(s')\ge \lambda(s)$ 
 when $s'>s$, $\lambda(s)$ is independent of $s$ when $s\gg 0$.
 
 Therefore, if $s\gg 0$, then
$$
 F_{sD+\epsilon H_1+\lambda(s)(K_X+B)}\subset F_{D}
$$
where the inclusion is strict because $H_1$ and $K_X+B$ are linearly independent on $F_D$.
Arguing inductively, there is a rational nef divisor $D'$ such that $F_{D'}\subseteq F_D$ 
and $\dim F_{D'}=1$. This contradicts the above assumptions. Therefore, $\mathcal{C}=\overline{NE}(X)$.

\emph{Step 5.} 
Let $D$ be a nef $\Q$-Cartier divisor such that $\dim F_D=1$ and such that $D=H+t(K_X+B)$ 
for some ample Cartier divisor $H$ and some rational number $t>0$. Then, by the base point 
free theorem $D$ is semi-ample. Therefore, the extremal ray $F_D$ can be contracted.

\emph{Step 6.}
Next we show that the set $\{F_D\}$, where $D$ is as in Step 5, does not have any accumulation 
in $\overline{NE}(X)_{K_X+B<0}$. Assume otherwise and let $D_1,D_2,\dots$ be a sequence 
such that $\{F_{D_i}\}$ has a limit away from $\overline{NE}(X)_{K_X+B=0}$. 
By assumptions, $D_i=H_i+t_i(K_X+B))$ for certain 
ample Cartier divisors $H_i$ and rational numbers $t_i>0$. Fix an ample Cartier divisor $H$. 
By the above arguments, for each $i$, there is some $s_i\gg 0$ such that if 
$$
D_i'=H+s_iD_i+t_i'(K_X+B)
$$
 where $t_i'$ is the largest number making $D_i'$ nef, then
$F_{D_i'}=F_{D_i}$. Now let $c_i\in F_{D_i}$ such that $(K_X+B)\cdot c_i=-1$. 
The $c_i$ converge to some $c$ with $(K_X+B)\cdot c=-1$
Then, 
$$
H\cdot c_i=-t_i'(K_X+B)\cdot c_i=t_i'
$$ 
By construction, $H\cdot c_i$ is a bounded number, and $t_i'$ is a rational number with 
bounded denominator. Therefore, there are only finitely many possibilities for the $t_i'$.
So, there are only finitely many possibilities for $H\cdot c_i$ and this is possible 
only if there are only finitely many $c_i$ if $H$ is chosen appropriately. 
This in particular, implies that
$$
 \overline{NE}(X)=\overline{NE}(X)_{K_X+B\ge 0}+\sum_D F_D
$$
where $D$ runs through nef $\Q$-Cartier divisors with $\dim F_D=1$. This in 
turn implies that 
$$
 \overline{NE}(X)=\overline{NE}(X)_{K_X+B\ge 0}+\sum R_i
$$
where $R_i$ run through the $(K_X+B)$-negative extremal rays because each 
such extremal ray has to be one of the above $F_D$. By construction, each 
$R_i$ is contractible and contains the class of some curve $C_i$.

\emph{Step 7.} 
The fact that $C_i$ can be chosen to be rational and satisfying $-2d\le (K_X+B)\cdot C_i$
is proved using very different arguments. We refer the interested reader to 
[\ref{Kollar-Mori}] and [\ref{Kawamata-ext-rays}]. 

\emph{Step 8.} Let $L$, $R$, and $f\colon X\to Y$ be as in the last claim of the theorem. 
Let $D$ be a nef Cartier divisor such that $F_D=R$ and $D=H+t(K_X+B)$ where $H$ is 
an ample Cartier divisor. If $a\gg 0$, then $L+aD$ is nef (because $D$ is positive 
on $\overline{NE}(X)_{K_X+B\ge 0}$ hence $L+aD$ is also positive on it when 
$a\gg 0$) and $F_{L+aD}=F_D$. 
By the base point free theorem, for some large $m$, $m(L+aD)$ and $(m+1)(L+aD)$ are both base point 
free hence both are pullbacks of Cartier divisors on $Y$. This implies that 
$L+aD$ is also the pullback of some Cartier divisor on $Y$. If we choose 
$a$ sufficiently divisible then $aD$ is the   pullback of some Cartier divisor on $Y$.
Therefore,  $L$ is the pullback of some Cartier divisor $L_Y$ on $Y$.\\
\end{proof}

\begin{thm}[Kamawata-Viehweg vanishing]
Let $(X/Z,B)$ be a klt pair. Let $N$ be an integral $\Q$-Cartier divisor on
$X$ such that $N\equiv K_X+B+M/Z$ where $M$ is nef and big$/Z$. Then,
$$
 R^if_*\mathcal{O}_X(N)=0
$$
for any $i>0$ where $f\colon X\to Z$ is the given morphism.
\end{thm}

\begin{thm}[Shokurov Nonvanishing]\label{t-Shokurov-nonvanishing}
Let $(X,B)$ be a projective klt pair where $B$ is rational. Let $G\ge 0$ be a Cartier divisor such
that $aD+G-(K_X+B)$ is nef and big for some nef Cartier divisor $D$ and rational number $a>0$. Then,
$$
H^0(X,mD+G)\neq 0
$$
for $m\gg 0$.
\end{thm}

\begin{rem}\label{rem-extension}
Let $S$ be a smooth prime divisor on a smooth projective variety $X$. Then, we have an exact
sequence
$$
0\to \mathcal{O}_X(-S)\to \mathcal{O}_X\to \mathcal{O}_S\to 0
$$

Moreover, if $D$ is a divisor on $X$ we get another exact sequence
$$
0\to \mathcal{O}_X(D-S)\to \mathcal{O}_X(D)\to \mathcal{O}_S(D\vert_S)\to 0
$$
which gives the following exact sequence of cohomologies
$$
0\to H^0(X, D-S)\to H^0(X, D)\toover{f} H^0(S, D\vert_S)\toover{g}
$$
$$
\hspace{2cm}  H^1(X, D-S)\toover{h} H^1(X, D)\to H^1(S, D\vert_S)
$$

It often occurs that we want to prove that $H^0(X, D)\neq 0$. The sequence above is
extremely useful in this case. However, in general $f$ is not surjective so
$H^0(S, D\vert_S)\neq 0$
does not necessarily imply $H^0(X, D)\neq 0$. But if $f$ is surjective 
$H^0(S, D\vert_S)\neq 0$ implies $H^0(X, D)\neq 0$. To prove that $f$ is surjective
we should prove that  $H^1(X, D-S)=0$ or the weaker condition that $h$ is injective. In particular,
if $f$ is surjective and $H^0(S, D\vert_S)\neq 0$, then $S$ is not a component of $\Fix D$.
\end{rem}

\begin{proof}(of the base point free theorem)
\emph{Step 1.} We will use induction on dimension. The case $\dim X=1$ is trivial so we assume 
that $\dim X\ge 2$.

\emph{Step 2.} By assumptions, $aD=K_X+B+A$ where $A$ is nef and big. By Corollary \ref{c-nef-to-ample}, 
we can write $A\sim_\Q A'+\frac{1}{n}C$ where $A'$ is ample, $C\ge 0$ 
with fixed support, and $n$ is arbitrarily large. Thus, we can choose $n$ 
so that $(X,B':=B+\frac{1}{n}C)$ is klt and $aD\sim_\Q K_X+B'+A'$. So, we 
could simply assume that $A$ is ample.

\emph{Step 3.} Applying the nonvanishing theorem with $G=0$ implies that $H^0(X,mD)\neq 0$ for $m\gg 0$.
Pick $b\in \N$ such that $H^0(X,bD)\neq 0$. We will prove that $\Bs |mD|\subset \Bs |mbD|$
 for any $m\gg 0$ where the inclusion is strict. 
  Now suppose that $bD$ is not free. By taking a log resolution
 $f\colon Y\to X$ where all the divisors involved are with simple normal crossings, we can
 further assume that $f^* bD\sim M+F$ where $M$ is free, $F$ is the fixed part, and
 $\Supp F=\Bs |f^*bD|$ (see Remark \ref{rem-decomposition}).  For any 
 natural number $l$, 
 $$
 lf^*bD\sim lM+lF\sim N+lF
 $$
  where we may assume that 
 $N\ge 0$  is reduced with smooth components.

\emph{Step 4.} 
If $F=0$ for any choice of $b$, then go to Step 7. We assume otherwise and we will prove that some
 component of $F$ is not in $\Bs |f^*mD|$ for any $m\gg 0$. For a fixed large $m_0\in\N$ we have

$$
m_0D=K_X+B+(m_0-a)D+A
$$
and by Corollary \ref{c-nef-to-ample} we can write 
$$
f^*((m_0-a)D+A)=E+H
$$
where $H$ is ample and $E\ge 0$ has fixed support with sufficiently small coefficients, 
and $\Supp (N+lF)\subseteq \Supp E$. So,

$$
 f^*(m+m_0)D=f^*(K_X+B+mD+(m_0-a)D+A)=
$$
$$
 f^*(K_X+B+tD)+f^*(m-t)D+f^*((m_0-a)D+A)=
$$
$$
 f^*(K_X+B+tD)+f^*(m-t)D+E+H
$$

Since $(X,B)$ is klt, we can write 
$$
f^*(K_X+B)=K_Y+B_Y'-L'
$$ 
where $(Y,B_Y')$ is klt, $L\ge 0$ is exceptional over $X$, and $L$ and $B_Y'$ 
have no common components.
 We can choose $t>0$ and possibly modify $E$ (by changing its coefficients slightly) 
 in a way that

$$
  f^*(K_X+B+tD)+E= K_Y+B_Y'-L'+\frac{t}{lb}N+\frac{t}{b}F+E=K_Y+B_Y+S-L
$$
where $(Y,B_Y)$ is klt and $S$ is reduced, irreducible and a component of $F$,  
$L\ge 0$ is exceptional over $X$, and $L$ and $B_Y+S$ 
have no common components
(note that we can take $l$ to be as large as we wish). Therefore,

$$
f^*(m+m_0)D+\lceil L \rceil-S= K_Y+B_Y+\lceil L\rceil-L+f^*(m-t)D+H
$$

\emph{Step 5.} 
Now $(Y, B_Y+\lceil L\rceil-L)$ is klt hence by applying Kawamata-Viehweg vanishing theorem, we get

$$
 H^i(Y, f^*(m+m_0)D+\lceil L \rceil-S)=0
$$
for all $i>0$. Therefore, by Remark \ref{rem-extension}, we get the following exact sequence

$$
0\to H^0(Y, f^*(m+m_0)D+\lceil L \rceil-S)\to H^0(Y, f^*(m+m_0)D+\lceil L \rceil)\to
$$
$$
\hspace{3cm} H^0(S, (f^*(m+m_0)D+\lceil L \rceil)\vert_S)\to 0
$$

\emph{Step 6.} On the other hand,

$$
H^0(S, (f^*(m+m_0)D+\lceil L \rceil)\vert_S)\neq 0
$$
for $m\gg 0$ by Shokurov nonvanishing theorem because
$$
  (f^*(m+m_0)D+\lceil L \rceil)\vert_S-(K_S+(B_Y+\lceil L \rceil-L)\vert_S)=(f^*(m-t)D+H)\vert_S
$$
is ample and $(S, (B_Y+\lceil L \rceil-L)\vert_S)$ is klt. Therefore,

$$
H^0(X,f^*(m+m_0)D)=H^0(Y, f^*(m+m_0)D+\lceil L \rceil)\neq 0
$$
for $m\gg 0$ because $\lceil L \rceil$ is effective and exceptional over $X$. This implies that 
$S$ is not a component of the fixed part of $f^*(m+m_0)D+\lceil L \rceil$
nor a component of the fixed part of $f^*(m+m_0)D$ for any $m\gg 0$. 

\emph{Step 7.} Note that $\Bs|mcD|\subseteq \Bs|cD|$ for any $m,c>0$. 
Now assume that $b$ is a prime number. Then, the results of Steps 1-6 show that 
$\Bs|b^{n}D|$ 
is strictly smaller than $\Bs|bD|$ for some $n>0$. Repeating this finitely 
many times we deduce that $\Bs|b^{n}D|$ is empty for some $n>0$.
Now pick another large prime number $b'$ and pick $n'>0$ so that 
$\Bs|b'^{n'}D|$ is empty. Any sufficiently large number $m$ can written as 
$m=\alpha b^{n}+\beta b'^{n'}$ for certain integers $\alpha,\beta\ge 0$.
Therefore, $mD$ is free for any $m\gg 0$.   
\end{proof}

\begin{rem}
An important corollary of the base point free theorem is that the abundance conjecture 
holds for klt pairs of general type. More precisely, let $(X,B)$ be a projective klt 
pair with $B$ rational and $K_X+B$ nef and big. Pick an integer $I>1$ such that 
$D:=I(K_X+B)$ is Cartier. Then, we can write 
$$
D=K_X+B+(I-1)(K_X+B)
$$
hence we can apply the base point free theorem to deduce that $mI(K_X+B)$ is base point free 
for some $m>0$. The same arguments also apply in the relative situation. When $B$ has 
real coefficients, one can reduce abundance to the rational case using methods that 
we will develop later.
\end{rem}

\begin{rem}[Riemann-Roch formula]
Let $D_1,\dots,D_n$ be cartier divisors on a projective variety $X$ of dimension $d$. Then,
$\mathcal{X}(\sum m_iD_i)$ is a polynomial in $m_i$ of degree $\le d$ [\ref{Kollar-Mori}, 1.36]. 
Moreover, one can write
$$
\mathcal{X}(\sum m_iD_i)=\frac{(\sum m_iD_i)^d}{d!}+(\mbox{lower degree terms})
$$
\end{rem}

\begin{rem}\label{rem-Euler}
If two Cartier divisors $D,D'$ on a smooth projective variety satisfy $D\equiv D'$, then
it is well-known that $\mathcal{X}(D)=\mathcal{X}(D')$ [\ref{Kollar-Mori}, Proposition 2.57].
\end{rem}

\begin{rem}[Multiplicity of linear systems]\label{r-multiplicity-linear-systems}
Let $D$ be a Cartier divisor on a smooth variety $X$ with $h^0(X,D)>0$, and $x\in X$ a closed point.
Maybe after replacing $D$ with some $D'\sim D$, we can assume that $x\notin \Supp D$.
Suppose that $\{h_1,\dots,h_n\}$ is a basis of $H^0(X,D)$ over $\C$. 
Under the above assumptions, each $h_i$ is regular at $x$.
Each
element of $h\in H^0(X,D)$ is uniquely written as
$$
h=\sum a_i h_i
$$
where $a_i\in \C$. On the other hand, $h$ can be described as a
formal power series $\Phi_h$ in a set of local parameters $t_1,\dots, t_d$
at $x$.  Now, the multiplicity of $(h)+D\ge 0$ at $x$ is the same as the multiplicity 
of $h$ at $x$ which can measured in the local ring $\mathcal{O}_x$ using the above local 
parameters. Let $m_x$ be the maximal ideal of $\mathcal{O}_x$. 
Then, the multiplicity of $h$ at $x$ is larger than  $l\in \N$ 
iff $h\in m_x^{l+1}$ iff every homogeneous term of $\Phi_h$ of degree $\le l$ vanishes. In particular,
this condition on multiplicity can be translated into
$$
e(l):=\#\{\mbox{monomials of degree $\le l$}\}=\frac{(l+1)^d}{d!}+(\mbox{lower degree terms})
$$
equations on the $a_i$ since $\Phi_h=\sum a_i\Phi_{h_i}$. In particular, if $e(l)<n$, then we know that there is 
at least some effective divisor $D'\sim  D$ having multiplicity $>l$ 
at $x$.
\end{rem}

\begin{proof}(of Shokurov nonvanishing theorem)
We will proceed by induction on dimension. The case $\dim X=1$ is a trivial exercise 
so we assume that $d=\dim X>1$.

\emph{Step 1.} 
We first reduce the problem to the smooth situation. Take a log resolution $f\colon Y\to X$
such that all the divisors involved have simple normal crossings. We can write

$$
f^*(aD+G-(K_X+B))=H+E'
$$
where $H$ is ample and $E'\ge 0$ has sufficiently small coefficients. On the other hand,
we can write
$$
 f^*(aD+G-(K_X+B))=f^*aD+f^*G-f^*(K_X+B)=
$$
$$
f^*aD+f^*G-(K_Y+B_Y-G')=f^*aD+f^*G+G'-(K_Y+B_Y)
$$
where $(Y,B_Y)$ is klt and $G'\ge 0$ is a Cartier divisor exceptional$/X$. Thus

$$
 f^*aD+f^*G+G'-(K_Y+B_Y+E')=H
$$
where $E'$ can be chosen such that $(Y,B_Y+E')$ is klt. So, we can assume that we are in the
smooth situation, i.e. $X$ is smooth, $B,G$ have simple normal crossing singularities, etc.

\emph{Step 2.} 
Now assume that $D\equiv 0$. Then, by the Kawamata-Viehweg vanishing theorem 
$$
 h^i(X,mD+G)=0 ~~~\mbox{and} ~~~h^i(X,G)=0
$$
if $i>0$ which implies that 
$$
 h^0(X,mD+G)=\mathcal{X}(mD+G)=\mathcal{X}(G)=h^0(X,G)\neq 0
$$
by Remark \ref{rem-Euler}. So, we can assume that
$D$ is not numerically zero.

\emph{Step 3.}
 Put $A:=aD+G-(K_X+B)$ which we may assume to be ample by Step 1. 
 Then,

$$
(nD+G-(K_X+B))^d=((n-a)D+A)^d\ge d(n-a)D\cdot A^{d-1}
$$
and since $A$ is ample and $D$ is nef but not numerically zero, $D\cdot A^{d-1}> 0$.
 Let $k\in\N$ such that $kA$ is Cartier.
  By the Riemann-Roch theorem and the Serre vanishing theorem, if $l\gg 0$, then 

$$
 h^0(X,lk(nD+G-(K_X+B)))= \mathcal{X}(lk(nD+G-(K_X+B)))= \hspace{2cm}
$$
$$
\hspace{3cm} \frac{l^dk^d(nD+G-(K_X+B))^d}{d!}+(\mbox{lower degree terms})
$$
where $d=\dim X$ and the latter is a polynomial in $l$.

  On the other hand, if $x\in X\setminus\Supp G$ is a closed point, then, by 
  Remark \ref{r-multiplicity-linear-systems}, 
  to find a divisor
$$
0\le N\sim lk(nD+G-(K_X+B))
$$
of multiplicity $>2lkd$ we need
$$
e(2lkd)=\frac{(2lkd+1)^d}{d!}+(\mbox{lower degree terms})
$$
conditions on $H^0(X,lk(nD+G-(K_X+B)))$. We can be sure that we can find such an $N$ 
for some $l\gg 0$ if we choose $n$ sufficiently large by the above Riemann-Roch formula.
Put $L_n:=\frac{1}{lk}N$. Therefore, 
$$
L_n\sim_{\Q} nD+G-(K_X+B)
$$ 
such that $\mu_x(L_n)> 2d$.
In particular, $K_X+B+L_n$ is not lc at $x$.

 Now by taking a log resolution $g\colon Y\to X$,
for $m\gg 0$ and some $t\in(0,1)$ we can write
$$
 g^*(mD+G)= g^*(K_X+B+tL_n+mD+G-(K_X+B)-tL_n)\sim_\Q
$$
$$
 g^*(K_X+B+tL_n+mD+G-(K_X+B)-tnD-tG+t(K_X+B))=
$$
$$
  g^*(K_X+B+tL_n+(m-tn)D+(1-t)G+(t-1)(K_X+B))=
$$
$$
 g^*(K_X+B+tL_n+(m-tn-(1-t)a)D+(1-t)aD+(1-t)G-(1-t)(K_X+B))=
$$
$$
 g^*(K_X+B+tL_n)+g^*((m-tn-(1-t)a)D+(1-t)A)\sim_\Q
$$
$$
 g^*(K_X+B+tL_n)+H+E'
$$
and perhaps after changing $E'$ and $H$ slightly the latter can be written as 
$$
 K_Y+B_Y+S-G'+H
$$
where $(Y,B_Y)$ is klt, $S$ is reduced and irreducible and not a component of 
$B_Y$ and $G'$, 
$G'\ge 0$ is a Cartier divisor exceptional over $X$,
$H$ is ample (this construction is similar to Step 4 in the proof 
of the base point free theorem). Hence,
$$
  g^*mD+g^*G+G'-S \sim_\Q K_Y+B_Y+H
$$

Now by applying Kawamata-Viehweg vanishing theorem

$$
H^i(Y, g^*mD+g^*G+G'-S)=0
$$
for $i>0$. From this we deduce that

$$
 H^0(S,(g^*mD+g^*G+G')\vert_S)\neq 0 ~\implies~  H^0(Y,g^*mD+g^*G+G')\neq 0
$$

On the other hand, $g^*D\vert_S$ is a nef Cartier divisor on $S$ such that

$$
 g^*mD\vert_S+(g^*G+G')\vert_S-(K_S+B_Y\vert_S)\sim_\Q H\vert_S
$$
is ample. Therefore, we are done by induction.\\
\end{proof}

\begin{rem}
The proof of Shokurov nonvanishing and the base point free theorem already contains 
some of the most important techniques of the LMMP. One can very clearly see the 
importance of vanishing theorems, pairs, etc. 
The nonvanishing theorem will 
reappear later in a different form but the proof is almost the same as the proof just given. 
\end{rem}

\begin{rem}[LMMP for $\Q$-factorial dlt pairs]
Let $(X/Z,B)$ be a $\Q$-factorial dlt pair with $B$ having real coefficients. 
If $K_X+B$ is not nef$/Z$, then there is an extremal ray $R/Z$ such that 
$(K_X+B)\cdot R<0$. Since  $(X/Z,B)$ is $\Q$-factorial dlt, 
we can find a rational $B'$ close to $B$ such that $(X/Z,B')$ is klt and
$(K_X+B')\cdot R<0$. Now the cone theorem gives the contraction $f\colon X\to Y/Z$ of $R$. 

$\bullet$ If $\dim Y<\dim X$, we have a Mori fibre space and we stop. If not,  
$f$ is birational. 

$\bullet$ If $f$ is a divisorial contraction we let $B_Y=f_*B$. 
By the cone theorem again, $Y$ is $\Q$-factorial, 
in particular, $K_Y+B_Y$ is $\R$-Cartier: indeed, if $D$ is any Weil divisor on $Y$, 
and if $D^\sim$ is its birational transform on $X$, then $a(D^\sim+eE)$ is Cartier 
and $a(D^\sim+eE)\cdot R=0$ for some $a>0$ and $e>0$ where $E$ is the exceptional 
divisor of $f$ (note that only one prime divisor is contracted by $f$); 
The cone theorem says that $a(D^\sim+eE)$ is pullback of some Cartier 
divisor on $Y$ which is simply $aD$ hence $D$ is $\Q$-Cartier. 
On the other hand, there is some real number $e'<0$ such that 
$K_X+B+e'E=f^*(K_Y+B_Y)$. In particular, this implies that 
$d(S,X,B)\le d(S,Y,B_Y)$ for any prime divisor $S$ on birational models of 
$X$ where strict inequality holds if the centre of $S$ on $X$ is inside $E$. 
In other words, the singularities of $(Y/Z,B_Y)$ is "better" than 
the singularities of $(X/Z,B)$. In particular, $(Y/Z,B_Y)$  is $\Q$-factorial dlt.
Now, by replacing $(X/Z,B)$ with $(Y/Z,B_Y)$ we can continue the process. 

$\bullet$ Suppose that $f$ is a flipping contraction. First of all, we will see that the 
flip exists for $(X/Z,B)$ iff if it exists for $(X/Z,B')$ (in fact the flip will be 
given by the variety $X^+$). Assume that the flip 
$(X^+/Z,B^+)$ exists. Then, $X^+$ is again $\Q$-factorial: indeed if $D$ is any 
Weil divisor on $X^+$ and if $D^\sim$ is its birational transform on $X$, then 
$a(D^\sim+e(K_X+B'))\cdot R=0$ where $a(D^\sim+e(K_X+B'))$ is Cartier 
for some $a>0$ and $e$. The cone theorem then implies that 
$a(D+e(K_{X^+}+B'^+))$ is Cartier which in turn implies that $D$ is $\Q$-Cartier because 
$K_{X^+}+B'^+$ is $\Q$-Cartier. Moreover, using the negativity lemma (\ref{l-negativity}), 
one can check that 
 $d(S,X,B)\le d(S,Y,B_Y)$ for any prime divisor $S$ on birational models of 
$X$ where strict inequality holds if the centre of $S$ on $X$ is inside the 
exceptional locus of $f$. In particular, $(X^+/Z,B^+)$ is $\Q$-factorial dlt.
Now, we can replace $(X/Z,B)$  with $(X^+/Z,B^+)$ and continue. 

$\bullet$ If the process terminates, we end up with a Mori fibre space or a log minimal model.
\end{rem}

\begin{rem}[LMMP for surfaces]
Since no flip occurs in dimension two, running the LMMP starting with a $\Q$-factorial 
dlt pair $(X/Z,B)$ of dimension two terminates. Indeed, each birational contraction 
contracts a curve and this drops the 
Picard number $\rho$. So, the process must stop at some point. We end up with a 
Mori fibre space or a log minimal model. 
\end{rem}

\clearpage
~ \vspace{2cm}
\section{\textbf{D-flips and local finite generation}}
\vspace{1cm}

We have already mentioned the finite generation problem for log canonical divisors 
(Conjecture \ref{conj-fg}). It is the connection of this conjecture with existence 
of log flips that makes it a lot more interesting. More generally we have the 
following problem.

\begin{prob}[Finite generation]\label{f.g.}
Let $f\colon X\to Z$ be a contraction of normal varieties and $D$ a $\Q$-divisor on $X$.
We associate
$$
 \mathcal{R}(X/Z,D):=\bigoplus_{m\in \Z_{\ge 0}} f_*\mathcal{O}_X(\lfloor mD\rfloor)
$$
to $D$ which is a $\mathcal{O}_Z$-algebra. When is $\mathcal{R}(X/Z,D)$ a finitely generated
$\mathcal{O}_Z$-algebra? Here by finite generation we mean that for each $P\in Z$,
there is an open affine set $U\subseteq Z$ containing $P$ such that $\mathcal{R}(X/Z,D)(U)$ is a
finitely generated $\mathcal{O}_Z(U)$-algebra. When $Z=pt.$, we usually drop $Z$. 
When $Z$ is affine we let ${R}(X/Z,D):=\mathcal{R}(X/Z,D)(Z)$.
\end{prob}

\begin{exe}[Truncation principle]\label{exe-truncation-principle}
Prove that  $\mathcal{R}(X/Z,D)$ is a finitely generated $\mathcal{O}_Z$-algebra iff 
$\mathcal{R}(X/Z,ID)$ is a finitely generated $\mathcal{O}_Z$-algebra 
for some $I\in \N$ [use the fact that for each open affine subset $U\subseteq Z$, 
$\mathcal{R}(X/Z,D)(U)$ is an integral domain]. More generally, prove that if 
$R=\oplus_{m\ge 0}R_m$ is a Noetherian graded integral domain, then $R$ is a finitely generated 
$R_0$-algebra iff $\oplus_{m\ge 0}R_{mI}$ is a finitely generated $R_0$-algebra.
\end{exe}

\begin{exe}
 Let $f\colon X'\to X$ be a contraction of normal varieties and $D'$ a divisor on $X'$ such that
 $D'=f^*D$ for a Cartier divisor $D$ on $X$. Prove that $H^0(X',mD')=H^0(X,mD)$ for
 any $m\in \N$. In particular, $\mathcal{R}(X/Z,D)$ is a finitely generated $\mathcal{O}_Z$-algebra 
 iff  $\mathcal{R}(X'/Z,D')$ is a finitely generated $\mathcal{O}_Z$-algebra. 
\end{exe}

Recall the definition of free and ample divisors in the relative situation from 
Definition \ref{d-free-ample-divisor/Z}.

\begin{exe}\label{exe-fg-ample}*
Let $f\colon X\to Z$ be a contraction of normal varieties and $D$ a divisor on $X$
which is ample over $Z$. Prove that $\mathcal{R}(X/Z,D)$
is a finitely generated $\mathcal{O}_Z$-algebra (Hint: reduce to the case: $Z$ is affine, 
$X=\PP_Z^n$, $D=\mathcal{O}_X(1)$).
\end{exe}

\begin{thm}[Zariski]\label{t-fg-Zariski}
Let $f\colon X\to Z$ be a contraction of normal varieties and $D$ a Cartier divisor on
$X$ which is free over $Z$. Then, $\mathcal{R}(X/Z,D)$ is a finitely generated
$\mathcal{O}_Z$-algebra.
\end{thm}
\begin{proof}
By shrinking $Z$ we may assume that $D$ is a free divisor on $X$ and $Z$ is affine.
Now $D$ defines a contraction $\phi_D\colon X\to Z'$ such that $f$ is factored
 as $X\toover{\phi_D} Z'\to Z$ and there is a Cartier divisor $H$ on $Z'$ which is
ample over $Z$ and such that $D=f^*H$. In particular, this implies that
$$
 \bigoplus_{m\in \Z_{\ge 0}} H^0(X,mD)= \bigoplus_{m\in \Z_{\ge 0}} H^0(Z',mH)
$$
and so we are done.\\
\end{proof}

\begin{defn}[Log canonical ring]
Let $(X,B)$ be a pair and $f\colon X\to Z$ a contraction of normal varieties. The log canonical
ring of this pair over $Z$ is defined as

$$
 \mathcal{R}(X/Z,K_X+B)=\bigoplus_{m\in \Z_{\ge 0}} f_*\mathcal{O}_X(\lfloor m(K_X+B)\rfloor)
$$

\end{defn}

\begin{rem}
 Let $(X,B)$ be a projective klt pair where $B$ is rational. Then, the log canonical ring
 is a finitely generated $\C$-algebra. This is a major recent result of Birkar-Cascini-Hacon-McKernan 
 [\ref{BCHM}]. We will discuss the proof later in the course. But first we need to deal 
 with finite generation in the local situation.
\end{rem}

We now give a more general definition of a flip.

\begin{defn}[D-flip]\label{d-D-flip}
 Let $f\colon X\to Y/Z$ be the contraction of an extremal ray of small type where $X,Y$ are normal. 
 Let $D$ be an $\R$-Cartier divisor on $X$ which is numerically negative$/Y$, i.e. 
 $-D$ is ample$/Y$. We say that the flip of $D$ (or the $D$-flip) exists if there exists a  diagram
$$
\xymatrix{ X\ar[rd]_{f}& \dashrightarrow & X^+\ar[ld]^{f^+} \\
&Y&}
$$

such that

\begin{itemize}
\item $X^+$ is a normal variety, projective$/Z$,
\item $f^+$ is a small projective birational contraction$/Z$,
\item $D^+$, the birational transform of $D$, is ample over $Y$.
 \end{itemize}
\end{defn}

We will shortly see that if a $D$-flip exists, then it is unique.

\begin{thm}\label{t-flip-and-fg}
Under the notation of Definition \ref{d-D-flip}. If $D$ is a $\Q$-divisor, then the $D$-flip exists iff $\mathcal{R}(X/Z,D)$ is a finitely generated $\mathcal{O}_Z$-algebra;
\end{thm}
\begin{proof}
 Suppose that a $D$-flip exists. Since $X$ and $X^+$ are isomorphic in codimension one, 
$$
\mathcal{R}(X/Z,D)\simeq \mathcal{R}(X^+/Z,D^+)
$$
By Exercise \ref{exe-fg-ample}, $\mathcal{R}(X^+/Z,D^+)$ is a finitely generated 
$\mathcal{O}_Z$-algebra hence $\mathcal{R}(X/Z,D)$ is also finitely generated. 
By Remark \ref{r-truncation-proj},  for any $I\in \N$, 
$$
\Proj \mathcal{R}(X/Z,ID^+)\simeq \Proj \mathcal{R}(X/Z,D)
$$
Since $D^+$ is ample over $Z$, if we take $I$ sufficiently divisible, then 
$$
X^+\simeq \Proj \mathcal{R}(X/Z,ID^+)
$$ 
as can be checked locally over $Z$. Therefore, $f^+\colon X^+\to Z$ is uniquely determined
by the algebra  $\mathcal{R}(X/Z,D)$.

Conversely, assume that $\mathcal{R}(X/Z,D)$ is a finitely generated 
$\mathcal{O}_Z$-algebra. By replacing $D$ with a positive multiple, we 
can assume that $\mathcal{R}(X/Z,D)$ is generated by elements of degree one. 
Now, put 
$$
X^+:=\Proj \mathcal{R}(X/Z,ID)
$$ 
and let $f^+\colon X^+\to Z$ 
be the natural morphism and $\mathcal{O}_{X^+}(1)$ the invertible sheaf 
given by $\Proj \mathcal{R}(X/Z,ID)$ (cf. [\ref{Hart}, Chapter II, Proposition 7.10]). 
Moreover, for any $m\gg 0$, 
$$
f_*^+\mathcal{O}_{X^+}(m)=f_*\mathcal{O}_X(mID)
$$
(cf. [\ref{Hart}, Chapter II, Exercise 5.9]). By Remark \ref{r-truncation-proj}, 
$X^+$ is a normal variety. 

We show that $f^+$ does not contract any divisors. Assume otherwise and let 
$E$ be a prime divisor contracted by $f^+$. Let $H$ be the Cartier divisor 
corresponding to $\mathcal{O}_{X^+}(1)$. The problem is local so we may assume 
$Z$ to be affine.
Since $H$ is ample over $Z$, $E$ is not in the
  base locus of $|mH+E|$ for large $m\in\N$. Hence,
  $$
  \mathcal{O}_{X^+}(mH)\subsetneq \mathcal{O}_{X^+}(mH+E)
  $$ 
  otherwise $E$ would be in the
  base locus of $|mH+E|$ which is not possible.
  Thus, $f^+_*\mathcal{O}_{X^+}(mH)\subsetneq f^+_*\mathcal{O}_{X^+}(mH+E)$
  for large $m\in\N$. On the other hand, since $E$ is exceptional
  $$
  f^+_*\mathcal{O}_{X^+}(mH+E)\subseteq \mathcal{O}_Z(mH_Z)
  $$ 
  where $H_Z=f_*^+H$. But $f_*^+H=f_*D=:D_Z$ because $X$ and $X^+$ are isomorphic 
  over the smooth points of $Z$. So, $\mathcal{O}_Z(mH_Z)=\mathcal{O}_Z(mD_Z)$ 
  and 
  $$
  f^+_*\mathcal{O}_{X^+}(mH+E)\subseteq \mathcal{O}_Z(mD_Z)=f^+_*\mathcal{O}_{X^+}(mH)
  $$ 
    This is a  contradiction.
  So $f^+$ is small and, in particular, $H$ is the birational transform of $D$.
\end{proof}

\begin{rem}\label{r-truncation-proj}
Let $Z$ be a normal variety and $\mathcal{R}$ a finitely generated graded $\mathcal{O}_Z$-algebra 
with the degree $0$ piece $\mathcal{R}_0=\mathcal{O}_Z$. 
Let $I\in \N$, and let  $\mathcal{R}^{[I]}$ be the subalgebra of $\mathcal{R}$ consisting 
summands of degree divisible by $I$, i.e. the piece of degree $0$, $I$, $2I$, $\dots$. 
The injection  $\mathcal{R}^{[I]}\to \mathcal{R}$ induces a rational map 
$$
\phi\colon \Proj \mathcal{R} \bir \Proj \mathcal{R}^{[I]}
$$
over $Z$.
By replacing $Z$ with an open affine subset, from now on we assume that $Z$ is affine, 
and instead of the above sheaves we consider the corresponding algebras 
$R=\mathcal{R}(Z)$ and $R^{[I]}=\mathcal{R}^{[I]}(Z)$. 
Now, $\phi$ is actually a morphism. Indeed, if $P\in \Proj R$, 
then $P\cap R^{[I]}\in \Proj R^{[I]}$ because $P\cap R^{[I]}$ does not contain all 
elements of  $R^{[I]}$ of positive degree since $\alpha\in R$ implies that $\alpha^I\in R^{[I]}$.

Moreover, we show that $\phi$ is locally an isomorphism. If $\alpha\in R^{[I]}$ has positive 
degree $nI$, 
the induced localised map $R^{[I]}_{(\alpha)}\to R_{(\alpha)}$ is again injective. 
It is actually, also surjective. Indeed, let $\frac{\beta}{\alpha^r}$ be an element of 
$R_{(\alpha)}$. Then, be definition, $\deg \beta=\deg \alpha^r=rnI$. So, 
$\beta \in R^{[I]}_{(\alpha)}$ hence $\frac{\beta}{\alpha^r}$ is inside $R^{[I]}_{(\alpha)}$.
Now, let $\alpha_1,\dots,\alpha_l$ be elements of $R^{[I]}_{(\alpha)}$ which 
generate  $R^{[I]}_{(\alpha)}$ as an $R$-algebra where $R=\mathcal{O}_Z(Z)$.

Let $\mathcal{I}$ is the ideal of $R$ consisting of all elements of positive degree,  
$\mathcal{I}^{[I]}:=\mathcal{I}\cap R^{[I]}$, and $\mathcal{J}$ the ideal of $R$ 
generated by the $\alpha_1,\dots,\alpha_l$. Then, 
$\mathcal{I}=\sqrt{\mathcal{I}^{[I]}R}\subseteq \sqrt \mathcal{J}\subseteq \mathcal{I}$.
So, the principal open sets $D_+(\alpha_i)\subseteq \Proj R$ and 
$D_+^{[I]}(\alpha_i)\subseteq \Proj R^{[I]}$ defined by the $\alpha_i$ cover 
$\Proj R$ and $\Proj R^{[I]}$. Now, the isomorphisms  
$R^{[I]}_{(\alpha_i)}\to R_{(\alpha_i)}$ imply that 
$$
\phi\colon D_+(\alpha_i)\to D_+^{[I]}(\alpha_i)\
$$ 
are isomorphisms hence $\phi$ itself is an isomorphism. 

Let $\mathcal{S}$ be the graded $\mathcal{O}_Z$-algebra whose degree $n$ 
summand is the degree $nI$ summand of $\mathcal{R}^{[I]}$, and multiplication 
in $\mathcal{S}$ is the one induced by $\mathcal{R}^{[I]}$. Then, one can 
see that $\Proj \mathcal{S}$ is naturally isomorphic to $\Proj \mathcal{R}^{[I]}$ 
as schemes over $Z$.

If $\mathcal{R}=\mathcal{R}(X/Z,D)$ where $D$ is a $\Q$-divisor on some normal 
variety $X$ projective over $Z$, then $\mathcal{S}=\mathcal{R}(X/Z,ID)$. Moreover, 
in this case, $\Proj \mathcal{R}(X/Z,D)$ is a normal variety: we assume $Z$ is affine 
and that $D\ge 0$; 
note that since ${R}(X/Z,D)$ is an integral domain, $\Proj \mathcal{R}(X/Z,D)$ is an 
integral scheme with function field ${R}(X/Z,D)_{(0)}$. If $\alpha$ is any homogeneous 
element of positive degree $l$, then we show that ${R}(X/Z,D)_{(\alpha)}$ is integrally 
closed in ${R}(X/Z,D)_{(0)}$. Assume that $\frac{\beta}{\gamma}\in {R}(X/Z,D)_{(0)}$ 
satisfying an equation 
$$
(\frac{\beta}{\gamma})^n+\frac{\lambda_1}{\alpha^{r_1}}(\frac{\beta}{\gamma})^{n-1}+\cdots+
\frac{\lambda_n}{\alpha^{r_n}}=0
$$
where $\frac{\lambda_i}{\alpha^{r_i}}\in {R}(X/Z,D)_{(\alpha)}$. Let $r=\max r_i$. 
Multiplying the equation by $\alpha^{nr}$, and replacing 
$\frac{\beta\alpha^r}{\gamma}$ by $\theta$ and replacing $\lambda_i\alpha^{ir-r_i}$ by 
$\rho_i$ we get an equation 
$$
\theta^n+\rho_1\theta^{n-1}+\cdots+\rho_n=0
$$
and it enough to prove that $\theta\in {R}(X/Z,D)$ and that it has degree $rl$, 
i.e. $\theta\in H^0(X,rl D)$. Let $P$ be a prime divisor on $X$. Then, from the 
equation we can estimate:
$$
n\mu_P(\theta)=\mu_P(\theta^n)\ge \min\{ \mu_P(\rho_i)+(n-i)\mu_P(\theta)\}
$$
where $\mu$ stand for multiplicity. If the minimum is attained at 
$\mu_P(\rho_j)+(n-j)\mu_P(\theta)$, then since $\rho_i\in H^0(X,rljD)$,
$\mu_P(\rho_j)\ge rlj\mu_PD$. Thus, 
$$
j\mu_P(\theta)\ge lrj\mu_PD
$$
Therefore, $\theta\in H^0(X,rl D)$ and we are done.
\end{rem}

\begin{defn}[Log canonical model]
 Let $(X/Z,B)$, $(Y/Z,B_Y)$ be lc pairs and $f\colon X\bir Y/Z$ a birational map whose inverse does not
 contract any divisors such that $B_Y=f_*B$. We say that $(Y/Z,B_Y)$ is 
 a \emph{log canonical model} for $(X/Z,B)$
 if $K_Y+B_Y$ is ample$/Z$ and if $d(E,X,B)\le d(E,Y,B_Y)$ for any prime divisor 
 $E$ on $X$ contracted by $f$.
\end{defn}

\begin{thm}
Let $(X/Z,B)$ be a lc pair where $B$ is rational and $\kappa(K_X+B)=\dim X$.
Then, $(X/Z,B)$ has a log canonical model
iff $\mathcal{R}(X/Z,K_X+B)$ is a finitely generated $\mathcal{O}_Z$-algebra. 
Moreover, in this case the log
canonical model is given by $\Proj \mathcal{R}(X,K_X+B)$.
\end{thm}
\begin{proof}
Suppose that $(X/Z,B)$ has a lc model $(Y/Z,B_Y)$. By definition,
$$
\mathcal{R}(X/Z, K_X+B)=\mathcal{R}(Y/Z, K_Y+Y)
$$
and since $K_Y+B_Y$ is ample$/Z$, the $\mathcal{R}(Y/Z, K_Y+Y)$ is a 
finitely generated $\mathcal{O}_Z$-algebra.

Conversely, assume that  $\mathcal{R}(X/Z,K_X+B)$ is a finitely generated 
$\mathcal{O}_Z$-algebra. We localise the problem so assume that $Z$ is affine.
Take $I\in\N$ such that ${R}(X,I(K_X+B))$ is generated by elements of degree 
one. Let $f\colon W\to X$ be a log resolution such that $f^*I(K_X+B)=M+F$ where $M$ is free
and $F$ is the fixed part. Since ${R}(X,I(K_X+B))$ is generated by elements of degree 
one, it is easy to check that 
$$
\Mov mf^*I(K_X+B)=mM ~~~\mbox{and}~~~ \Fix mf^*I(K_X+B)=mF
$$
for every $m>0$. In particular, 
$$
{R}(W/Z,M)\simeq {R}(X/Z,I(K_X+B))
$$
Since $M$ is free it defines a contraction $W\to Y:=\Proj {R}(W/Z,M)$ over $Z$.  
Arguments similar to those in the proof of Theorem \ref{t-flip-and-fg} 
show that the inverse of the induced birational map $X\bir Y$ does not 
contract any divisors. Finally, observe that  
$$
\Proj \mathcal{R}(X/Z,K_X+B)=\Proj \mathcal{R}(X/Z,I(K_X+B))
$$
\end{proof}

\clearpage
~ \vspace{2cm}
\section{\textbf{Pl flips and extension theorems}}
\vspace{1cm}

\begin{defn}
Let $(X/Z,B)$ be a $\Q$-factorial dlt pair, and $f\colon X\to Y/Z$ a 
$K_X+B$-negative flipping contraction. We say that $f$ is a \emph{pl flipping 
contraction} if there is a component $S$ of $\rddown{B}$ such that 
$S$ is numerically negative$/Y$, i.e. $-S$ is ample over $Y$. If the 
flip exists, we call it a pl flip.  Pl stands for pre-limiting.
\end{defn}

One of the most important insights of Shokurov was that to prove the existence 
of log flips for klt pairs (or $\Q$-factorial dlt pairs), it is enough to 
verify the existence of pl flips, up to certain standard assumptions in lower 
dimensions. We will discuss this in detail in later lecture. The crucial advantage 
of pl flipping contractions is that we have at least one component 
$S$ of $\rddown{B}$ which is numerically negative$/Y$ and this allows us 
to do induction by reducing finite generation of to  $\mathcal{R}(X/Z,K_X+B)$
to a subtle finite generation property on $S$ which is of smaller 
dimension hence induction.

\begin{lem}\label{l-fg-linear-change}
Let $(X/Z,B)$ be a $\Q$-factorial dlt pair, $f\colon X\to Y/Z$ a 
$K_X+B$-negative flipping contraction, and $D$ 
an $\R$-Cartier divisor on $X$ which is numerically negative$/Y$. Then, the 
$K_X+B$-flip exists iff the $D$-flip exists. 
\end{lem}
\begin{proof}
First note that we can choose a rational boundary $B'\le B$ close to $B$ 
so that $(X/Z,B')$ is $\Q$-factorial dlt pair and $f$ is a 
$K_X+B'$-negative flipping contraction. We may replace $B$ by $B'$ 
and assume that $B$ is rational. Moreover, there exist rational divisors 
$D_1,\dots,D_r$ sufficiently close to $D$ (having the same support as $D$), 
and positive numbers $\alpha_i$ such that $D=\sum \alpha_iD_i$ and 
$\sum \alpha_i=1$, and that $-D_i$ is ample$/Y$.  
By the cone theorem, for each $i$, there is $a_i>0$ such that 
$K_X+B\sim_\Q a_iD_i/Y$. Therefore, 
$$
\sum \frac{\alpha_i}{a_i}(K_X+B)\sim_\R \sum \alpha_iD_i=D/Y
$$
 which implies that 
$K_X+B\sim_\R aD/Y$ for some positive $a\in \R$.
 
Assume that the $D$-flip exists and given by a diagram $X\to Y\leftarrow X^+$.
By the above arguments, $K_{X^+}+B^+\sim_\R aD^+/Y$ which implies that 
$K_{X^+}+B^+$ is $\R$-Cartier and actually ample$/Y$ hence the diagram also 
gives the $K_X+B$-flip. The converse is proved similarly.
\end{proof}

\begin{cor}
Let $(X/Z,B)$ be a $\Q$-factorial dlt pair, $f\colon X\to Y/Z$ a pl flipping 
contraction and $S$ a component of $\rddown{B}$ which is numerically negative$/Y$.
Then, the $K_X+B$-flip exists iff the $S$-flip exists iff $\mathcal{R}(X/Y,S)$ 
is a finitely generated $\mathcal{O}_Y$-algebra. 
\end{cor}

By replacing $B$ with a rational boundary, we can from now on assume that $B$ is rational.

\begin{rem}[Adjunction for dlt pairs]\label{rem-adjunction-dlt-pairs}
Let $(X/Z,B)$ be a dlt pair and $S$ a component of $\rddown{B}$. 
It is well-known that $S$ is normal (cf. Koll\'ar-Mori [\ref{Kollar-Mori}, corollary 5.52]). 
More generally, the lc centres of 
$(X/Z,B)$ are normal and they are the irreducible components of $S_1\cap\cdots\cap S_r$ 
where $S_1,\cdots, S_r$ are among the irreducible components of $\rddown{B}$.
Moreover, there is a boundary $B_S$ on $S$ such that $(K_X+B)|_S=K_S+B_S$ 
and $(S,B_S)$ is dlt: indeed let $g\colon W \to X$ be a log resolution so that
any exceptional$/X$ prime divisor $E$ on $W$ satisfies $d(E,X,B)>-1$ (this exists 
by definition of dlt pairs). Write $K_W+B_W=g^*(K_X+B)$ and let $T\subset W$ be the birational 
transform of $S$ and $h\colon T\to S$ the induced morphism. Then, by letting $B_T=(B_W-T)|_T$ 
we have 
$$
K_T+B_T=(K_W+B_W)|_T=g^*(K_X+B)|_T=h^*(K_X+B)|_S
$$
If $B_S=h_*B_T$, then $K_S+B_S=(K_X+B)|_S$,  $K_T+B_T=h^*(K_S+B_S)$ and 
the properties of $B_T$ ensure that $(S,B_S)$ is dlt (we also need to use the 
fact that if $S'$ is another component of $\rddown{B}$, then $S\cap S'$ is empty 
or else its irreducible 
components  have dimension $\dim X-2$; see Fujino 
[\ref{Fujino-dlt-pairs}, Proposition 3.9.2] for more information). 
Another interesting fact is that if $B=\sum b_kB_k$, then the coefficients
 of any component of $B_S$ looks like
 
 $$
 \frac{m-1}{m}+\sum \frac{l_kb_k}{m}
 $$
 for certain $m\in\N$ and $l_k\in\N\cup\{0\}$ [\ref{Shokurov-log-flips}, Corollary 3.10].
Finally, note that if the irreducible components of $\rddown{B}$ are disjoint then $(S,B_S)$ is 
klt. 
\end{rem}

Let $(X/Z,B)$ be a $\Q$-factorial dlt pair, $f\colon X\to Y/Z$ a pl flipping 
contraction and $S$ a component of $\rddown{B}$ which is numerically negative$/Y$.
Moreover, since we are interested in the finite generation of $\mathcal{R}(X/Y,S)$ locally 
on $Y$, we can assume that $Z=Y=\Spec R$ is affine. Since $f$ is a contraction of small type, 
$S\sim D\ge 0$ such that $S$ is not in $\Supp D$. Note that $R(X/Y,S)$ is 
a finitely generated $R$-algebra iff $R(X/Y,D)$ is a finitely generated $R$-algebra.
Let $\tau$ be a rational function such that $D+(\tau)=S$. In particular, 
$\tau$ is an element of $R(X/Y,D)$ of degree one.

Now, for each $m$, we have an exact sequence 
$$
0\to \mathcal{O}_X(mD-S) \to \mathcal{O}_X(mD)\to \mathcal{F}_m\to 0
$$ 
where $\mathcal{F}_m$ is supported on $S$. 
Note that since $mD$ is not necessarily Cartier, $\mathcal{F}_m$
may not be isomorphic to $\mathcal{O}_S(mD|_S)$. Anyway, we get exact sequences
$$
0\to H^0(X,\mathcal{O}_X(mD-S)) \to H^0(X,\mathcal{O}_X(mD))\to H^0(X,\mathcal{F}_m)
$$ 
which in turn give an exact sequence  
$$
0\to \bigoplus_{m\ge 0}H^0(X,\mathcal{O}_X(mD-S)) \to 
\bigoplus_{m\ge 0}H^0(X,\mathcal{O}_X(mD))\to^\phi 
R(X/Y,D)|_S\to 0
$$
for a certain algebra $R(X/Y,D)|_S$ on $S$ which is simply the image of $R(X/Y,D)$. 
We will shortly see that $\ker \phi$ is the ideal of $R(X/Y,D)$ generated by the 
element $\tau$.

If $\alpha\in \ker \phi$ is homogeneous of degree $m$, then 
$$
0\le (\alpha)+mD-S=(\alpha)+mD-D-(\tau)=(\alpha\tau^{-1})+(m-1)D
$$
which means that $\beta:=\alpha\tau^{-1}\in R(X/Y,D)$ having degree $m-1$. 
So, $\alpha=\beta\tau$ hence $\ker \phi$ is inside the ideal of $R(X/Y,D)$ 
generated by $\tau$. Since $\tau$ vanishes on $S$ one sees 
that the ideal of $R(X/Y,D)$ generated by $\tau$ is inside $\ker \phi$ hence the claim.
The above arguments give the following important result.

\begin{lem}
$R(X/Y,S)$ is a finitely generated $R$-algebra iff $R(X/Y,D)$ is a finitely generated $R$-algebra iff the restriction $R(X/Y,D)|_S$ 
is a finitely generated $R$-algebra.
\end{lem} 
\begin{proof}
It was mentioned that $R(X/Y,S)$ is 
a finitely generated $R$-algebra iff $R(X/Y,D)$ is a finitely generated $R$-algebra.
If $R(X/Y,D)$ is a finitely generated $R$-algebra then obviously the restriction $R(X/Y,D)|_S$ 
is a finitely generated $R$-algebra. Conversely, if  $R(X/Y,D)|_S$ is generated by the images 
of finitely many homogeneous elements $\alpha_1,\dots,\alpha_r\in R(X/Y,D)$, 
then each homogeneous $\alpha\in R(X/Y,D)$ 
of degree $m$ can be written as $\alpha=\alpha'+\alpha''$ where  $\alpha'$ belongs to the $R$-algebra 
generated by the   $\alpha_1,\dots,\alpha_r$, $\alpha''\in \ker \phi$, and 
$\deg \alpha=\deg \alpha''$. So, 
$\alpha''=\beta\tau$ for some $\beta$ of degree $m-1$. Therefore, $\alpha$ belongs to the 
algebra  generated by $\alpha_1,\dots,\alpha_r$ and $\tau$.\\
\end{proof}

By construction, $ID\sim J(K_X+B)$ for certain $0<I,J\in \Z$ such that $ID$ 
and $J(K_X+B)$ are Cartier. By the truncation principle (Exercise \ref{exe-truncation-principle}), 
$R(X/Y,ID)|_S$ is finitely generated iff $R(X/Y,D)|_S$ is finitely generated. 
Moreover, 
$$
R(X/Y,ID)|_S\subseteq R(S/Y,ID|_S)
$$
Unfortunately, in general, equality does not necessarily hold. If equality holds, 
then we could use the fact that $R(S/Y,ID|_S)$ is finitely generated iff 
$R(S/Y,J(K_X+B)|_S)=R(S/Y,J(K_S+B_S))$ is finitely generated where $K_S+B_S$ is 
given by adjunction. In a sense this is the ideal situation since inductively we 
assume that $R(S/Y,J(K_S+B_S))$ is finitely generated and we would be done. 
Though equality does not hold but we will see that we can essentially get equality 
by going on a high resolution. 

In any case, it is enough to prove that $R(X/Y,J(K_X+B))|_S$ is finitely generated.
Since we can perturb the coefficients of $B$, we can assume that $S=\rddown{B}$. 
Let $g\colon W\to X$ be a log resolution. Write $K_W+B_W=g^*(K_X+B)+E$ where 
$B_W,E\ge 0$ with no common components and $E$ is exceptional$/X$. 
Then,  
$$
R(X/Y,J(K_W+B_W))=R(X/Y,J(K_X+B))
$$ 
If $T\subset W$ is the birational 
transform of $S$ and $h\colon T\to S$ the induced morphism, then 
$$
h_*R(X/Y,J(K_W+B_W))|_T=R(X/Y,J(K_X+B))|_S
$$
hence it is enough to prove that $R(X/Y,J(K_W+B_W))|_T$ is finitely generated.

One of the most useful things that naturally appear in  our situation is that 
we can afford to have ample divisor around. More precisely, since $K_X+B$ is 
negative$/Y$, we can add a little ample divisor to $B$ hence assume that 
$A\le B$ for some general ample $\Q$-divisor $A$. Now, $g^*A\le B_W$ 
and $g^*A$ is nef and big. However, $g^*A\sim_\Q A_W+G_W$ where $A_W$ 
is a general ample $\Q$-divisor and $G_W\ge $ has sufficiently small 
coefficients (Corollary \ref{c-nef-to-ample}). The main point is that we can 
choose $A_W$ and $G_W$ so that $T$ is not a component of $A_W+B_W$. 
Moreover, we can assume that $(W/Z,\Delta_W)$ is plt and $T=\rddown{\Delta_W}$ where 
$$
B_W=B_W-g^*A+g^*A\simeq  B_W-g^*A+A_W+G_W:=\Delta_W
$$
So, by replacing $B_W$ by $\Delta_W$ we will assume that $T=\rddown{B_W}$ 
and that $B_W=A_W+B_W'$ where $A_W$ is ample and $B_W'\ge 0$. Moreover, 
by replacing $W$ with a sufficiently high resolution and putting 
$$
K_T+B_T:=(K_W+B_W)|_T
$$
 we can assume that 
$(T,B_T)$ is "canonical" meaning that the  discrepancy $d(N,T,B_T)\ge 0$ for 
any prime divisor $N$ on birational models of $T$ if $N$ is exceptional$/T$. 
Now we are in a position to apply the  extension theorem of Hacon-McKernan 
[\ref{Hacon-McKernan-II}, 7.1].
That is, up to certain assumptions in lower dimension 
(to be discussed) and maybe after replacing $J$ with some multiple we have: 
$$
R(X/Y,J(K_W+B_W))|_T\simeq R(S/Y,J(K_T+\Theta_T))
$$
for some rational boundary $\Theta_T\le B_T$. More precisely, $\Theta_T$ is calculated 
as follows: for each $m>0$ define 
$$
F_{m!}=\frac{1}{m!}\Fix|m!(K_W+B_W)|_T
$$
where $\Fix|m!(K_W+B_W)|_T$ is the largest divisor on $T$ such that 
$\Fix|m!(K_W+B_W)|_T\le M|_T$ for any $0\le M\sim m!(K_W+B_W)$ provided $T\nsubseteq M$. 
Put $F=\lim F_{m!}$. Then, $\Theta=B_T-B_T\wedge F$ where $B_T\wedge F$ is the largest 
divisor satisfying $B_T\wedge F\le B_T$ and $B_T\wedge F\le F$.

\clearpage
~ \vspace{2cm}
\section{\textbf{Existence of log minimal models and Mori fibre spaces}}
\vspace{1cm}

Though the minimal model conjecture is still widely open but some 
very important cases have already been established, in particular, 
the following theorem of Birkar-Cascini-Hacon-McKernan [\ref{BCHM}] (see also [\ref{Birkar-Paun}]).
We will work out its proof which involves many interesting features
of higher dimensional geometry.

\begin{thm}\label{t-main-BCHM}
Let $(X/Z,B)$ be a klt pair where $B$ is big$/Z$. Then: 

$(1)$  if $K_X+B$ is pseudo-effective$/Z$, then $(X/Z,B)$ has a log minimal model $(Y/Z,B_Y)$. 
Moreover, here abundance holds, that is, $K_Y+B_Y$ is semi-ample$/Z$;  

$(2)$ if $K_X+B$ is not pseudo-effective$/Z$, then  $(X/Z,B)$ has a Mori 
fibre space.
\end{thm}
\begin{proof}
We do induction on $d$. So, assume that the theorem holds in dimension $d-1$.

(1) If $K_X+B\sim_\R M/Z$ for some $M\ge 0$, then by Theorem \ref{t-nonvanishing-to-mmodel}, 
$(X/Z,B)$ has a log minimal model. The semi-ampleness claim follows from the 
base point free theorem and Lemma \ref{l-ray-stability}. In general, 
by Theorem \ref{t-nonvanishing-induction}, we can always find such $M$.

(2) Let $C$ be an ample$/Z$ $\R$-divisor such that $K_X+B+C$ is klt and nef$/Z$. 
By Theorem \ref{t-term}, the LMMP$/Z$ on $K_X+B$ with scaling of $C$ terminates 
with a Mori fibre space. 
\end{proof}

Some immediate corollaries of the theorem are:

\begin{cor}[Finite generation]\label{cor-fg}
 Let $(X/Z,B)$ be a klt pair such that $K_X+B$ is a 
$\Q$-divisor. Then, the log canonical algebra 
$$
\mathcal{R}(X/Z,B):=\bigoplus_{m\ge 0} f_*\mathcal{O}_X(\rddown{m(K_X+B)}) 
$$
is a finitely generated $\mathcal{O}_Z$-algebra where $f$ is the given morphism 
$X\to Z$.
\end{cor}
\begin{proof}
If  $\mathcal{R}(X/Z,B)=\C$, finite generation is 
trivial in this case. So, we can assume that 
$\mathcal{R}(X/Z,B)\neq \C$ which in particular means that $f_*\mathcal{O}_X(\rddown{m(K_X+B)})\neq 0$
for some $m>0$.
Actually, a theorem of Fujino and Mori [\ref{Fujino-Mori}] reduces finite generation to the case 
when $K_X+B$ is big$/Z$. So, we assume that $K_X+B$ is big$/Z$.
Now, by Theorem \ref{t-main-BCHM}, $(X/Z,B)$ has a log minimal model 
$(Y/Z,B_Y)$ such that $K_Y+B_Y$ is semi-ample$/Z$. 
By definition of log minimal models, there is a common resolution $g\colon W\to X$ 
and $h\colon W\to Y$ such that we can write $g^*(K_X+B)=h^*(K_Y+B_Y)+E$ 
where $E\ge 0$ is exceptional$/Y$. If $I>0$ is an integer so that  
$I(K_X+B)$ and $I(K_Y+B_Y)$ are Cartier, then  
$$
\mathcal{R}(X/Z,I(K_X+B))\simeq \mathcal{R}(Y/Z,I(K_Y+B_Y))
$$ 
hence it is enough to have finite generation on $Y$. Since $K_Y+B_Y$ is semi-ample$/Z$, 
finite generation follows from 
 Zariski's Theorem \ref{t-fg-Zariski}.
\end{proof}

\begin{cor}[Log flips]\label{cor-log-flips}
Log flips exist for klt pairs.
\end{cor}
\begin{proof}
Let  $(X/Z,B)$ be a klt pair and $f\colon X\to Z'$ a $K_X+B$-negative 
extremal flipping contraction$/Z$. If $B$ is rational, then the flip exists 
simply because the 
algebra $\mathcal{R}(X/Z,K_X+B)$ is finitely generated and we can use 
Theorem \ref{t-flip-and-fg}. 
If $B$ is not rational we need a different argument. By Theorem \ref{t-main-BCHM}, 
$(X/Z',B)$ has a log minimal model $(Y/Z',B_Y)$ such that $K_Y+B_Y$ is 
semi-ample$/Z'$. Since $f$ does not contract divisors, $Y\to Z'$ also 
does not contract divisors. Since  $K_Y+B_Y$ is semi-ample$/Z'$, 
there is a morphism $\pi\colon Y\to X^+/Z'$ such that $K_Y+B_Y\sim_\R 
\pi^*A/Z'$ for some ample$/Z$ $\R$-divisor $A$ on $X^+$. Now, $X^+\to Z'$ 
gives the flip of $f$. 
\end{proof}

We briefly explain the basic strategy for proving Theorem \ref{t-main-BCHM}. 
We will go through these steps in an order that is pedagogically helpful.

$\bullet$ \emph{From nonvanishing to minimal models:} 
if $K_X+B\sim_\R M/Z$ for some $M\ge 0$, then we construct a 
log minimal model for $(X/Z,B)$ using induction;
The main idea  is that 
the effective divisor $M$ allows us to artificially add the components of 
$M$ to $B$ to create components in $B$ with coefficient one. Perhaps, 
after replacing $(X/Z,B)$ with a log smooth pair through some resolution 
of singularities, we can find $C$ supported on  $M$ such that every component of 
$M$ in $B+C$ has coefficient one. In this case, we first construct a 
log minimal model for $(X/Z,B+C)$. Here if we run an LMMP$/Z$ on $K_X+B+C$, 
we only need pl flips and termination of the program is reduced to lower 
dimension by restricting to components of $\rddown{B+C}$. We end up with 
a log minimal model $(Y/Z,B_Y+C_Y)$.

$\bullet$ \emph{Termination with scaling:} the LMMP we need is a special 
kind of LMMP explained in Definition \ref{d-LMMP-scaling}. The basic 
idea is that we can add a divisor $A$ so that $K_X+B+C+A$ is nef$/Z$. 
Now try to decrease $A$ but preserving the nefness of $K_X+B+C+A$. 
This naturally leads to a sequence of divisorial contractions and flip.
Termination of such a sequence is reduced to lower dimensions via 
$\rddown{B+C}$. Anyway, we construct a log minimal model $(Y/Z,B_Y+C_Y)$ 
for $(X/Z,B+C)$. Since $K_Y+B_Y+C_Y$ is nef$/Z$, we can
 try to decrease $C_Y$ but preserving the nefness. 
This again leads to a sequence of divisorial contractions and flip 
and termination in lower dimensions.
 So, we need termination in lower dimension that appear 
in LMMP with scaling.

$\bullet$ \emph{Finiteness of models:} the LMMP with scaling (in lower dimension) produces 
a sequence of birational klt pairs $(X_i'/Z,B_i'+\lambda_iC_i')$ such that 
$K_{X_i'}+B_i'+\lambda_iC_i'$ is nef$/Z$. An interesting phenomenon  is that 
there can be only finitely many $X_i$. This implies that there cannot be an infinite 
sequence of log flips in the LMMP with scaling hence termination.
 
$\bullet$ \emph{Nonvanishing:} the above steps rely on the fact that $K_X+B\sim_\R M/Z$ 
for some $M\ge 0$. If $K_X+B$ is pseudo-effective$/Z$, we will prove that 
in this case in fact there always is such an $M$. 
The proof of this fact is quite similar to Shokurov nonvanishing with some 
new ingredients. The proof again relies on creating components in $B$ with 
coefficients one and doing induction as usual.

$\bullet$ \emph{Pl flips:} as mentioned above we need pl flips to run the LMMP 
we need. We will show that existence of pl flips follows from the above 
statement in lower dimension hence completing the induction process.

\begin{rem}[Special termination]\label{rem-special-termination}
The following result of Shokurov is an important element of induction arguments in 
the LMMP. Let $(X/Z,B)$ be a dlt pair. Suppose that there is a sequence
$X_i\bir X_{i+1}/Z_i$ of $K_X+B$-negative log flips starting with $X_1=X$. 
For a divisor $M$ on $X$, we denote by  $M_i$ its birational transform on $X_i$.
Let $S$ be a component of $\rddown{B}$   
and $S_i\bir S_{i+1}$ the induced birational map. It is not difficult to see that 
$S_{i+1}\bir S_{i}$ does not contract any divisors when $i\gg 0$ (cf. Fujino 
[\ref{Fujino-dlt-pairs}]).  Obviously, we get induced birational maps 
$S_i\to T_i \leftarrow S_{i+1}$ where $T_i$ is the
 normalisation of the image of $S_i$ on $Z_i$.  By adjunction
$(K_{X_i}+B_i)\vert_{S_i}\sim_{\R} K_{S_i}+B_{S_i}$. The map 
$S_i\bir S_{i+1}/T_i$ is not necessarily a $K_{S_i}+B_{S_i}$-flip 
but 
one can decompose it into a sequence of log flips if we know 
 termination in dimension $<\dim X$. In that case,  
the sequence $X_i\bir X_{i+1}/Z_i$ terminates near $\rddown{B}$, that is, 
$\rddown{B_i}$ does not intersect the extremal ray defining the contraction 
$X_i\to Z_i$. 

One important thing to remember is that the termination we need in dimension 
$<\dim X$ depends on the kind of termination we try to prove on $X$. For example, 
if the sequence $X_i\bir X_{i+1}/Z_i$ is a sequence of log flips with scaling 
of some divisor $C$, then we only  need termination with scaling in dimension 
$<\dim X$. Moreover, if $B=B_1$ is big, usually this condition is inherited by 
the $B_{S_i}$ (at least this is the case when we apply special termination 
later on). 
\end{rem}

\begin{rem}\label{rem-perturbation}
Let $(X/Z,B)$ be a klt pair such that $B$ is big$/Z$.
The bigness ensures that $B\sim_\R G+A/Z$ for some $\R$-divisor 
$G\ge 0$ and some ample$/Z$ $\R$-divisor $H$. We can write 
$$
K_X+B= K_X+(1-\epsilon) B+\epsilon B\sim_\R K_X+(1-\epsilon) B+\epsilon (G+H)/Z
$$
and if $\epsilon>0$ is small enough then $(X/Z,B'=(1-\epsilon) B+\epsilon G)$ is klt.
Take  $A\sim_\R \epsilon H$ general so that $(X/Z,B'+A)$ is klt. In many places 
we can replace $(X/Z,B)$ with $(X/Z,B'+A)$ hence assume $B\ge A$ for some ample $\R$-divisor 
$A$. We use this repeatedly in the sequel.\\ 
\end{rem}

\clearpage
~ \vspace{2cm}
\section{\textbf{From nonvanishing to log minimal models}}\label{s-from-nonvanishing}
\vspace{1cm}

\begin{thm}\label{t-nonvanishing-to-mmodel}
Assume Theorem \ref{t-main-BCHM} in dimension $d-1$. 
Let $(X/Z,B)$ be a klt pair of dimension $d$ such that $B$ is big$/Z$, and  such that
$K_X+B\sim_\R M/Z$ for some $M\ge 0$. Then, $(X/Z,B)$ has a log minimal model.
\end{thm}
\begin{proof}
We closely follow the proof of [\ref{Birkar-mmodel}, Theorem 1.3]. 
 
\emph{Step 1.}  Since $B$ is big/$Z$, 
by perturbing the coefficients we can assume that it has a general ample$/Z$ component 
which is not a component of $M$  (see Remark \ref{rem-perturbation}).  
 By taking a log resolution we can further 
assume that $(X/Z,B+M)$ is log smooth.
 To construct log minimal models in this situation we need 
to pass to a more general setting.

Let $\mathfrak{W}$ be the set of triples $(X/Z,B,M)$ which satisfy\\\\
(1) $(X/Z,B)$ is dlt of dimension $d$, $(X/Z,B+M)$ is log smooth, and $K_X+B\sim_\R M/Z$,\\
(2) $(X/Z,B)$ does not have a log minimal model,\\
(3) $B$ has a component which is ample$/Z$ but it is not a component of $\rddown{B}$ nor a component 
of $M$.\\

Note that, since $B$ has a general ample component, we have the following crucial property:

$\bullet$ for each component $S$ of $\rddown{B}$, we can write $K_X+B\sim_\R K_X+\Delta/Z$ with 
$(X/Z,\Delta)$ being dlt, $S=\rddown{\Delta}$, and $(S/Z,\Delta_S)$ klt 
where we define $\Delta_S$ by adjunction: $K_S+\Delta_S=(K_X+\Delta)|_S$ (see Remark 
\ref{rem-adjunction-dlt-pairs}). 

The property remains true even after divisorial 
contractions and log flips with respect to $K_X+B$.

Obviously, it is enough to prove that $\mathfrak W$ is empty. Assume otherwise and
choose  $(X/Z,B,M)\in\mathfrak W$ with minimal $\theta(X/Z,B,M)$ where 
$$
\theta(X/Z,B,M):=\#\{i~|~m_i\neq 0 ~~\mbox{and}~~ b_i\neq 1\}
$$
where $B=\sum b_iD_i$ and $M=\sum m_iD_i$ with the $D_i$ being distinct prime divisors.

If $\theta(X/Z,B,M)=0$, then every component of $M$ is a component of $\rddown{B}$. 
This is the ideal situation for induction. If $M=0$, we already have a log minimal model.
If $M\neq 0$, we use special termination as follows (see Remark \ref{rem-special-termination}).
Run the LMMP/$Z$ on $K_X+B$ with scaling of a suitable ample$/Z$ $\R$-divisor $H$. 
Let $\lambda_i$ be the numbers appearing in the LMMP: by definition, we have a 
sequence $X_i\bir X_{i+1}/Z_i$ of divisorial contractions and log flips$/Z$ 
such that $K_{X_i}+B_i+\lambda_iH_i$ is nef$/Z$ and $\sim_\R 0/Z_i$, and 
$X_1=X$. Here $X_i\to Z_i$ is a $K_{X_i}+B_i$-negative extremal contraction: 
if it is of flipping type, then $X_i^+=X_{i+1}$ gives the flip of $X_i\to Z_i$ 
but if it is of divisorial type, then $X_{i+1}=Z_i$. 
 
There are a few important points concerning the above LMMP. 
First of all, to run the program we only need pl flips. Secondly, since divsorial 
contractions drop the Picard number, we can assume that $X_i\bir X_{i+1}/Z_i$ 
is a log flip for any $i\ge j$ for some $j$. Moreover, by 
Remark \ref{rem-special-termination}, we can also assume that the induced 
map $S_i\bir S_{i+1}$ is an isomorphism in codimension one when $i\ge j$.
Thirdly, since $B$ has a general ample component, 
for each component $S$ of $\rddown{B}$, we can write $K_X+B\sim_\R K_X+\Delta/Z$ such that 
$(X_j/Z,\Delta_j)$ is dlt, $S_j=\rddown{\Delta_j}$, $(S_j/Z,\Delta_{S_j})$ is klt, and 
$\Delta_{S_j}$ is big$/T$ where $T$ is the image of $S$ in $Z$, and  
 we define $\Delta_{S_j}$ by adjunction: $K_{S_j}+\Delta_{S_j}=(K_{X_j}+\Delta)|_{S_j}$ (see Remark 
\ref{rem-adjunction-dlt-pairs}).  By  Remark \ref{rem-special-termination}, 
the LMMP induces an LMMP with scaling on $K_{S_j}+\Delta_{S_j}$ which we may assume to terminate 
by induction using Theorem \ref{t-s-term}.
So, the LMMP terminates near $\rddown{B}$. 
Finally, by the  $\theta(X/Z,B,M)=0$ assumption, the LMMP terminates everywhere.\\

\emph{Step 2.} We may then assume that $\theta(X/Z,B,M)> 0$. Notation: for an $\R$-divisor $D=\sum d_iD_i$ we define $D^{\le 1}:=\sum d_i'D_i$ in which  $d_i'=\min\{d_i,1\}$. Now put
$$
\alpha:=\min\{t>0~|~~\rddown{(B+tM)^{\le 1}}\neq \rddown{B}~\}
$$

In particular, $(B+\alpha M)^{\le 1}=B+C$ for some $C\ge 0$ supported in $\Supp M$, and $\alpha M=C+M'$ 
where $M'$ is supported in $\Supp \rddown{B}$. Thus, outside $\Supp \rddown{B}$ 
we have $C=\alpha M$. The pair $(X/Z,B+C)$ is obviously log smooth and
$(X/Z,B+C,M+C)$ is a triple which satisfies (1) and (3) above. By construction
$$
\theta(X/Z,B+C,M+C)<\theta(X/Z,B,M)
$$
 so $(X/Z,B+C,M+C)\notin \mathfrak W$. Therefore, $(X/Z,B+C)$ has a log minimal model,
say $(Y/Z,B_Y+C_Y)$. By definition, $K_Y+B_Y+C_Y$ is nef/$Z$.\\

\emph{Step 3.} Now run the LMMP$/Z$ on $K_Y+B_Y$ with scaling of $C_Y$. Note that we only need pl flips here 
because every extremal ray contracted in the process would have negative intersection with some component of $\rddown{B}$ by the properties of $C$ mentioned in Step 2. 
By using special termination as in Step 1, after finitely many steps, 
$\Supp \rddown{B}$ does not intersect the extremal rays
contracted by the LMMP hence the LMMP terminates on a model $Y'$ on which $K_{Y'}+B_{Y'}$ is nef/$Z$. 
Unfortunately, $(Y'/Z,B_{Y'})$ may not be a log minimal model of $(X/Z,B)$ because the condition on singularities 
in Definition \ref{d-mmodel} may not be satisfied, i.e. the singularities of 
$(Y'/Z,B_{Y'})$ might be worse than the singularities of $(X/Z,B)$. However, we can tackle this 
problem by a limiting argument that will be presented in the following steps.\\

\emph{Step 4.} Let
$$
\mathcal T=\{t\in [0,1]~|~ (X/Z,B+tC)~~\mbox{has a log minimal model}\}
$$
Since $1\in\mathcal T$, $\mathcal T\neq \emptyset$. Let $t\in\mathcal T\cap (0,1]$ and let 
$(Y_t/Z,B_{Y_t}+tC_{Y_t})$ be any log minimal model of $(X/Z,B+tC)$. Running the LMMP/$Z$ on $K_{Y_t}+B_{Y_t}$ with scaling
of $tC_{Y_t}$ shows that there is $t'\in(0,t)$ sufficiently close to $t$ such that $[t',t]\subset \mathcal{T}$
because the condition on singularities in Definition \ref{d-mmodel} is an open condition 
(strictly speaking, one needs to use the negativity lemma which implies that 
it is enough to compare discrepancies for those prime divisors on $X$ which are 
contracted over $Y_t$). The
LMMP terminates for the same reasons as in Step 1-3 and we note again that the log flips 
required are all pl flips.\\

\emph{Step 5.} Let $\tau=\inf \mathcal T$. If $\tau\in\mathcal{T}$, then by Step 4, $\tau=0$
and so we are done by deriving a contradiction. Thus, we may assume that $\tau\notin\mathcal{T}$. In this case, there is a sequence
$t_1>t_2>\cdots$ in $ \mathcal{T}\cap (\tau,1]$
such that $\lim_{k\to +\infty} t_k=\tau$. For each $t_k$ let $(Y_{t_k}/Z,B_{Y_{t_k}}+t_kC_{Y_{t_k}})$ be any log minimal model of
$(X/Z,B+t_kC)$ which exists by the definition of $\mathcal{T}$ and from which we get a model $(Y_{t_k}'/Z,B_{Y_{t_k}'}+\tau C_{Y_{t_k}'})$ by running the LMMP/$Z$ on 
$K_{Y_{t_k}}+B_{Y_{t_k}}+\tau C_{Y_{t_k}}$ with scaling of $(t_k-\tau)C_{Y_{t_k}}$: 
so $K_{Y_{t_k}'}+B_{Y_{t_k}'}+\tau C_{Y_{t_k}'}$ is nef$/Z$.

Let $D\subset X$ be a prime divisor contracted/$Y_{t_k}'$. 
If $D$ is contracted/$Y_{t_k}$, then 
$$
a(D,X,B+t_k C) < a(D,Y_{t_k},B_{Y_{t_k}}+t_kC_{Y_{t_k}})
$$
$$
\le a(D,Y_{t_k},B_{Y_{t_k}}+\tau C_{Y_{t_k}})\le a(D,Y_{t_k}',B_{Y_{t_k}'}+\tau C_{Y_{t_k}'})
$$
 but if $D$ is not contracted/$Y_{t_k}$ we have
$$
a(D,X,B+t_k C) =a(D,Y_{t_k},B_{Y_{t_k}}+t_kC_{Y_{t_k}})
$$
$$
\le a(D,Y_{t_k},B_{Y_{t_k}}+\tau C_{Y_{t_k}}) <a(D,Y_{t_k}',B_{Y_{t_k}'}+\tau C_{Y_{t_k}'})
$$
because $(Y_{t_k}/Z,B_{Y_{t_k}}+t_kC_{Y_{t_k}})$ is a log minimal model of $(X/Z,B+t_kC)$ and  $(Y_{t_k}'/Z,B_{Y_{t_k}'}+\tau C_{Y_{t_k}'})$
is a log minimal model of $(Y_{t_k}/Z,B_{Y_{t_k}}+\tau C_{Y_{t_k}})$. Thus, in any case we have 
$$
a(D,X,B+t_k C) < a(D,Y_{t_k}',B_{Y_{t_k}'}+\tau C_{Y_{t_k}'})
$$

Replacing the sequence $\{t_k\}_{k\in \N}$ with a subsequence, we can assume that all the induced rational maps $X\bir Y_{t_k}'$
contract the same components of $B+\tau C$. Now an easy application of the negativity lemma 
(cf. [\ref{Birkar-mmodel}, Claim 3.5]) implies that the log discrepancy $a(D,Y_{t_k}',B_{Y_{t_k}'}+\tau C_{Y_{t_k}'})$ is independent of $k$.
Therefore, each $(Y_{t_k}',B_{Y_{t_k}'}+\tau C_{Y_{t_k}'})$ 
satisfies
$$
 a(D,X,B+\tau C)=\lim_{l\to +\infty} a(D,X,B+t_l C)\leq a(D,Y_{t_k}',B_{Y_{t_k}'}+\tau C_{Y_{t_k}'})
$$
for any prime divisor $D\subset X$ contracted/$Y_{t_k}'$.\\

\emph{Step 6.} To get a log minimal model of  $(X/Z,B+\tau C)$ we just need to extract those prime divisors $D$ on $X$ 
contracted$/Y_{t_{k}}'$ for which 
$$
 a(D,X,B+\tau C)=a(D,Y_{t_k}',B_{Y_{t_k}'}+\tau C_{Y_{t_k}'})
$$
Since $B$ has a component which is ample$/Z$, we can find $\Gamma$ on $X$ such that 
$\Gamma\sim_{\R} B+\tau C/Z$ and
such that $(X/Z,\Gamma)$  and $(Y_{t_k}'/Z,\Gamma_{Y_{t_k}'})$ are klt. 
Now by Lemma \ref{l-extraction-klt} we can  construct a crepant model of $(Y_{t_k}'/Z,\Delta_{Y_{t_k}'})$ which would be a log minimal model of $(X/Z,\Delta)$. This in turn induces a 
log minimal model of $(X/Z,B+\tau C)$.  Thus, $\tau\in\mathcal T$ and this gives a contradiction. Therefore, $\mathfrak{W}=\emptyset$. $\Box$
\end{proof}

\clearpage
~ \vspace{2cm}
\section{\textbf{Finiteness of log minimal models}}
\vspace{1cm}

To construct log minimal models in dimension $d+1$ assuming the nonvanishing, 
as in Section \ref{s-from-nonvanishing}, one needs the 
special termination with scaling in dimension $d+1$ which is reduced to termination with 
scaling in dimension $d$ (see Remark \ref{rem-special-termination}). Thus, to do induction, we need to prove 
the latter termination in dimension $d$ assuming Theorem \ref{t-main-BCHM} in dimension $d$.
More precisely, we need termination in the following situation.
Let $(X/Z,B+C)$ be a klt pair of dimension $d$ 
where $B\ge 0$ is big$/Z$,  
$C\ge 0$ is $\R$-Cartier, and $K_X+B+C$ is nef$/Z$. Run the LMMP$/Z$ on $K_X+B$ with 
scaling of $C$. We need to prove that this terminates. Since there can 
be only finitely many divisorial contractions in the process, we can assume that 
the LMMP consists of only log flips. Assume that 
$X_i\bir X_{i+1}/Z_i$ is the sequence of log flips$/Z$ we get with $X=X_1$. 
Let $\lambda_i$ be as in Definition \ref{d-LMMP-scaling} and put $\lambda=\lim \lambda_i$. 
So, by definition, $K_{X_i}+B_i+\lambda_iC_i$ is nef$/Z$ and numerically zero over $Z_i$ 
where $B_i$ and $C_i$ are the birational transforms of $B$ and $C$ respectively. 

By construction, $(X_i/Z,B_i+\lambda_iC_i)$ is a log minimal model of $(X/Z,B+\lambda_iC)$.
The sequence terminates if we can prove that there are only finitely many possible 
log minimal models. More generally, we will try to prove that the set of log minimal models 
of all the pairs  $(X/Z,B+tC)$ with $t\in [0,1]$ is finite. To have a painless proof of 
such a finiteness result, it is actually better to go to an even more general setting as 
follows.\\

\noindent ({\bf P}) Let $X\to Z$ be a projective morphism of normal quasi-projective varieties, $A\ge 0$ a $\Q$-divisor on $X$, and 
$V$ a rational (i.e. with a basis consisting of rational divisors) finite dimensional affine subspace of the space of $\R$-Weil divisors on $X$. Define 
$$
\mathcal{L}_{A}(V)=\{B=L+A \mid 0\le L\in V, ~ \mbox{and $(X/Z,B)$ is lc} \}
$$
By Shokurov [\ref{Shokurov-log-flips}, 1.3.2][\ref{Shokurov-log-models}], $\mathcal{L}_{A}(V)$ is a rational polytope (i.e. a polytope with rational vertices) inside the rational affine space $A+V$. 
We will be interested in rational polytopes inside $\mathcal{L}_{A}(V)$.

\begin{rem}\label{r-local}
With the setting as in {\rm({\bf P})} above assume that $A$ is big$/Z$. 
Let $B\in \mathcal{L}_{A}(V)$ such that $(X/Z,B)$ is klt. 
We can write $A\sim_\R A'+G/Z$ where $A'\ge 0$ is an ample$/Z$ $\Q$-divisor 
and $G\ge 0$ is also a $\Q$-divisor. Then, there is a sufficiently small rational number $\epsilon>0$  
such that 
$$
(X/Z,\Delta_B:=B-\epsilon A+\epsilon A'+\epsilon G)
$$
 is klt. 
Note that 
$$
K_X+\Delta_B\sim_\R K_X+B/Z
$$
Moreover, there is a neighborhood of $B$ in $\mathcal{L}_{A}(V)$ such that for any $B'$ in that neighborhood 
$$
(X/Z,\Delta_{B'}:=B'-\epsilon A+\epsilon A'+\epsilon G)
$$ 
is klt. In particular, if 
$\mathcal{C}\subseteq \mathcal{L}_{A}(V)$ is a rational polytope containing $B$, 
then perhaps after shrinking $\mathcal{C}$ (but preserving its dimension) we can assume 
that $\mathcal{D}:=\{\Delta_{B'} \mid B'\in \mathcal{C}\}$ is a rational 
polytope of klt boundaries in $\mathcal{L}_{\epsilon A'}(W)$ where $W$ is the rational 
affine space $V+(1-\epsilon)A+\epsilon G$.  
The point is that we can change $A$ 
and get an ample part $\epsilon A'$ in the boundary. So, when we are concerned with a problem locally around $B$ we feel free to assume that $A$ is actually ample by replacing it with $\epsilon A'$.
\end{rem}

\begin{lem}[Stability of extremal rays]\label{l-ray-stability}
With the setting as in {\rm({\bf P})} above assume that $A$ is big$/Z$. 
Then, 

$(1)$ the set 
$$
\mathcal{N}_{A}(V)=\{B\in\mathcal{L}_A(V) \mid \mbox{$K_X+B$ is nef$/Z$} \}
$$
is a rational polytope. Moreover, if $B\in \mathcal{N}_{A}(V)$ and if $(X/Z,B)$ is klt, 
then $K_X+B$ is semi-ample$/Z$;

$(2)$ Fix  $B\in \mathcal{N}_{A}(V)$ with $(X/Z,B)$ klt. Then, 
there is $\epsilon>0$ (depending on $X\to Z,V,A,B$)
such that if $R$ is a $K_X+B'$-negative extremal ray$/Z$ for some $B'\in \mathcal{L}_{A}(V)$ with $||B-B'||<\epsilon$ 
then $(K_X+B)\cdot R=0$.   
\end{lem}
\begin{proof}
This is proved by Shokurov [\ref{Shokurov-log-flips}][\ref{Shokurov-ordered}, Corollary 9] 
in a more general situtation (see also Birkar [\ref{Birkar-mmodel-II}, \S 3] for a short treatment). 

(1) We will not give the proof of the fact that $\mathcal{N}_{A}(V)$ is a rational 
polytope. See Birkar [\ref{Birkar-mmodel-II}, Proposition 3.2]. Assume that 
$B\in \mathcal{N}_{A}(V)$ with $(X/Z,B)$ being klt. Since $\mathcal{N}_{A}(V)$ is a rational 
polytope, we can find boundaries $B_i\in \mathcal{N}_{A}(V)$ and real numbers $c_i\ge 0$ 
such that $\sum c_i=1$, $K_X+B=\sum c_i(K_X+B_i)$, and each $(X/Z,B_i)$ is klt.
By the base point free theorem \ref{t-base-point-free}, each $K_X+B_i$ is semi-ample$/Z$ hence 
$K_X+B$ is also semi-ample$/Z$.

(2) By (1), $K_X+B$ is semi-ample$/Z$ hence there is a contraction $f\colon X\to S/Z$ and an ample$/Z$ 
$\R$-Cartier divisor $H$ on $S$ such that $K_X+B\sim_\R f^*H/Z$. We can write $H\sim_\R \sum a_iH_i/Z$ where $a_i>0$ 
and the $H_i$ are ample$/Z$ Cartier divisors on $S$. Therefore, there is $\delta>0$ such that for any curve 
$C/Z$ in $X$ either $(K_X+B)\cdot C=0$ or $(K_X+B)\cdot C>\delta$. 

Now let $\mathcal{C} \subset \mathcal{L}_{A}(V)$ be a rational polytope of maximal dimension which contains an open 
neighborhood of $B$ in $\mathcal{L}_{A}(V)$ and such that $(X/Z,B')$ is klt for any $B'\in \mathcal{C}$. Pick
$B'\in \mathcal{C}$ and let $B''$ be the point on the boundary of $\mathcal{C}$ such that $B'$ belongs to the line segment 
determined by $B,B''$. Assume that $R$ is a $K_X+B'$-negative extremal ray$/Z$ such that $(K_X+B)\cdot R>0$. Then 
$(K_X+B'')\cdot R<0$ and there is a rational curve $\Gamma$ in $R$ such that 
$(K_X+B'')\cdot \Gamma\ge -2d$. Since $(K_X+B')\cdot \Gamma<0$, 
$$
(B'-B)\cdot \Gamma=(K_X+B')\cdot \Gamma-(K_X+B)\cdot \Gamma <-\delta
$$ 
Now, $||B''-B||>\alpha$ for some $\alpha>0$ independent of $B''$. Thus, if $||B-B'||$ is too small, then $(B''-B')\cdot \Gamma$ is too negative and we cannot have 
$$
(K_X+B'')\cdot \Gamma=(K_X+B')\cdot \Gamma+(B''-B')\cdot \Gamma\ge -2d
$$
So, we get a contradiction.\\
\end{proof}

\begin{thm}\label{t-finiteness}
Assume $\rm (1)$ of Theorem \ref{t-main-BCHM} in dimension $d$. With the setting as in {\rm({\bf P})} above assume that $A$ is big$/Z$. Let $\mathcal{C}\subseteq \mathcal{L}_{A}(V)$ 
be a rational polytope such that $(X/Z,B)$ is klt for any $B\in \mathcal{C}$ where $\dim X=d$. Then, there are finitely many birational 
maps $\phi_i\colon X\bir Y_i/Z$ such that for any $B\in \mathcal{C}$ with $K_X+B$ pseudo-effective$/Z$, there 
is $i$ such that $(Y_i/Z,B_{Y_i})$ is a log minimal model of $(X/Z,B)$.
\end{thm}
\begin{proof}
Remember that as usual $B_{Y_i}$ is the birational transform of $B$. 
We do induction on the dimension of $\mathcal{C}$. In particular, we may assume that the dimension of $\mathcal{C}$ is positive.
We may proceed locally, so fix $B\in \mathcal{C}$. If $K_X+B$ is not pseudo-effective$/Z$ then the same 
holds in a neighborhood of $B$ inside $\mathcal{C}$, so 
we may assume that $K_X+B$ is pseudo-effective$/Z$. By assumptions, $(X/Z,B)$ 
has a log minimal model $(Y/Z,B_Y)$. Moreover, the polytope $\mathcal{C}$ determines a rational polytope $\mathcal{C}_Y$ of 
$\R$-divisors on $Y$ by taking birational transforms of elements of $\mathcal{C}$. If we shrink $\mathcal{C}$ 
around $B$ we can assume that the inequality on discrepancies in Definition \ref{d-mmodel} is satisfied 
for prime divisors on $X$ excpetional$/Y$, for every $B'\in \mathcal{C}$. That is, 
$$
a(D,X,B')< a(D,Y,B_Y')
$$ 
for any prime divisor $D\subset X$ contracted$/Y$ and any $B'\in \mathcal{C}$. Moreover, for each $B'\in \mathcal{C}$, a log minimal model of 
 $(Y/Z,B_Y')$ is also a log minimal model of $(X/Z,B')$. 
Therefore, we can replace $(X/Z,B)$ by $(Y/Z,B_Y)$ and assume from now on that $(X/Z,B)$ 
is a log minimal model of itself, in particular, $K_X+B$ is nef$/Z$.

Since $B$ is big$/Z$, by Lemma \ref{l-ray-stability}, $K_X+B$ is semi-ample$/Z$ so there is a contraction $f\colon X\to S/Z$ such that $K_X+B\sim_\R f^*H/Z$ 
for some ample$/Z$ $\R$-divisor $H$ on $S$. Now by induction on the dimension 
of $\mathcal{C}$, we may assume that 
there are finitely many birational maps $\psi_j\colon X\bir Y_j/S$ such that 
for any $B''$ on the boundary of $\mathcal{C}$ with $K_X+B''$ pseudo-effective$/S$, there is $j$ such that 
$(Y_j/S,B_{Y_j}'')$ is a log minimal model of $(X/S,B'')$.

By Lemma \ref{l-ray-stability}, there is a sufficiently small $\epsilon>0$ such that for any $B'\in \mathcal{C}$ 
with $||B-B'||<\epsilon$, any 
$j$, and any $K_{Y_j}+B_{Y_j}'$-negative extremal ray $R/Z$ we have the equality $(K_{Y_j}+B_{Y_j})\cdot R=0$. 
Note that all the pairs $(Y_j/Z,B_{Y_j})$ are klt and $K_{Y_j}+B_{Y_j}\sim_\R 0/S$ and nef$/Z$ because $K_X+B\sim_\R 0/S$.

Pick $B'\in \mathcal{C}$ with $0<||B-B'||<\epsilon$ such that $K_X+B'$ is pseudo-effective$/Z$, and let $B''$ be the unique point on 
the boundary of $\mathcal{C}$ such that $B'$ belongs to the line segment given by $B$ and $B''$. Since $K_X+B\sim_\R 0/S$, there is some $t>0$ such that 
$$
K_X+B''=K_X+B+B''-B\sim_\R B''-B=t(B'-B)\sim_\R t(K_X+B')/S 
$$
hence 
$K_X+B''$ is pseudo-effective$/S$, and $(Y_j/S,B_{Y_j}'')$ is a log minimal model of $(X/S,B'')$ for some $j$. Moreover, $(Y_j/S,B_{Y_j}')$ is a log minimal model of $(X/S,B')$. Furthermore, $(Y_j/Z,B_{Y_j}')$ is a log minimal model of $(X/Z,B')$ because any $K_{Y_j}+B_{Y_j}'$-negative extremal ray $R/Z$ would be over $S$ by the choice of $\epsilon$. Finally, we just need to 
shrink $\mathcal{C}$ around $B$ appropriately.
\end{proof}

\clearpage
~ \vspace{2cm}
\section{\textbf{Termination with scaling}}
\vspace{1cm}

\begin{rem}[$\Q$-factorialisation]\label{r-klt-Q-factorialisation}
Assume $\rm (1)$ of Theorem \ref{t-main-BCHM} in dimension $d$ and let $(X/Z,B)$ 
be a klt pair of dimension $d$.  Let $f\colon W\to X$ be a log resolution of
$(X/Z,B)$ and let $G$ be the reduced exceptional divisor of $f$. Then, 
if we let $B_W=B^\sim+(1-\epsilon)G$ where $B^\sim$ is the birational transform of 
$B$ and $\epsilon>0$ is sufficiently small, then 
$$
K_W+B_W=f^*(K_X+B)+E
$$
 where $E\ge 0$ and $\Supp E=\Supp G$. Let $(Y/X,B_Y)$ be a 
log minimal model of $(W/X,B_W)$. By the negativity lemma \ref{l-negativity}, 
the morphism $Y\to X$ is small, and by definition of log minimal models, 
$Y$ is $\Q$-factorial. The pair $(Y/Z,B_Y)$ is often called a small $\Q$-factorialisation 
of $(X/ZB)$. 
\end{rem}

\begin{rem}\label{r-decompose-flips}
Assume $\rm (1)$ of Theorem \ref{t-main-BCHM} in dimension $d$ and let $(X/Z,B+C)$ be a 
klt pair of dimension $d$ with $B,C\ge 0$ being $\R$-Cartier, and $K_X+B+C\equiv 0/Z$. Assume that 
$K_X+B$ is big$/Z$ and let $X'/Z$ be its lc model and $X\bir X'/Z$ the induced rational map.  Let $Y\to Z$ be a small $\Q$-factorialisation of $X$ 
which exists by Remark \ref{r-klt-Q-factorialisation}, and let $B_Y,C_Y$ denote birational transforms 
as usual. Now run the LMMP$/Z$ on $K_Y+B_Y$ with scaling of $C_Y$. 
If the LMMP terminates with a log minimal model $(Y'/Z,B_{Y'})$, then $Y'\to Z$ factors through 
$X'$ because $X'/Z$ is 
the lc model of $(X/Z,B)$ as well as of $(Y/Z,B_{Y})$. We call the birational map $Y\bir Y'$ a \emph{$\Q$-factorial lift} 
of $X\bir X'$. By construction, $Y\bir Y'$ is an isomorphism or else it is decomposed into a finite  sequence $Y_i\bir Y_{i+1}/Z_i$ 
of divisorial contractions and log flips$/Z$ such that $K_{Y_i}+B_{Y_i}+C_{Y_i}\equiv 0/Z_i$ and $C_{Y_i}$ is ample$/Z_i$. If $X\bir X'$ is an isomorphism in codimension one (e.g. a log flip), 
then only log flips can occur in the sequence $Y_i\bir Y_{i+1}/Z_i$ . One can use this construction to lift a 
sequence of log flips in the non-$\Q$-factorial case to a sequence of log flips in the 
$\Q$-factorial case.
\end{rem}

\begin{lem}\label{l-term-1}
Let $(X/Z,B+C)$ be a klt pair of dimension $d$ where $B\ge 0$ is big$/Z$,  
$C\ge 0$ is $\R$-Cartier, and $K_X+B+C$ is nef$/Z$. Assume that there is an 
LMMP$/Z$ on $K_X+B$ with scaling of $C$ and let $\lambda_i$ be the numbers 
appearing in the LMMP, and $\lambda=\lim \lambda_i$. Assume that $(X/Z,B+\lambda C)$ 
has a log minimal model $(Y/Z,B_Y+\lambda C_Y)$ and that if $Y\to T$ is the 
contraction associated to the semi-ample$/Z$ divisor $K_Y+B_Y+\lambda C_Y$, 
then $(Y/T,B_Y+\lambda_i C_Y)$ has a log minimal model for any $i\gg 0$. 
Then, the LMMP terminates.
\end{lem}
\begin{proof}
We may assume that the LMMP consists of only a sequence $X_i\bir X_{i+1}/Z_i$ of 
log flips$/Z$, and that $X_1=X$. 
Take $i$ sufficiently large so that $(Y/T,B_Y+\lambda_i C_Y)$ has a log minimal model 
$(Y'/T,B_{Y'}+\lambda_i C_{Y'})$. Since $K_{X_i}+B_i+\lambda_iC_i$ is semi-ample$/Z$, 
$K_{Y}+B_Y+\lambda_iC_Y$ is movable$/T$ hence $Y$ and $Y'$ are isomorphic in codimension 
one. Moreover, since $K_Y+B_Y+\lambda C_{Y}\sim_\R 0/T$, 
$K_{Y'}+B_{Y'}+\lambda C_{Y'}$ is nef$/Z$ and actually 
$(Y'/Z,B_{Y'}+\lambda C_{Y'})$ is a log minimal model of $(X/Z,B+\lambda C)$. 
So, by replacing $Y$ with $Y'$ we can assume that 
$K_{Y}+B_{Y}+\lambda_i C_{Y}$ is nef$/T$. Now by Lemma \ref{l-ray-stability}, 
if $i$ is large enough $K_{Y}+B_{Y}+\lambda_i C_{Y}$ is nef$/Z$ hence 
$(Y/Z,B_{Y}+\lambda_i C_{Y})$ is a log minimal model of 
$(X/Z,B+\lambda_i C)$.
Moreover, $K_{Y}+B_{Y}+\lambda_j C_{Y}$ is nef$/Z$ for any $j\ge i$ hence 
$(Y/Z,B_{Y}+\lambda_j C_{Y})$ and $({X_j}/Z+B_j+\lambda_jC_j)$ are both log minimal models of 
$(X/Z,B+\lambda_j C)$ for such $j$.
 
Let $Y\to T_j$ be the contraction associated to the semi-ample$/Z$ divisor
$K_{Y}+B_{Y}+\lambda_j C_{Y}$. If $j\gg 0$, $Y\to T_j$ is independent of $j$ 
and $T_j$ maps to $T$. Put $T'=T_j$ for  $j\gg 0$. 
By construction, $K_{X_j}+B_j+\lambda_jC_j\sim_\R 0/T'$ 
and $X_j\to Z$ factors through $T'$. 
But $K_{X_j}+B_j+\lambda C_j$ is negative 
on some curve $C_j/T'$ and on other hand $K_{X_j}+B_j+\lambda C_j\sim_\R 0/T'$ as 
$K_{Y}+B_{Y}+\lambda C_Y\sim_\R 0/T$. This is a contradiction.
\end{proof}

\begin{thm}\label{t-term}
Assume $\rm (1)$ of Theorem \ref{t-main-BCHM} in dimension $d$ and let $(X/Z,B+C)$ be a klt pair of dimension $d$ 
where $B\ge 0$ is big$/Z$,  
$C\ge 0$ is $\R$-Cartier, and $K_X+B+C$ is nef$/Z$. Then, any LMMP$/Z$ on $K_X+B$ with scaling of $C$ terminates.  
\end{thm}
\begin{proof}
Apply Lemma \ref{l-term-1}.

\end{proof}

Note that existence of klt log flips in dimension $d$ follows from the assumptions of Theorem 
\ref{t-term} (see the proof of Corollary \ref{cor-log-flips}). So, if $X$ is $\Q$-factorial, under the 
assumptions of the theorem, we can actually run an LMMP$/Z$ on $K_X+B$ with scaling 
of $C$ by Lemma \ref{l-ray-scaling}.

\begin{thm}\label{t-s-term}
Assume $\rm (1)$ of Theorem \ref{t-main-BCHM} in dimension $d-1$ and let $(X/Z,B+C)$ be a $\Q$-factorial dlt pair of dimension 
$d$ where $B-\rddown{B}-A\ge 0$ for some ample$/Z$ $\R$-divisor $A\ge 0$, and $C\ge 0$. Assume that 
$(Y/Z,B_Y+C_Y)$ is a log minimal model of $(X/Z,B+C)$. Then, the special termination holds for any LMMP$/Z$ on $K_Y+B_Y$ with scaling of $C_Y$.
\end{thm}
\begin{proof}
 Suppose that we have an LMMP$/Z$ on $K_Y+B_Y$ with scaling of $C_Y$
producing a sequence $Y_i\bir Y_{i+1}/Z_i$ of log flips$/Z$. Let $S$ be 
a component of $\rddown{B}$ and let $S_Y$ and $S_{Y_i}$ be its birational transforms on $Y$ and $Y_i$ respectively. 
Let $\lambda_i$ be as in Definition \ref{d-LMMP-scaling} for the sequence 
$Y_i\bir Y_{i+1}/Z_i$. That is, $K_{Y_i}+B_{Y_i}+\lambda_iC_{Y_i}$ is nef$/Z$ but 
numerically zero over $Z_i$, and $C_{Y_i}$ is ample$/Z_i$. 

First suppose that $\lambda_i=1$ for every $i$. Since $B-\rddown{B}-A\ge 0$ and since $A$ is ample$/Z$, using a simple perturbation of coefficients we can write 
$$
B+C\sim_\R A'+B'+C'/Z
$$ 
where for a sufficiently small rational number $\epsilon>0$  
$$
A'\sim_\R A+\epsilon {C}+\epsilon \rddown{B-S} 
$$ 
is ample/$Z$, and 
$$
B'=B-A-\epsilon \rddown{B-S}\ge 0,~~~ C'=(1-\epsilon)C,
$$ 
$$
\rddown{B'}=\rddown{A'+B'+C'}=S
$$ 
Moreover, perhaps after another small perturbation we may assume that $A'$ is a $\Q$-divisor and that the pairs 
$$
(X/Z,A'+B'+C')~~\mbox{and}~ ~(Y/Z,A'_Y+B_Y'+C_Y')
$$ 
are plt, 
 and that the above LMMP/$Z$ on $K_Y+B_Y$ with scaling of $C_Y$ is an LMMP/$Z$ on $K_Y+A_Y'+B_Y'$ 
with scaling of $C_Y'$. 

Assume that $S_{Y_1}\neq 0$ otherwise there is nothing to prove. Following some standard arguments (cf. [\ref{Fujino-st}]) we may assume that the induced birational maps $S_{Y_{i+1}}\bir S_{Y_{i}}$ do not contract divisors.
Since $A'_{Y_i}$ is the pushdown of 
an ample$/Z$ divisor, $A'_{Y_i}|_{S_{Y_i}}$ is big$/Z$. Moreover, if $T_i$ is the 
normalisation of the image of $S_{Y_i}$ in $Z_i$, then 
$$
(K_{Y_i}+A_{Y_i}'+B_{Y_i}'+\lambda_i C_{Y_i}')|_{S_{Y_i}}\sim_\R 0/T_i 
$$ 
Furthermore, by taking $\Q$-factorial lifts of the maps $S_{Y_i}\bir S_{Y_{i+1}}$ as in Remark \ref{r-decompose-flips} and applying Theorem \ref{t-term} in dimension $d-1$ we deduce that 
$S_{Y_i}\bir S_{Y_{i+1}}$ are isomorphisms for $i\gg 0$ hence the log flips in the sequence  $Y_i\bir Y_{i+1}/Z_i$ do not intersect $S_{Y_i}$ for $i\gg 0$.

Now assume that we have  $\lambda_i<1$ for some $i$. Then, $\rddown{B+\lambda_i C}=\rddown{B}$ for any $i\gg 0$. So, we may assume that $\rddown{B+C}=\rddown{B}$. Since $B-\rddown{B}-A\ge 0$ and since $A$ is ample/$Z$, similar to the above, we can write 
$B\sim_\R A'+B'/Z$ such that $A'\ge 0$ is an ample/$Z$ $\Q$-divisor, 
$$
B'\ge 0, ~~\rddown{B'}=\rddown{A'+B'+C}=S
$$
 and 
$$
(X/Z,A'+B'+C) ~~\mbox{and}~~ (Y/Z,A_Y'+B_Y'+C_Y)
$$ 
are plt. The rest goes as before by restricting to the birational transforms of $S_Y$.\\
\end{proof}

An application of the last theorem is the following "extraction" result which is useful
in many situations.

\begin{lem}\label{l-extraction-klt}
Assume $\rm (1)$ of Theorem \ref{t-main-BCHM} in dimension $d-1$ and assume existence of pl 
flips in dimension $d$. 
Let $(X/Z,B)$ be a klt pair of dimension
 $d$ and let $\{D_i\}_{i\in I}$ be a finite set of exceptional$/X$ prime divisors (on birational
 models of $X$) such that the log discrepancy $a(D_i,X,B)\le 1$. Then, there is a $\Q$-factorial klt pair $(Y/X,B_Y)$
 such that\\\\
 $\rm (1)$ $Y\to X$ is birational and $K_Y+B_Y$ is the crepant pullback of $K_X+B$,\\
 $\rm (2)$ the set of exceptional/$X$ prime divisors of $Y$ is exactly $\{D_i\}_{i\in I}$.\\
\end{lem}
\begin{proof}
 Let $f\colon W\to X$ be a log resolution of
$(X/Z,B)$ and let $\{E_j\}_{j\in J}$
be the set of prime exceptional divisors of $f$. We can assume that
for some $J'\subseteq J$, $\{E_j\}_{j\in J'}=\{D_i\}_{i\in I}$. Since $f$ is birational, there is an ample$/X$ $\Q$-divisor $H\ge 0$ on $W$ whose 
support is irreducible smooth and distinct from the birational transform of the components of $B$, and an $\R$-divisor $G\ge 0$ such that $H+G\sim_\R 0/X$. 
Moreover, there is $\epsilon>0$ such that $(X/Z,B+\epsilon f_*H+\epsilon f_*G)$ is klt. Now define
$$
K_W+\overline{B}_W:=f^*(K_{X}+B+\epsilon f_*H+\epsilon f_*G)+\sum_{j\notin J'} a(E_j,X,B+\epsilon f_*H+\epsilon f_*G)E_j
$$
 for which obviously there is an exceptional$/X$ $\R$-divisor $\overline{M}_W\ge 0$ such that 
$$
K_W+\overline{B}_W\sim_\R \overline{M}_W/X 
~~\mbox{and}~~ \theta(W/X,\overline{B}_W,\overline{M}_W)=0
$$
Note that $\overline{B}_W-\rddown{\overline{B}_W}\ge \epsilon H$. We can run an LMMP/$X$ on $K_W+\overline{B}_W$ with scaling of a suitable ample$/X$ $\R$-divisor, and 
using the special termination of Theorem \ref{t-s-term} we get a log minimal model  of $(W/X,\overline{B}_W)$ which we 
may denote by $(Y/X,\overline{B}_Y)$. Note that here we only need pl flips to run the LMMP/$X$ because any extremal ray in the process 
 intersects some component of $\rddown{\overline{B}_W}$ negatively.

The exceptional divisor $E_j$ is contracted$/Y$ exactly when $j\notin J'$. By taking 
$K_Y+B_Y$ to be the crepant pullback of  $K_X+B$ we get the result.
\end{proof}

\clearpage
~ \vspace{2cm}
\section{\textbf{The nonvanishing}}
\vspace{1cm}

The constructions of Section \ref{s-from-nonvanishing} relied on having an 
effective divisor $M$ satisfying $K_X+B\sim_\R M/Z$.
To prove Theorem \ref{t-main-BCHM}, we come across divisors 
$K_X+B$ that are only pseudo-effective. We need to turn this pseudo-effectivity  
into a geometric effectivity as in the next theorem to complete the induction 
process.

\begin{thm}[Nonvanishing]\label{t-nonvanishing}
Let $(X/Z,B)$ be a klt pair where $B$ is big$/Z$. 
If $K_X+B$ is pseudo-effective$/Z$, then $K_X+B\sim_\R M/Z$ for some $M\ge 0$.
\end{thm}

If $K_X+B$ is nef$/Z$, then one can actually just apply the base point free theorem 
(see Theorem \ref{t-base-point-free} and Lemma \ref{l-ray-stability}) to show that such an 
$M$ exists and even to deduce that $M$ is semi-ample$/Z$. However, when $K_X+B$ is not 
nef$/Z$, we need to follow the general strategy of the proof of Shokurov nonvanishing 
Theorem \ref{t-Shokurov-nonvanishing} though some of the tools have to replaced.
First we will deal with the case $Z=\rm pt$ and at the end we prove the general 
statement which easily follows from the case $Z=\rm pt$.

\begin{rem}[Nakayama Kodaira dimension]\label{rem-Nakayama-decomposition}
Let $X$ be a smooth projective variety and $D$ a pseudo-effective $\R$-divisor on $X$. 
Nakayama [\ref{Nakayama}] studied the properties of such divisors in detail. 
In particular, he defined the numerical Kodaira dimension 
$\kappa_\sigma(D)$ of $D$ as 
the largest integer such that for some ample divisor $H$ 
$$
\limsup_{m\to +\infty} \frac{h^0(X,\rddown{mD}+H)}{m^{\kappa_\sigma(D)}}>0
$$ 
Moreover, he proved that one can give a decomposition 
$D=P_\sigma(D)+N_\sigma(D)$ where $P_\sigma(D)$ is pseudo-effective and
$N_\sigma(D)\ge 0$ is canonically defined in a limiting process. 
This resembles the classical Zariski decomposition. One of the important properties 
of the decomposition is that if $S$ is a smooth prime divisor on $X$ 
such that $S$ is not a component of $N_\sigma(D)$, then $D|_S$ is again 
pseudo-effective: more precisely, there is an ample divisor $H$ such that 
$S$ is not in $\Bs|\rddown{mD}+H|$ for any $m>0$ [\ref{Nakayama}, Theorem 6.1.3].

Some other properties that Nakayama proved are: (1) $\kappa_\sigma(D)=\kappa_\sigma(kD)$
for any $k\in \N$; (2) if $\kappa_\sigma(D)=0$, then  $P_\sigma(D)\equiv 0$ 
hence $D\equiv N_\sigma(D)$; (3) if $\kappa_\sigma(D)>0$, then we can choose 
$H$ so that $h^0(X,\rddown{mD}+H)\ge m\beta$ for a fixed $\beta>0$ and every $m\gg 0$. 

\end{rem}

\begin{lem}\label{l-nv-bounded} Assume Theorem \ref{t-main-BCHM} (1) for projective pairs 
$(X',B')$ of dimension $d$ such that $K_{X'}+B'\sim_\R M'$ for some $M'\ge 0$.  Let $(X,B)$ be a 
projective log smooth klt pair of dimension $d$ such that
$K_X+B$ is pseudo-effective and $B-A\geq 0$ for an ample $\mathbb{Q}$-divisor
$A$. If $\kappa_\sigma(K_X+B)=0$, 
 then there is an $\mathbb{R}$-divisor $M\ge 0$ such that $K_X+B\sim_{\mathbb{R}}M$.  
\end{lem}
\begin{proof} 

By Remark \ref{rem-Nakayama-decomposition}, $K_X+B\equiv M'$ for some $M'\ge 0$. 
 Since $M'-(K_X+B)\equiv 0$,
$$
A':=A+M'-(K_X+B)\equiv A
$$
is ample.  Thus 
$$
K_X+B':=K_X+A'+B-A
$$
satisfies $K_X+B'=M'$. So, by assumptions, $(X,B')$ has a log minimal model 
$(Y,B_Y')$ and the construction ensures that $(Y,B_Y)$ is a log minimal model 
of $(X,B)$. We can replace $(X,B)$ by $(Y,B_Y)$ hence assume that $K_X+B$ is 
nef. Now we simply use the base point free theorem (Theorem \ref{t-base-point-free} 
and Lemma \ref{l-ray-stability}) to finish the proof.
\end{proof}

When $\kappa_\sigma(K_X+B)>0$ the proof of the nonvanishing theorem is a lot more 
complicated. In this case we try to create a component with coefficient one in 
$B$ and use induction.

\begin{lem}\label{l-nv-unbounded} Let $(X,B)$ be a projective log smooth klt 
pair of dimension $d$ such that $B\ge A$ where $A$ is an ample $\mathbb{Q}$-divisor.
 Suppose that $\kappa_\sigma(K_X+B)>0$. 
Then we can find a projective, log smooth, plt pair $(W,B_W)$ and an
ample $\mathbb{Q}$-divisor $A_W$ on $W$ such that
\begin{itemize} 
\item $W$ is birational to $X$, 
\item $B_W-A_W\geq 0$, and
\item $S=\rddown{B_W}$ is an irreducible divisor, which is not a component of
$N_{\sigma}(K_W+B_W)$.  
\end{itemize} 

Moreover the pair $(W,B_W)$ has the property that $K_X+B \sim_{\mathbb{R}} M$ for some 
$\mathbb{R}$-divisor $M\ge 0$ iff $K_W+B_W \sim_{\mathbb{R}} M_W$ for some
$\mathbb{R}$-divisor $M_W\ge 0$.
\end{lem}
\begin{proof} Since $\kappa_\sigma(K_X+B)>0$ and since $A$ is ample, 
there is some $k\in \N$ such that $kA$ is integral and 
$$
\limsup_{m\to +\infty} \frac{h^0(X,\rddown{mk(K_X+B)}+kA)}{m}>0
$$

Pick $m$ large enough so that
$$
h^0(X,\rddown{mk(K_X+B)}+kA) \gg \frac{(kd)^d}{d!}
$$
By Remark \ref{r-multiplicity-linear-systems}, given any fixed point $x\in X$, we may find
$$
0\le D\sim \rddown{mk(K_X+B)}+kA
$$
such that the multiplicity $\mu_xD >kd$.  In particular, we may find an effective
$\mathbb R$-divisor
$$
L\sim_{\mathbb{R}} m(K_X+B)+A,
$$
such that $\mu_xL>d$.  

Given $t\in [0,m]$, consider
\begin{align*} 
(t+1)(K_X+B) &= K_X-\frac{t}{m}A+B+t(K_X+B+\frac 1mA) \\ 
                  &\sim_{\mathbb{R}}K_X-\frac{t}{m}A+B+\frac t mL \\ 
                  &=:K_X+B_t
\end{align*} 

Fix $0<\epsilon\ll 1$, let $A'=\epsilon/m A$ and $u=m-\epsilon$.  We have:
\begin{enumerate} 
\item $K_X+B_0$ is klt,
\item $B_t \ge A'$,  for any $t\in[0,u]$ and
\item the locus of log canonical singularities of $(X,B_u)$ is not contained in
$\Supp N_{\sigma}(K_X+B_u)=\Supp N_{\sigma}(K_X+B)$.
\end{enumerate} 

Let $\pi\colon W\to X$ be a log resolution of the $(X,B+L)$.  We may write
$$
K_W+G_t=\pi^*(K_X+B_t)+E_t,
$$
where $E_t$ and $G_t$ are effective, with no common components, $\pi_*G_t=B_t$
and $E_t$ is exceptional.  Pick an effective exceptional divisor $F$ and a positive
integer $l$ such that $l(\pi^*A'-F)$ is very ample and let $lC$ be a very general element
of the linear system $|l(\pi^*A'-F)|$.  For any $t\in [0,u]$, let
$$
H_t=G_t-\pi^*A'+C+F\sim_\mathbb{R} G_t.
$$
After cancelling common components of $H_t$ and $N_{\sigma}(K_W+H_t)$,
properties (1-3) above become
\begin{enumerate} 
\item $K_W+H_0$ is kawamata log terminal,
\item $H_t \ge C$,  for any $t\in[0,u]$ and
\item the locus of log canonical singularities of $(W,H_u)$ is not contained in
$N_{\sigma}(K_W+H_u)$.
\end{enumerate} 
Moreover
\begin{enumerate}
\setcounter{enumi}{3}
\item $(W,H_t)$ is log smooth, for any $t\in [0,u]$.   
\end{enumerate} 

Let
$$
s=\sup \{\, t\in [0,u] \,|\, \text{$K_W+H_t$ is log canonical} \,\}.  
$$
Thus, setting $B_W=H_s$ and $A_W=C$, we may write 
$$
B_W=S+A_W+B_W'
$$
where $\rddown{B_W}=S$, $A_W$ is ample and $B_W'$ is effective.  Possibly perturbing
$B_W$, we may assume that $S$ is irreducible, so that $K_W+B_W$ is plt 
and we may assume that $A_W$ is $\mathbb{Q}$-Cartier.\\ 
 \end{proof}

We will need the following consequence of Kawamata-Viehweg vanishing:

\begin{lem}\label{l-nv-extension}
 Let $(X,B=S+A+B')$ be a $\mathbb{Q}$-factorial projective 
plt pair and let $m$ be a positive integer.  Suppose that
\begin{enumerate} 
\item $S=\rddown{B}$ is irreducible,
\item $m(K_X+B)$ is integral,
\item $m(K_X+B)$ is Cartier in a neighbourhood of $S$, 
\item $h^0(S,m(K_X+B)|_S)> 0$, 
\item $(X,A+B'-(m-1)tH)$ is klt for some $H\geq 0$ and $t$, 
\item $K_X+B+tH$ is big and nef.  
\end{enumerate} 

Then $h^0(X,m(K_X+B))>0$.
\end{lem}
\begin{proof} Considering the exact 
sequence,
$$
 H^0(X,\mathcal{O}_X(m(K_X+B))) 
\to H^0(S,\mathcal{O}_S(m(K_X+B)|_S))) 
$$
$$
\to  H^1(X,\mathcal{O}_X(m(K_X+B)-S))
$$

it suffices to observe that
$$
H^1(X,\mathcal{O}_X(m(K_X+B)-S))=0
$$
by Kawamata-Viehweg vanishing, since
\begin{align*} 
m(K_X+B)-S &= (m-1)(K_X+B)+K_X+A+B'\\ 
                &= K_X+A+B'-(m-1)tH+(m-1)(K_X+B+tH)
\end{align*} 
and $K_X+B+tH$ is big and nef.\\  
\end{proof}

\begin{thm}\label{t-to-nonvanishing}
Assume  Theorem \ref{t-main-BCHM} (1) in dimension $d-1$, 
and also Theorem \ref{t-main-BCHM} (1) in dimension $d$ for projective pairs 
$(X',B')$ such that $K_{X'}+B'\sim_\R M'$ for some $M'\ge 0$.
Let $(X,B)$ be a projective  klt pair of dimension $d$ where $B$ is big. 
If $K_X+B$ is pseudo-effective, then $K_X+B\sim_\R M$ for some $M\ge 0$.
\end{thm}
\begin{proof} 
 By taking a log resolution we may assume that $(X,B)$ is log smooth.  
By Remark \ref{r-local}, we may
assume that $B\ge A$ where $A$ is an ample $\Q$-divisor.  By Lemma \ref{l-nv-bounded} and
Lemma \ref{l-nv-unbounded}, we may pass to the plt situation 
in which $B=S+A+B'$, where $A$ is an ample
$\mathbb{Q}$-divisor, $B'$ is effective and $\rddown{B}=S$ is irreducible and not a
component of $N_{\sigma}(K_X+B)$ (of course $(X,B)$ is not klt any more). When 
$B$ is not rational the proof involves some technicalities. 
But if $B$ is rational the proof is more transparent and still contains 
the main points. For simplicity we then assume from now 
on that $B$ is rational (see [\ref{BCHM}] for the full story).

Let $a>0$ be a sufficiently large rational number. Since $A$ is ample, 
$aA\sim_\Q A'$ such that $K_X+B+A'$ is plt and nef. Now, by Theorem \ref{t-s-term}, 
the LMMP on $K_X+B$ with scaling of $A'$ terminates near $S$. Moreover, 
by Lemma \ref{l-term-1}, if $\lambda_i$ are the numbers appearing in the 
LMMP with scaling, then $\lambda=\lim \lambda_i=0$. Pick $i\gg 0$, 
and let $Y:X_i$ be the model in the LMMP corresponding to $\lambda_i$. 
So, $K_Y+B_Y+\lambda_iA_Y'$ is nef, and also $K_{S_Y}+B_{S_Y}:=(K_Y+B_Y)|_{S_Y}$ 
is nef and independent of $i$. Pick $m\in \N$  large enough so that 
$m(K_Y+B_Y)$ is integral, $m(K_Y+B_Y)$ is Cartier in a neighbourhood of $S_Y$, 
and $h^0(S_Y,m(K_{S_Y}+B_{S_Y}))> 0$. Moreover, we can choose $i$ so that 
$A_Y-(m-1)a\lambda_iA_Y\ge 0$ hence 
$$
m(K_Y+B_Y)-S_Y= 
$$
$$
K_Y+A_Y+B_Y'-(m-1)a\lambda_iA_Y+(m-1)(K_Y+B_Y+\lambda_iA_Y')             
$$
which allows us to apply the Kawamata-Viehweg vanishing as in Lemma \ref{l-nv-extension}.
\end{proof}

\begin{thm}\label{t-nonvanishing-induction}
Under the assumptions of Theorem \ref{t-to-nonvanishing}, Theorem \ref{t-nonvanishing} 
holds in dimension $d$.
\end{thm}
\begin{proof}
Apply Theorem \ref{t-to-nonvanishing} to the generic fibre of $X\to Z$ from 
which one gets $M\ge 0$ such that $K_X+B\sim_\R M/Z$ by standard arguments.
\end{proof}

\clearpage
~ \vspace{2cm}
\section{\textbf{A few consequences of Theorem \ref{t-main-BCHM}}}
\vspace{1cm}

In this section, we give a few of the immediate consequences of 
 Theorem \ref{t-main-BCHM}. We have already mentioned finite generation of 
 log canonical rings  and existence of log flips for klt pairs (\ref{cor-fg} and 
 \ref{cor-log-flips}).\\ 

\emph{Log Fano varieties.} 
If $X\to Z$ is a projective morphism of normal varieties and $D$ 
an $\R$-Cartier divisor on $X$, then one can define the $D$-MMP 
similar to the usual LMMP (but in general there is no guarantee that the 
extremal rays exist or are contractible, etc). 

A lc pair $(X/Z,B)$ is called log Fano if 
$-(K_X+B)$ is ample$/Z$. It turns out that log Fano pairs are in a 
sense the ideal type of pairs as far as the LMMP is concerned.
One of the nice properties already follows from the cone theorem: 
if $(X/Z,B)$ is a klt log Fano then $\overline{NE}(X/Z)$ is 
a finite polyhedral cone, i.e. it is a cone generated by 
finitely many extremal rays.

\begin{thm}
Let $(X/Z,B)$ be a $\Q$-factorial dlt log Fano pair, and let 
$D$ be an $\R$-divisor on $X$. Then, some $D$-MMP holds. Moreover, 
if $D$ is nef$/Z$, then $D$ is semi-ample$/Z$.
\end{thm}
\begin{proof}
Under the assumptions, there is an ample$/Z$ $\R$-divisor $H$ such that 
$K_X+B+H\sim_\R 0/Z$ and $(X/Z,B+H)$ is dlt. Let $\epsilon>0$ be a 
sufficiently small number. Then, since $H$ is ample$/Z$ and  $(X/Z,B)$ is 
$\Q$-factorial dlt, there is $\Delta$ such that 
$$
K_X+\Delta\sim_\R K_X+B+H+\epsilon D\sim_\R \epsilon D/Z
$$
and $(X/Z,\Delta)$ is klt with $\Delta$ big$/Z$. Now by Theorem \ref{t-main-BCHM}, 
any LMMP$/Z$ with scaling on $K_X+\Delta$ terminates. This also gives an 
LMMP$/Z$ with scaling on $D$. 

The semi-ampleness claim follows from the base point free theorem and 
Lemma \ref{l-ray-stability}.\\
\end{proof}

\emph{$\Q$-factorial dlt blowups.} Let $(X/Z,B)$ be a lc pair. Assume that there is a projective birational 
morphism $f\colon Y\to X$ and a boundary $B_Y$ on $Y$ such that 
$K_Y+B_Y=f^*(K_X+B)$ and such that $(Y/Z,B_Y)$ is a $\Q$-factorial dlt pair. 
Moreover, assume that every exceptional$/X$ prime divisor on $Y$ 
has coefficient one in $B_Y$. We call   $(Y/Z,B_Y)$ 
a $\Q$-factorial dlt blowup of $(X/Z,B)$.

\begin{thm}\label{t-Q-factorial-blup}
Let $(X/Z,B)$ be a lc pair. Then, there is a $\Q$-factorial dlt blowup of 
$(X/Z,B)$.
\end{thm}
\begin{proof}
Let $g\colon W\to X$ be a log resolution and let $B_W=B^\sim+G$ where 
$B^\sim$ is birational transform of $B$ and $G$ is the reduced exceptional 
divisor of $g$. In particular, $K_W+B_W=g^*(K_X+B)+E$ where $E\ge 0$ is 
exceptional$/X$. The components of $E$ are those exceptional$/X$ prime divisors 
$D$ on $W$ such that $d(D,X,B)>-1$.

Let $C_W$ be an ample$/X$ divisor such that $K_W+B_W+C_W$ is 
dlt and nef$/X$. Run the LMMP$/X$ on $K_W+B_W$ with scaling of $C_W$ 
and let $\lambda_i$ be the corresponding numbers as in Definition \ref{d-LMMP-scaling}. 
If $\lambda:=\lim \lambda_i>0$, then the LMMP is also an LMMP on  $K_W+B_W+\frac{1}{2}\lambda C_W$
with scaling of $(1-\frac{1}{2}\lambda)C_W$. Now since $C_W$ is ample$/Z$, 
there is a klt $(W/Z,\Delta_W)$ such that $K_W+\Delta_W\sim_\R  K_W+B_W+\frac{1}{2}\lambda C_W$ 
and such that $\Delta_W$ is big$/Z$ (cf. Remark \ref{r-local}). 
The LMMP then terminates with a log minimal model $(Y/X,B_Y)$ of $(W/X,B_W)$
 by Theorem \ref{t-main-BCHM}. So, $E_Y$ is nef$/X$. By the negativity lemma 
$E_Y$ and we are done. So, we can assume that $\lambda=0$.  

Pick $i\gg 0$ and let $Y$ be the model corresponding to $i$, that is, 
on $Y$ we have: $K_Y+B_Y+\lambda_i C_Y$ is nef$/X$ and $\lambda_i$ is the smallest 
number with this property. Since $W\to X$ is birational, there is an  
$\R$-divisor $D_W\ge 0$ such that $C_W\sim_\R -D_W/X$. So, $C_Y\sim_\R -D_Y/X$
and $K_Y+B_Y-\lambda_i D_Y\sim_\R E_Y-\lambda_i D_Y$ is nef$/X$. Now 
by the negativity lemma, $ E_Y-\lambda_i D_Y\le 0$. Since $\lambda_i$ is 
sufficiently small, this is possible only if $ E_Y\le 0$ hence $E_Y=0$. 
Therefore, $K_Y+B_Y\sim_\R 0/X$ which implies that $K_Y+B_Y=f^*(K_X+B)$ 
where $f$ is morphism $Y\to X$. Note that every exceptional$/X$ prime divisor on $Y$ 
has coefficient one in $B_Y$. So,  $(Y/Z,B_Y)$ is 
a $\Q$-factorial dlt blowup of $(X/Z,B)$.\\
\end{proof}

\emph{Relations among log minimal models.} Let $(X/Z,B)$ be a klt pair and let 
$(Y_1/Z,B_{Y_1})$ and $(Y_2/Z,B_{Y_2})$ be two log minimal models of $(X/Z,B)$. 

\begin{lem}\label{l-mmodels-isom-cdmnsn1}
The induced map $Y_1\bir Y_2$ is an isomorphism in codimension one. 
Moreover, for any common resolution $f_i\colon W\to Y_i$ we have
$f_1^*(K_{Y_1}+B_{Y_1})=f_2^*(K_{Y_2}+B_{Y_2})$.
\end{lem}
\begin{proof}
Let $f_i\colon W\to Y_i$ be any common resolution, and let 
$$
E=f_1^*(K_{Y_1}+B_{Y_1})-f_2^*(K_{Y_2}+B_{Y_2})
$$ 
Then, $(f_1)_*E\ge 0$ for $i=1,2$: indeed let $D$ 
be a component of $E$ which is not exceptional$/Y_1$; if $D$ is also not 
exceptional$/Y_2$, then $D$ cannot be a component of $E$; if $D$ is 
exceptional$/Y_2$, then 
$$
d(D,Y_1,B_{Y_1})=a(D,X,B)<d(D,Y_1,B_{Y_1})
$$
hence $D$ should have non-negative coefficient in $E$.
Now $E$ is antinef$/Y_1$ so by the negativity lemma, $E\ge 0$. 
On the other hand, we can similarly prove that $-E\ge 0$. Therefore, 
$E=0$. In particular, this means that  $(Y_1/Z,B_{Y_1})$ and $(Y_2/Z,B_{Y_2})$ 
have the same discrepancy at any prime divisor on birational models 
of $Y_1,Y_2$. 

If $D$ is a prime divisor on $Y_1$ which is exceptional$/Y_2$, then 
a discrepancy calculation as above gives a contradiction. So, $Y_1\bir Y_2$ 
does not contract any divisors. Similarly, $Y_2\bir Y_1$ also does not 
contract any divisors.
\end{proof}

\begin{cor}
If $K_{Y_2}+B_{Y_2}$ is ample$/Z$, then $Y_1\bir Y_2$ is an isomorphism, i.e. 
the log minimal model is unique.
\end{cor}
\begin{proof}
In this case, by the Lemma, $K_{Y_1}+B_{Y_1}$ is semi-ample$/Z$. In fact,  $Y_1\bir Y_2$ 
is a morphism and $K_{Y_1}+B_{Y_1}$ is the pullcack of $K_{Y_2}+B_{Y_2}$. 
However, $Y_1\bir Y_2$ is a small morphism and $Y_1,Y_2$ are both $\Q$-factorial. 
This is possible only if $Y_1\bir Y_2$ is an isomorphism.\\
\end{proof}

\begin{thm}
If $B$ is big$/Z$, then $Y_1\bir Y_2$ can be decomposed into a sequence of 
flops with respect to $(Y_1/Z,B_{Y_1})$.
\end{thm}
\begin{proof}
Since $B$ is big$/Z$, $B_{Y_i}$ are also big$/Z$. So, $K_{Y_i}+B_{Y_i}$ is 
semi-ample$/Z$ and by Lemma \ref{l-mmodels-isom-cdmnsn1}, they are the 
pullback of the same ample$/Z$ $\R$-divisor, i.e. there are contractions 
$g_i\colon Y_i\to T/Z$ and an ample$/Z$ $\R$-divisor $H$ on $T$ such that 
$K_{Y_i}+B_{Y_i}\sim_\R g_i^*H$ for $i=1,2$. 

Let $A_{Y_2}$ be an ample $\R$-divisor on $Y_2$ such that $K_{Y_i}+B_{Y_i}+A_{Y_i}$ 
are both klt. In particular, $K_{Y_2}+B_{Y_2}+A_{Y_2}$ is ample$/Z$.
Now, run an LMMP$/T$ on $K_{Y_1}+B_{Y_1}+A_{Y_1}$ with scaling of an ample divisor.
The LMMP terminates with $Y_2$. In each step, $K_{Y_1}+B_{Y_1}$ is numerically 
trivial on the extremal ray contracted so each step is a flop.
\end{proof}

The previous theorem is true even if $B$ is not big$/Z$ as verified by Kawamata [\ref{Kawamata-flops}] 
(see also Birkar [\ref{Birkar-mmodel-II}, Corollary 3.3]).\\

\emph{Polytopes.} We recall the following from the section on finiteness of models.
Let $X\to Z$ be a projective morphism of normal quasi-projective varieties, $A\ge 0$ a $\Q$-divisor on $X$, and 
$V$ a rational (i.e. with a basis consisting of rational divisors) finite dimensional affine subspace of the space of $\R$-Weil divisors on $X$. Define 
$$
\mathcal{L}_{A}(V)=\{B=L+A \mid 0\le L\in V, ~ \mbox{and $(X/Z,B)$ is lc} \}
$$
By Shokurov [\ref{Shokurov-log-flips}, 1.3.2][\ref{Shokurov-log-models}], $\mathcal{L}_{A}(V)$ is a rational polytope (i.e. a polytope with rational vertices) inside the rational affine space $A+V$. 
We will be interested in rational polytopes inside $\mathcal{L}_{A}(V)$. 
Define 
$$
\mathcal{E}_{A}(V)=\{B\in\mathcal{L}_A(V) \mid \mbox{$K_X+B$ is pseudo-effective$/Z$} \}
$$
which is a convex closed set.

\begin{thm}
Assume that $A$ is ample and that $(X,0)$ is $\Q$-factorial klt. 
Then, $\mathcal{E}_{A}(V)$ is a rational polytope.
\end{thm}
\begin{proof}
It is enough to prove the statement locally, so fix $B\in \mathcal{E}_{A}(V)$. 
By a perturbation of coefficients, as in Remark \ref{r-local}, we can assume that 
$(X/Z,B)$ is klt, and that there is a rational polytope $\mathcal{C}\subseteq \mathcal{L}_{A}(V)$ 
containing an open neighbourhood of $B$ and such that $(X/Z,B')$ is klt for any 
$B'\in \mathcal{C}$. It is enough to prove that $\mathcal{C}\cap \mathcal{E}_{A}(V)$ 
is a rational polytope. By Theorem \ref{t-main-BCHM}, there is a a log minimal model
$(Y/Z,B_Y)$ for $(X/Z,B)$. Perhaps, after shrinking $\mathcal{C}$, we can 
replace  $(X/Z,B)$ by $(Y/Z,B_Y)$ hence assume that $K_X+B$ is nef$/Z$.
Thus, $K_X+B$ is semi-ample$/Z$. Let $X\to T/Z$ be the contraction associated 
to $K_X+B$.

Now by the finiteness Theorem \ref{t-finiteness} and Lemma \ref{l-ray-stability}, 
it is enough to prove that 
theorem over $T$ hence we can assume $Z=T$. Finally, we can use induction on 
dimension of $\mathcal{C}$ since it is enough to prove that result on the 
faces of $\mathcal{C}$.\\
\end{proof}

\emph{Zariski decomposition.} The Zariski decomposition problem for log divisors is closely 
related to the LMMP (see Birkar [\ref{Birkar-WZD}]). There are various definitions of 
Zariski decomposition in higher dimension (see the mentioned reference).  The 
next theorem is true using any of those definitions.

\begin{thm}
Let $(X/Z,B)$ be a klt pair such that $B$ is big$/Z$ and $K_X+B$ is pseudo-effective$/Z$.
Then, $K_X+B$ birationally has a Zariski decomposition. 
\end{thm}
\begin{proof}
By Theorem \ref{t-main-BCHM}, $(X/Z,B)$ has a log minimal model $(Y/Z,B_Y)$. 
Let $f\colon W\to X$ and $g\colon W\to Y$ be a common resolution. 
Then, we can write $f^*(K_X+B)=g^*(K_Y+B_Y)+E$ where $E\ge 0$ is exceptional$/Y$. 
This expression is a Zariski decomposition of $K_X+B$ in a birational sense.\\
\end{proof}

\clearpage

\vspace{1cm}

\end{document}